%% file: shell.tex
\documentclass[11pt]{amsart}

\usepackage{amsmath,amssymb,latexsym,amsthm,newlfont,
enumerate,enumitem}
\RequirePackage[dvipsnames,usenames]{color}

\usepackage{graphicx}
\usepackage[all]{xy}

\usepackage{hyperref}

\input{macros}

\begin{document}
\title[On the moduli part]{On the restriction of the moduli part to a reduced divisor}

\author{Enrica Floris}
\address{Universit\'e de Poitiers, Laboratoire de Math\'ematiques et Applications,\linebreak UMR~CNRS 7348, T\'el\'eport 2, Boulevard Marie et Pierre Curie, BP 30179, 86962 Futuroscope Chasseneuil Cedex, France}
\email{enrica.floris@univ-poitiers.fr}
\date{\today}

\thanks{This project started during the collaboration with V. Lazi\'c for \cite{FL19}.  
I am very grateful to Vlad for all the fruitful preliminary discussions around Theorem \ref{thm:step2} without which this paper would not have been possible.
I would like to thank J. Koll\'ar for pointing out some mistakes in a previous version of this work, and A. Petracci for many useful discussions on the topic.
I am very grateful to G. Bini, Y. Brunebarbe, P. Cascini,  G. Pacienza, A. Sarti, R. Svaldi and J. Witaszek for the useful discussions. I was supported by the ANR project "FIBALGA" ANR-18-CE40-0003 and the PEPS JCJC BFC 210024.}

\begin{abstract}
Let $f\colon(X,\Delta)\to Y$ be a fibration such that $K_X+\Delta$ is torsion along the fibres of $f$.
 Assume that  $Y$ has dimension 2, or that $Y$ has dimension 3 and the fibres have dimension at most 3.
 Then the restriction of the moduli part to its augmented base locus is semiample.
\end{abstract}

\maketitle
\setcounter{tocdepth}{1}
\tableofcontents

\section{Introduction}

In this paper we study fibrations $f\colon (X,\Delta)\to Y$ such that $K_X+\Delta$ is the pullback of a $\mathbb Q$-Cartier divisor $D$ on $Y$.
Those arise naturally, as
the abundance conjecture predicts that every log canonical pair is birational to either a Mori fibre space or a pair $(X,\Delta)$ with $K_X+\Delta$ semiample.
The induced fibration $f\colon X\to Y$ is such that $K_X+\Delta\sim_{\mathbb Q}f^*D$
for an ample $\mathbb Q$-Cartier divisor $D$ on $Y$.
The canonical bundle formula is a way of writing $D$ as the sum of three divisors: the canonical divisor of $Y$, a divisor $B_Y$ called discriminant defined in terms of the singularities of the fibration, and a nef (on a birational model of $Y$) divisor $M_Y$ called moduli part or moduli divisor, describing the variation in moduli of the fibres.
For example, by \cite[Theorem 3.3, 3.5]{Amb05a} if the moduli part is numerically zero and $(X,\Delta)$ is klt, then the fibration is essentially a product.

 The theory of the canonical bundle formula has its roots in the work by Kodaira and Ueno on elliptic surfaces. It has been developed and generalised in \cite{Kaw81, Amb04, Amb05a, FM00, Kol07}.

The idea of considering divisors of the form $K_Y+B+M$ where $K_Y$ is the canonical divisor, $(Y,B)$ satisfies certain regularity conditions and $M$ is nef on a higher model of $Y$ is central in the works by Birkar--Zhang, Birkar who consider generalised polarised pairs instead of pairs.

The most important conjecture on the canonical bundle formula has been formulated in \cite[Conjecture 7.13]{PS09}:
\begin{bSemCon}
Let $(X,\Delta)$ be a pair and let $f\colon (X,\Delta)\to Y$ be an lc-trivial fibration to an $n$-dimensional variety $Y$, where the divisor $\Delta$ is effective over the generic point of $Y$. If $Y$ is an Ambro model of $f$, then the moduli divisor $M_Y$ is semiample.
\end{bSemCon}

Several special cases of the conjecture are proved, mainly when the dimension of the fibre is at most two by the classical work of Kodaira and by \cite{PS09, Fuj03, Fil18} 
and if the moduli part is numerically zero by \cite{Amb05a, Flo14}.
For the klt case, if the moduli part is torsion, then by \cite[Theorem 3.3]{Amb05a} the variation of $f$ is zero.

In this paper we consider a connected divisor $\mathcal T=\cup T$ and assume the B-Semiampleness Conjecture in lower dimension.
In \cite{FL19} we proved that the divisor $M_Y\vert_T$ is semiample for every $T$.
In this work we study the gluing of the global sections of $mM_Y\vert_T$ to obtain global sections of $mM_Y\vert_{\mathcal T}$.

The main result of this paper is the following:
\begin{thmA}\label{cor:surfaces}
Let $(X,\Delta)$ be a pair and let $f\colon (X,\Delta)\to Y$ be a klt-trivial fibration to a surface $Y$, where the divisor $\Delta$ is effective over the generic point of $Y$. Assume that $Y$ is an Ambro model for $f$ and that $M_Y$ is big.

Then there is a birational base change $Y'\to Y$ such that
the restriction of $M_{Y'}$ to the augmented base locus is torsion.
\end{thmA}

The semiampleness of the moduli part turns out to be deeply related to the variation of the fibres of $f$.
The variation, introduced by Viehweg \cite{Vieh83} is roughly speaking the dimension of the moduli space of fibres of $f$ in the sense of birational geometry (see Definition \ref{def:var} for a precise definition).
The Kodaira dimension of the moduli part is at most the variation of $f$, and conjecturally they coincide.
On the other hand, for a fibration of maximal variation there should be only a finite number of fibres birational to a given one:
\begin{con}\label{con:maxvar}
Let $X$ be a $\mathbb{Q}$-factorial variety.
 Let $f\colon(X,\Delta)\to Y$ be a klt-trivial fibration of maximal variation.
 Then there is an open set $U\subseteq Y$ such that for every $y\in U$ the set
 $$\{z\in U\vert\; (f^{-1}y,\Delta^h\vert_{f^{-1}y})\; \text{is crepant birational to}\; (f^{-1}z,\Delta^h\vert_{f^{-1}z}) \}$$
 is finite, where $\Delta^h$ denotes the horizontal part of $\Delta$.
\end{con}

Conjecture \ref{con:maxvar} is true for fibrations of relative dimension at most 2.
Using this fact we are able to prove

\begin{thmA}\label{cor:f3}
Let $(X,\Delta)$ be a pair and let $f\colon (X,\Delta)\to Y$ be a klt-trivial fibration to a variety $Y$ 
of dimension 3 and $\dim X\leq \dim Y+3$, where the divisor $\Delta$ is effective over the generic point of $Y$. Assume that $Y$ is an Ambro model for $f$ and that $M_Y$ is big.

Then there is a birational base change $Y'\to Y$ such that
the restriction of $M_{Y'}$ to the augmented base locus is semiample.
\end{thmA}

For the proof of Theorem \ref{cor:surfaces} and Theorem \ref{cor:f3}, we embrace the approach developed in \cite{Kol13} and successfully applied in \cite{HX13} to the study of the semiampleness of the log canonical divisor of a 
slc pair (roughly speaking a simple normal crossings divisor in a smooth variety).

By \cite{FL19} we are in the following setting: we have a line bundle $\mathcal L$ on a reduced, non irreducible variety $\mathcal T$ which is semiample on every irreducible component of $\mathcal T$.
We want to prove that $\mathcal L$ is semiample on $\mathcal T$.
The approach consists in translating the  semiampleness of a line bundle into the finiteness of a certain equivalence relation.
For the sake of simplicity, assume that $\mathcal T=T_1\cup T_2$.
Let $\phi_i\colon T_i\to V_i$ be the fibration induced by $\mathcal L$ for $i=1,2$.
We say that $x_1\in V_1$ is equivalent to $x_2\in V_2$ if $\phi_1^{-1}(x_1)\cap\phi_2^{-1}(x_2)\neq\emptyset$ and we take the closure of this equivalence relation.
This is the natural relation to consider. Indeed, if $\mathcal L\vert_{T_1\cup T_2}$ is semiample and $\phi\colon T_1\cup T_2\to V$ is the induced fibration, 
then $\phi_1^{-1}(x_1)$ and $\phi_2^{-1}(x_2)$ are sent to the same point by $\phi$.

By considering the union of the fibres of $\phi_1$ and $\phi_2$ which intersect, we construct subsets of $T_1\cup T_2$ called pseudofibres.

The reason why we cannot fully apply Koll\'ar's gluing theory is that many of the required regularity hypotheses are not satisfied in our setup.

\medskip

We now describe the structure of the paper as well as the techniques used in every section.
Section \ref{Prel} contains some preliminary results as well as some refinements of results on the canonical bundle formula. 
Section \ref{def:snc} is a semiampleness criterion for a line bundle on a simple normal crossings surface. In section \ref{Prof} we recall the basic notions on equivalence 
relations and prove some technical lemmas necessary for the study of the equivalence relation $R_{\mathcal L}$, which is done in section \ref{Glu}.
In section \ref{Graph} we gather some results from \cite{Stallings} and we apply them in section \ref{Trivial} where we develop a criterion for the triviality a line bundle on a simple normal crossing variety. Section \ref{Restr} uses techniques from the minimal model program and is a study of the restriction of the moduli part to higher codimensional log canonical centres.

In section \ref{thm1} we prove that, assuming the B-Semiampleness Conjecture in dimension $n-1$ and Conjecture \ref{con:maxvar} in dimension  $d-1$, 
the equivalence relation is finite for $\mathcal L=\OO(mM_Y)$ for $Y$ of dimension $n$ and $X$ of dimension $d+n$. 
In section \ref{thm2} we prove that the restriction of  $\mathcal L$ to a simple normal crossings pseudofibre is torsion.

The last section contains the proofs of Theorems \ref{cor:surfaces} and \ref{cor:f3}.

\input{preliminariesS}

\input{supS}

\input{profiniteS}

\input{cerbero1S}

\input{graphS}

\input{trivialSS}

\input{dlt4lctS}

\input{thm1}

\input{thm2}

\section{Proof of the main results}

We are now ready to prove our main results.

\begin{proof}[Proof of Theorem \ref{cor:surfaces}]
Let $\mathcal T$ be a connected component of $\sB_+(M_Y)$.
For every component $T\subseteq\mathcal T$, the restriction $M_Y\vert_T$ is a torsion divisor.
Therefore $\phi_T$ contracts $T$ to a point $p_T$.
If $\mathcal L=\mathcal O(mM_Y)\vert_{\mathcal T}$, then the relation $\mathcal R_{\mathcal L}$
is finite because it is a subset of $\sqcup\{p_T\}\times\sqcup\{p_T\}$.
By Theorem \ref{thm:step2} the line bundle  $\mathcal L$ is torsion.
\end{proof}

\begin{proof}[Proof of Theorem \ref{cor:f3}]
Set $\mathcal L=\mathcal O(mM_Y)\vert_{\mathcal T}$.

Conjecture \ref{con:maxvar} is true for fibrations of relative dimension at most 2.
Indeed, let $(F_1,\Delta_1)$, $(F_2,\Delta_2)$ be crepant birational fibres.
If $\dim F_i=1$, then $(F_1,\Delta_1)$ and $(F_2,\Delta_2)$ are isomorphic and the Conjecture follows from Proposition \ref{pro:finvar}.
If $\dim F_i=2$, then let $(p_1,p_2)\colon G\to F_1\times F_2$ be a resolution of the indeterminacy such that $K_G=p_i^*(K_{F_i}+\Delta_i)+\sum a_j E_j$, where the $a_j$ do not depend on $i$ by the definition of crepant birational map.
Set $\Delta_G=\sum_{a_j<0}-a_j E_j$.
Thus $\sB_-(K_G+\Delta_G)=\cup_{a_j>0}E_j$ and $(F_1,\Delta_1)$ and $(F_2,\Delta_2)$ are minimal models of 
$(G,\Delta_G)$.
Thus they are connected by flops. As $\dim F_i=2$, they are isomorphic.
The Conjecture then follows from Proposition \ref{pro:finvar}.

We can assume that the augmented base locus is a simple normal crossings divisor $\mathcal T$.

By \cite[Corollary D]{FL19} for every irreducible component $T$, the restriction $\mathcal L\vert_T$ is semiample. We denote by $\phi_T\colon T\to V$ the induced fibration.
By Theorem \ref{thm:step1}, $\mathcal R_{\mathcal L}$ is a finite equivalence relation.
The relation is therefore stratifiable by \cite[Remark 9.20]{Kol13}.
We notice that, as $\dim Y=3$,  the normal variety $\bigsqcup V$ is such that $\dim V\in\{0,1\}$. The strata of the stratification have dimension $0$ or $1$.
Therefore the stratification satisfies the regularity hypotheses (HN) and (HSN) \cite[Definition 9.8]{Kol13}.
By \cite[Theorem 9.21]{Kol13} the quotient $\pi\colon \bigsqcup V\to Q$ for $\mathcal R_{\mathcal L}$ exists and is 
reduced because $\pi$ is surjective, separated by \cite[Definition 47, Corollary 48]{Kol12}. 
Moreover $Q$ is seminormal and
 there is a fibration $\phi\colon \mathcal T\to Q$ whose fibres are the pseudofibres.

If $\mathcal L\vert_T$ is torsion for every $T$, then $\dim V=0$ for every $V$ and $Q$ is a point, hence projective. Then the claim follows from Theorem \ref{thm:step2}.

Otherwise, for every component $Q_0$ of $Q$ of dimension 1, there is $V_0\subseteq\bigsqcup V$ together with a finite surjective morphism $\pi\colon V_0\to Q_0$. 
By \cite[Proposition II.6.8]{Har77} $Q$ is complete, by \cite[Proposition II.6.7]{Har77}  $Q$ is projective.

Let $\varepsilon\colon Y'\to Y$ be a birational morphism such that $\varepsilon \Exc(\varepsilon)\subseteq \sB_+(M_Y)$ and every (set-theoretic) fibre of the restriction of $\phi\circ\varepsilon$ to $\varepsilon^{-1}\sB_+(M_Y)$
is simple normal crossing in the sense of Section \ref{def:snc}.

We have $\varepsilon^{-1}\sB_+(M_Y)=\sB_+(\varepsilon^*M_Y)$. The latter is the augmented base locus of the moduli part $M_{Y'}$ of the base changed fibration, because $Y$ is an Ambro model.

We replace thus $Y$ with $Y'$ and $\phi$ with  $\phi\circ\varepsilon$.

By Theorem \ref{thm:step2}, for every fibre $F$ of $\phi$, the restriction of $\mathcal L$ to the reduced part of $F$ is torsion.

After replacing $\mathcal L$ with $\mathcal L^{\otimes m}$ for $m$ divisible enough, we can assume that for every fibre $F$ of $\phi$, the restriction of $\mathcal L$ to the reduced part of $F$ is trivial.

By Theorem \ref{thm:sup}, the line bundle $\mathcal L$ is semiample.

\end{proof}

\bibliographystyle{amsalpha}

\bibliography{biblio}

\end{document}

%% file: macros.tex
\theoremstyle{plain}

\newtheorem{thm}{Theorem}[section]

\newtheorem{pro}[thm]{Proposition}

\newtheorem{lem}[thm]{Lemma}

\newtheorem{cla}[thm]{Claim}
\newtheorem{cor}[thm]{Corollary}
\newtheorem{con}[thm]{Conjecture}
\newtheorem*{bSemCon}{B-Semiampleness Conjecture}
\newtheorem{ass}[thm]{Assumption}

\newtheorem{thmA}{Theorem}

\theoremstyle{definition}

\newtheorem{dfn}[thm]{Definition}

\newtheorem{nt}[thm]{Notation}

\newtheorem{rem}[thm]{Remark}

\theoremstyle{remark}
\newtheorem{const}[thm]{Construction}

\newcommand{\C}{\mathbb{C}}
\newcommand{\R}{\mathbb{R}}
\newcommand{\Q}{\mathbb{Q}}

\newcommand{\OO}{\mathcal{O}}

\newcommand{\red}{\mathrm{red}}

\DeclareMathOperator{\rk}{rk}

\DeclareMathOperator{\Sing}{Sing}
\DeclareMathOperator{\Exc}{Exc}
\DeclareMathOperator{\mult}{mult}

\DeclareMathOperator{\Supp}{Supp}

\DeclareMathOperator{\Hilb}{Hilb}
\DeclareMathOperator{\Emb}{Emb}
\DeclareMathOperator{\Univ}{Univ}

\DeclareMathOperator{\GL}{GL}

\DeclareMathOperator{\Bir}{Bir}

\DeclareMathOperator{\ddiv}{div}
\DeclareMathOperator{\sB}{\mathbf{B}}

%% file: preliminariesS.tex
\section{Preliminary results}\label{Prel}
We work over the complex numbers. For the notions on the minimal model program and singularities of pairs we refer to \cite{KM92}. We will use without defining them the notions of log canonical, klt and dlt singularities, as well as of centre of a log canonical singularity. We refer to \cite{Kol97} and \cite{KM92} for a presentation of these concepts and to \cite[Definition 2.5]{FL19} for a summary of all the required notions in our setup.

We recall that a \textit{pair} $(X,\Delta)$ is the data of a normal projective variety $X$ and a $\mathbb Q$-Weil divisor $\Delta$ such that $K_X+\Delta$ is $\mathbb Q$-Cartier. In this paper we do not require $\Delta$ to be an effective divisor.

We say that a
closed subvariety $S$ of $X$ is a \textit{minimal log canonical centre of $(X,\Delta)$ over $Z$}
if $S$ is a minimal log canonical centre of $(X,\Delta)$ (with respect to inclusion)
which dominates $Z$.

\subsection{Semistable morphisms}
In this paragraph we recall the definition of semistable morphisms and the statement of the semistable reduction theorem, proved in \cite{ALT18}, which will be crucial in the proof of our main results, Theorem \ref{thm:step1} and \ref{thm:step2}.
We refer to \cite[section II.1.1]{Ogus} for the definition of log scheme and morphism of log schemes, and to \cite[Definition 15.52.1, Proposition 15.52.3]{Stacks} for the definition of quasi excellent rings and the first properties. 

\begin{dfn}[4.2.1 \cite{ALT18}]
A morphism of log schemes $f\colon X\to B$, $f^{\sharp}\colon f^{-1}\mathcal O_Y\to\mathcal O_X$ is semistable if the following conditions hold:
\begin{enumerate}\item 
$X$ and $B$ are regular and the log structures are given by normal crossings divisors
$Z\subseteq X$ and $W\subseteq B$.
\item \'Etale-locally at any $x\in X$ with $b=f(x)$ there exist regular parameters
$t_1,\ldots ,t_n, t'_1,\ldots,t'_{n'}\in\mathcal O_{X,x}$ and $\pi_1,\ldots ,\pi_l, \pi'_1,\ldots,\pi'_{l'}\in\mathcal O_{B,b}$ such that $Z=V(t_1\cdot\ldots\cdot t_n)$ at $x$, $W=V(\pi_1\cdot\ldots \cdot\pi_l)$ at $b$, 
$f^{\sharp}(\pi_i) =t_{n_i+1}\ldots t_{n_{i+1}}$ for 
$0 =n_1< n_2<\ldots < n_{l+1}\leq n$.
\item $f$ is log smooth.
\end{enumerate} In characteristic zero, the third condition can be replaced by the condition that 
$f^{\sharp}(\pi'_j) =t'_j$ for $1\leq j\leq l'$.
\end{dfn}

The following semistable reduction theorem is proved in \cite{ALT18} and uses a finer toroidalization proved in \cite{ATW20}.

\begin{thm}[Theorem 4.7 \cite{ALT18}]\label{thm:alt}
Assume that $X\to B$ is a dominant morphism of finite type between quasi excellent integral schemes of characteristic zero and $Z\subseteq X$ is a closed subset.  Then there exists a stack-theoretic modification $b\colon B'\to B$, a projective modification $a\colon X'\to (X\times_{B}B')^{pr}$, and divisors $W'\subseteq B'$, $Z'\subseteq X'$such that:
\begin{enumerate}
\item $a^{-1}Z\cup f^{\prime -1}W'\subseteq Z'$ and the morphism $f'\colon (X', Z')\to(B', W')$ is semistable. In particular, $X', B'$ are regular and $Z', W'$ are snc.
\item  If a regular open $B_0\subseteq B$ is such that $X_0=X\times_B B_0\to B_0$ is smooth and $Z_0=Z\times_B B_0\to B_0$ is a relative divisor over $B_0$ 
with normal crossings (in other words, $f\colon (X_0, Z_0)\to(B_0, W_0)$ is semistable), then $a$ and $b$ are isomorphisms over $X_0$ and $B_0$, respectively.
\end{enumerate}

\end{thm}

\begin{rem}\label{rem:restred}
Let $f\colon (X, Z)\to(B, W)$ be a semistable map, and
let $S\subseteq X$ be a stratum of $Z$. Let $C=f(S)$. It follows from the definition that $f\vert_S\colon S\to C$ is semistable.
Moreover, if $f\vert_S\colon S\overset{h}{\longrightarrow} C'\overset{\tau}{\longrightarrow} C$
is the Stein factorisation, then $h$ is semistable.
\end{rem}

\subsection{Groups of crepant birational automorphisms}
In this paragraph we state two results on the group of crepant birational selfmaps of a pair.
The first one is the finiteness of pluricanonical representations \cite[Theorem 4.5]{Gon13}
and the second one is a generalisation to pairs of the finiteness of the group of selfmaps of a manifold of general type.

\begin{dfn}
Let $f_1\colon(X_1, \Delta_1)\to Y$ and $f_2\colon(X_2, \Delta_2)\to Y$ be two fibrations of pairs to the same base $Y$. A birational map $\theta\colon X_1\dasharrow X_2$ is \emph{crepant birational over $Y$} if $a(E,X_1,\Delta_1)=a(E,X_2,\Delta_2)$ for every geometric valuation $E$ over $X_1$ and $X_2$ and we have the commutative diagram
$$
\xymatrix{
X_1\ar[rd]_{f_1} \ar@{-->}[rr]^{\theta}&&X_2\ar[ld]^{f_2}\\
&Y.&
}
$$
The map $\theta$ is \emph{crepant birational} if $Y$ is a point.

The set of all crepant birational maps of a pair $(X,\Delta)$ to itself is a group, denoted by $\Bir^\mathrm{c}(X,\Delta)$. For a positive integer $m$ such that $m(K_X+\Delta)$ is Cartier, every $\sigma\in \Bir^\mathrm{c}(X,\Delta)$ defines an automorphism of $H^0(X,m(K_X+\Delta))$, and hence the \emph{pluricanonical representation}
$$\rho_m\colon \Bir^\mathrm{c}(X,\Delta)\to \GL\big(H^0(X,m(K_X+\Delta))\big).$$
\end{dfn}

\begin{rem}\label{rem:cbir}
If the condition $p^*(K_{X_1}+\Delta_1)=q^*(K_{X_2}+\Delta_2)$ is true for one resolution of the indeterminacy, then it is true for every resolution of indeterminacy. Indeed, let  $(p',q')\colon W'\to X_1\times X_2$ be another resolution of the indeterminacy. Let $(\nu,\mu)\colon\widehat W\to W\times W'$ be a dominating birational model.
Then $\nu^* p^*(K_{X_1}+\Delta_1)=\nu^* q^*(K_{X_2}+\Delta_2)$.
By commutativity, $\nu^* p^*(K_{X_1}+\Delta_1)=\mu^* p^{\prime *}(K_{X_1}+\Delta_1)$ and 
$\nu^* q^*(K_{X_1}+\Delta_1)=\mu^* q^{\prime *}(K_{X_1}+\Delta_1)$. We conclude by pushing forward with $\nu$.
\end{rem}

\begin{thm}\label{thm:cbir}
Let $(X,\Delta)$ be a klt pair such that $K_X+\Delta\sim_\Q0$. Then for every $m$, the image of the pluricanonical representation $\rho_m$ is finite. In particular, there is a positive integer $\ell$ such that the image of $\rho_\ell$ is trivial.
\end{thm}

\begin{proof}
The first statement is \cite[Theorem 4.5]{Gon13}, and then the second statement is straightforward.
\end{proof}

\subsection{Canonical bundle formula}

In this subsection we define lc-trivial fibration and recall several fundamental results. We refer the reader to \cite{FL20} for a survey of the general results on the canonical bundle formula.

\begin{dfn}
Let $(X,\Delta)$ be a pair and let $\pi\colon X'\rightarrow X$ be a log resolution of the pair. 
A morphism $f \colon (X,\Delta) \rightarrow Y$ to a normal projective variety $Y$ is a \emph{klt-trivial}, respectively \emph{lc-trivial}, fibration if $f$ is a surjective morphism with connected fibres, $(X,\Delta)$ has klt, respectively log canonical, singularities over the generic point of $Y$, there exists a $\Q$-Cartier $\Q$-divisor $D$ on $Y$ such that
$$K_X+\Delta\sim_\Q f^*D,$$
and if $f'=f\circ\pi$, then
$$\rk f'_*\OO_X(\lceil K_{X'}-\pi^*(K_X+\Delta)\rceil) = 1,$$
respectively
$$\textstyle \rk f'_*\OO_X\big(\lceil K_{X'}-\pi^*(K_X+\Delta)+\sum_{a(E,X,\Delta)=-1} E\rceil\big) = 1.$$
\end{dfn}

\begin{rem}
This last condition in the previous definition is verified, for instance, if $\Delta$ is effective on the generic fibre, which is mostly the case in this paper. 
\end{rem}

%
%

\begin{dfn}\label{dfn:cbf}
Let $f\colon (X,\Delta)\to Y$ be an lc-trivial fibration, and let $P\subseteq Y$ be a prime divisor with the generic point $\eta_P$. The \emph{log canonical threshold} of $f^*P$ with respect to $(X,\Delta)$ is
$$\gamma_P=\sup\{t\in\R\mid (X,\Delta+tf^*P) \textrm{ is log canonical over } \eta_P\}.$$
The \emph{discriminant} of $f$ is
\begin{equation}\label{discriminant}
B_f=\sum_P(1-\gamma_P)P.
\end{equation}
This is a Weil $\Q$-divisor on $Y$, and it is effective if $\Delta$ is effective. Fix $\varphi\in\C(X)$ and the smallest positive integer $r$ such that $K_X + \Delta +\frac{1}{r}\ddiv\varphi = f^*D$. Then there exists a unique Weil $\Q$-divisor $M_f$, the \emph{moduli part} of $f$, such that 
\begin{equation}\label{cbf}
K_X + \Delta +\frac{1}{r}\ddiv\varphi = f^*(K_Y+B_f+M_f).
\end{equation}
The formula \eqref{cbf} is the \emph{canonical bundle formula} associated to $f$.
\end{dfn}

\begin{rem}
As in \cite{FL19}, we adopt here the notation $B_f, M_f$ for the discriminant and moduli part of $f$ instead of the usual one $B_Y, M_Y$.
We will occasionally write $B_Y, M_Y$ when the fibration is clear from the contest.
\end{rem}

%

\begin{rem}\label{bbir}
If $f_1\colon (X_1,\Delta_1)\to Y$ and $f_2\colon (X_2,\Delta_2)\to Y$ are two lc-trivial fibrations over the same base  which are crepant birational over $Y$, then $f_1$ and $f_2$ have the same discriminant and moduli part.
\end{rem}

The canonical bundle formula satisfies several desirable properties. The first is the \emph{base change property}, \cite[Theorem 0.2]{Amb04} and \cite[Theorem 2]{Kaw98}.

\begin{thm}\label{nefness}
Let $f \colon(X,\Delta)\to Y$ be a klt-trivial fibration. Then there exists a proper birational morphism $Y'\to Y$ such that for every proper birational morphism $\pi \colon Y''\to Y'$ we have:
\begin{enumerate}
\item[(i)] $K_{Y'}+B_{Y'}$ is a $\Q$-Cartier divisor and $K_{Y''}+B_{Y''}=\pi^*(K_{Y'}+B_{Y'})$,
\item[(ii)] $M_{Y'}$ is a nef $\Q$-Cartier divisor and $M_{Y''}=\pi^*M_{Y'}$.
\end{enumerate}
\end{thm}

In the context of the previous theorem, we say that $M_Y$ \emph{descends} to $Y'$, and we call $Y'$ an \emph{Ambro model} for $f$. One of the reasons why base change property is important is the following \emph{inversion of adjunction} \cite[Theorem 3.1]{Amb04}.

Moreover, by \cite[Proposition 8.4.9, Definition 8.3.6, Theorem 8.5.1]{Kol07} if $f\colon (X,\Delta)\to Y$ is an lc-trivial fibration such that the non-smooth locus $\Sigma$ of the fibration is a simple normal crossings divisor and $f^{-1}\Sigma+\Delta$ is simple normal crossings, then $Y$ is an Ambro model.

\begin{rem}\label{rem:modpseff}
Theorem \ref{nefness} implies in particular that the moduli part is always pseudoeffective, even when it is not nef, as it is the push-forward of a nef divisor by a birational model.
\end{rem}

We prove now that if the moduli part descends on $Y$, then it descends on $Y'$ with $Y'\to Y$ generically finite.

\begin{lem}\label{lem:gfiniteAm}
Let $f\colon (X,\Delta)\to Y$ be an lc-trivial fibration and $\tau\colon Y'\to Y$ be a genrically finite map. If $Y$ is an Ambro model, then $Y'$ is an Ambro model for the fibration obtained by base change.
\end{lem}
\begin{proof}
By taking the Stein factorisation, it is enough to consider $\tau$ finite.
Let $\tilde\nu\colon \widetilde Y\to Y'$ be a birational map.
By \cite[Lemma 2.4]{FL19} there is a diagram
$$
\xymatrix{
Y'\ar[d]_{\tau}&\widetilde Y \ar[l]_{\tilde\nu}&W'\ar@{-->}[l]_{\nu'}\ar[d]^{\sigma}\\
Y&&W\ar[ll]^{\nu}
}
$$
such that $\nu$ and $\nu'$ are birational and $(\nu')^{-1}$ is an isomorphism along the generic point of every $\tilde\nu$-exceptional divisor.

Let $(p,q)\colon \widehat W\to W'\times \widetilde Y$ be a resolution of the indeterminacies.
Then we have
$$M_{\widehat W}=p^*\sigma^*\nu^*M_Y=q^*\tilde\nu^*\tau^*M_Y=q^*\tilde\nu^*M_{Y'}$$
which implies $M_{\widetilde Y}=q_*M_{\widehat W}=\tilde\nu^*M_{Y'}$.

\end{proof}


\begin{thm}\label{thm:invAdjunction}
Let $f\colon(X,\Delta)\rightarrow Y$ be an lc-trivial fibration, and assume that $Y$ is an Ambro model for $f$. Then $(Y,B_Y)$ has klt, respectively log canonical, singularities in a neighbourhood of a point $y\in Y$ if and only if $(X,\Delta)$ has klt, rerspectively log canonical, singularities in a neighbourhood of $f^{-1}(y)$.
\end{thm}


The following is \cite[Theorem 3.3]{Amb05a}. It will be a key result in the proof of both Theorem \ref{thm:step1} and  \ref{thm:step2}.

\begin{thm}\label{ambro1}
Let $f\colon(X,\Delta)\to Y$ be a klt-trivial fibration between normal projective varieties such that $\Delta$ is effective over the generic point of $Y$. Then there exists a diagram
$$
\xymatrix{
(X,\Delta)\ar[d]_f && (X^+,\Delta^+)\ar[d]^{f^+}\\
Y & \widetilde Y \ar[l]^{\vartheta}\ar[r]_{\chi}&Y^+
}
$$
such that:
\begin{enumerate}
\item[(i)] $f^+\colon(X^+,\Delta^+)\rightarrow Y^+$ is a klt-trivial fibration,
\item[(ii)] $\vartheta$ is generically finite and surjective, and $\chi$ is surjective,
\item[(iii)] there exists a non-empty open set $U\subseteq \widetilde Y$ and an isomorphism
$$
\xymatrix{
(X,\Delta)\times_{Y} U\ar[rd]\ar[rr]^{\cong}&&(X^+,\Delta^+)\times_{Y^+} U\ar[ld]\\
&U,&
}
$$
\item[(iv)] the moduli part $M_{f^+}$ is big and, after possibly a birational base change, we have $\vartheta^*M_f=\chi^*M_{f^+}$.
\end{enumerate}
\end{thm}

The following remark will be useful at the end of this section.
\begin{lem}\label{lem:excvert}
Notation as in Theorem \ref{ambro1}. Assume that $Y$ is an Ambro model and $M_f$ is semiample, let $\phi\colon Y\to V$ be the fibration induced by $M_f$.
Then there is $\vartheta$ such that $\Exc(\vartheta)$ is vertical with respect to $\phi\circ\vartheta$.
Moreover there is a generically finite map $\lambda\colon Y^+\to V$.
\end{lem}
\begin{proof}
Since $M_f$ is semiample, $M_{f^+}$ is semiample as well.
Let $\phi^+\colon Y^+\to V^+$ be the fibration defined by $M_{f^+}$. We notice that as $M_{f^+}$
is big, the fibration $\phi^+$ is birational.
Since $\phi^+\circ\chi$ is a fibration, there is a finite map $V^+\to V$.
We set $\lambda\colon Y^+\to V$ the induced generically finite map.

By the proof of \cite[Theorem 2.2]{Amb05a}, we have $\vartheta=\varepsilon \circ\sigma\circ p$
where 
\begin{itemize}
\item $p\colon Y'\to Y$ is birational, such that $Y'$ is smooth and the period map extends to a fibration $q\colon Y'\to Y_0$ and $p$ can be taken as a composition of blow-ups along smooth centres;
\item $\sigma\colon Y''\to Y'$ is finite and such that, if $\sigma_0\circ\alpha$ is the Stein factorisation of $q\circ \sigma$ then $\alpha \colon Y''\to Y^+$ admits a section;
\item $\varepsilon$ is a desingularisation of $Y''$ and $\chi=\alpha\circ\varepsilon$.
\end{itemize}

We prove first that $p\Exc(p)$ is $\phi$-vertical. 
Indeed, let $C\subseteq\Exc(p)$ be a curve not contracted by $p$ but contracted by $q$.
Let $\widetilde C\subseteq Y''$ be such that $\sigma(\widetilde C)=C$.
Since $q\circ\sigma=\sigma_0\circ \alpha$, the image $\alpha(\widetilde C)$ is a point.
Then $\sigma\circ p\circ\phi(\widetilde C)=\lambda\circ\alpha(\widetilde C)$ is a point. Therefore $C$
is contracted by $\phi$. This implies that the indeterminacy locus of $q\circ p^{-1}$ is $\phi$-vertical.
Therefore the indeterminacy locus of $q\circ p^{-1}$ is $\phi$-vertical and we can find $p,q$ such that the exceptional locus of $p,q$ is $\phi\circ p$-vertical.

The morphism $\sigma$ is generically \'etale, therefore the singularities of $Y''$ are vertical with respect to $\alpha$. Therefore we can chose $\varepsilon$ which is an isomorphism over the generic point of $V$.
\end{proof}

The following is \cite[Proposition 4.4]{Amb05a}, and it allows to extend the isomorphism from Theorem \ref{ambro1}(iii) to a suitable bigger open subset.

\begin{pro}\label{ambro2}
Let $f\colon(X,\Delta)\rightarrow Y$ be a klt-trivial fibration of normal projective varieties such that there exists an isomorphism
$$\Phi\colon (X,\Delta)\times_Y U\to (F, \Delta_F)\times U$$
over a non-empty open subset $U\subseteq Y$. Then $\Phi$ extends to an isomorphism over 
$$Y^0=Y\setminus \big(\Supp B_Y\cup \Sing (Y)\cup f(\Supp \Delta_v^{<0})\big),$$
where $\Delta_v^{<0}$ consists of the vertical components of $\Delta$ with negative coefficients in $\Delta$.
\end{pro}

The following two lemmas were written in collaboration with V. Lazi\'c.
\begin{lem}\label{redgen}
Let $S, T,\widetilde T$ be quasi-projective varieties, assume that $T$ is smooth. Let $h\colon S\to T$ be a projective fibration and let $\vartheta\colon\widetilde T\to T$ be a finite map. Let 
$$
\xymatrix{
S\ar[d]_h&F\times \widetilde T\ar[d]^{\tilde h}\ar[l]_{\tau}\\
 T&\widetilde T\ar[l]^{\vartheta}
}
$$
be a base change where $\tilde h$ is the second projection. Let $G$ be a reduced fibre of $h$. Let $y\in\widetilde T$ be such that $\tau(y)=x$. Then $\tau\colon F\times\{y\}\rightarrow G$ is an isomorphism.
\end{lem}

\begin{proof}
After cutting the base with $\dim T-1$ hyperplane sections through $x$, we can assume that $\dim T=1$.

The morphisms $\vartheta$ and $\tau$ have the same degree, set $d=\deg\tau=\deg \vartheta$. Let $x \in \tilde T$ be such that $G=h^*x$. 
Write $\tau^* G=\sum a_i F \times\{ y_i\}$ and $\vartheta^* x=\sum e_i y_i$. 
Thus $\tau^* h^* x= \tau^* G=\sum a_i F  \times\{ y_i\} =\tilde h^* \vartheta^* x=\tilde h^* \sum e_i y_i=\sum e_i F \times\{ y_i\}$. 
It follows, perhaps after renumbering the $y_i$, that $a_i=e_i$ for all $i$. Moreover, $ d=\sum a_i \deg(F  \times\{ p_i\}\to G)=\sum e_i \deg(F  \times\{ p_i\}\to G) \geq \sum e_i=d$. Thus $\deg(F  \times\{ p_i\}\to G)=1$ for all $i$. 
\end{proof}

\begin{lem}\label{lem:reducedfibres}
Let $g\colon (Z,\Delta_Z)\rightarrow T$ be a klt-trivial fibration, where $\Delta_Z\geq0$ and the discriminant $B_g$ is a reduced divisor. Assume we have a base change diagram
$$
\xymatrix{
(Z,\Delta_{Z})\ar[d]_{g} & (\overline{Z},\Delta_{\overline{Z}})\ar[d]_{\overline{g}}\ar[l]_{\overline \alpha}\\
 T & \overline{T},\ar[l]_{\alpha}
}
$$
where $\alpha$ and $\overline \alpha$ are finite morphisms and $\overline g$ is weakly semistable in codimension $1$. Let $B_{\overline g}$ be the discriminant of $\overline g$ and assume that $\Delta_{Z,v}=(g^*B_g)_\red$. 
Then there exists an open subset $U\subseteq\overline T$ with complement of codimension at leat 2 in $\overline T$ such that:
\begin{enumerate}
\item[(i)] $(\overline g^*B_{\overline g})|_{\overline g^{-1}(U)}=\big((\overline\alpha^*\Delta_{Z,v})_{\red}-\overline g^*R'_T \big)|_{\overline g^{-1}(U)}$, where $R'_T$ is an effective divisor supported on the ramification divisor of $\alpha$ and having no common components with $\alpha^*B_g$;
\item[(ii)] $\Delta_{\overline Z}\vert_{\overline g^{-1}(U)}=\big(\overline \alpha^*\Delta_{Z,h}+(\overline \alpha^*\Delta_{Z,v})_\red -\overline g^*R'_T \big)\vert_{\overline g^{-1}(U)}$.
\end{enumerate}
In particular, $(\Delta_{\overline Z}-\overline g^*B_{\overline T})\vert_{\overline g^{-1}U}\geq0$
and if $T$ is a curve, then $\Delta_{\overline Z}-\overline g^*B_{\overline T}\geq0$.
\end{lem}

\begin{proof}
\emph{Step 1.}
Note that $B_g\geq0$ since $\Delta_Z\geq0$. Let $R_T\subseteq\overline T$ and $R_Z\subseteq\overline Z$ be the ramification divisors of the finite maps $\alpha$ and $\overline \alpha$, respectively. We have 
\begin{equation}\label{suppf}
 \Supp R_Z\subseteq \overline{g}^{-1}(\Supp R_T)
\end{equation}
since the base change by an \'etale map is \'etale. We can write 
\begin{equation}\label{ram1f}
 R_T=R'_T+\alpha^*B_T-(\alpha^*B_T)_\red,
\end{equation}
where $R'_T\geq0$, and $R_T'$ and $\alpha^*B_T$ have no common components.
By \cite[Lemma 5.1]{Amb04} we have
\begin{equation}\label{eq:two}
K_{\overline{T}}+B_{\overline{g}}=\alpha^*(K_T+B_g)\quad\text{and}\quad M_{\overline{g}}=\alpha^*M_g,
\end{equation}
where $M_{\overline{g}}$ is the moduli part of $\overline g$. Then \eqref{ram1f} gives
\begin{equation}\label{eq:ffff}
 B_{\overline{g}}=\alpha^*B_g-R_T=(\alpha^*B_g)_\red-R'_T.
\end{equation}
Similarly, we can write 
\begin{equation}\label{ram2f}
 R_Z=R_Z'+\overline \alpha^*\Delta_{Z,v}-(\overline \alpha^*\Delta_{Z,v})_\red,
\end{equation}
where $R'_Z\geq0$, and $R_Z'$ and $\overline \alpha^*\Delta_{Z,v}$ have no common components. Then \eqref{ram2f} implies
\begin{align}\label{eq:Delta}
 \Delta_{\overline{Z}}&=\overline \alpha^*\Delta_Z-R_Z =\overline \alpha^*\Delta_{Z,h} +\overline \alpha^*\Delta_{Z,v}-R_Z\\
 &=\overline \alpha^*\Delta_{Z,h} +(\overline \alpha^*\Delta_{Z,v})_\red-R_Z'.\notag
\end{align}
We claim that for a prime divisor $P\subseteq \alpha(\Supp R'_T)$,
\begin{equation}\label{eq:cclaim}
\text{$g^*P$ is reduced over the generic point of $P$.}
\end{equation} 
Indeed, otherwise we would have $P\subseteq\Supp B_g$ by the definition of the discriminant. However, 
this would contradict the fact that $R_T'$ and $\alpha^*B_g$ have no common components.

Let $U\subseteq \overline T$ be a big open subset with the following property: $\overline g$ is weakly semistable over $U$, and if a prime divisor $D\subseteq \overline g^{-1}(\Supp R_T')$ is $\overline g$-exceptional, then $\overline g(D)\cap U=\emptyset$. We show in Steps 2 and 3 that $U$ is satisfies (i) and (ii). 

\medskip

\emph{Step 2.}
To show (i), by \eqref{eq:ffff} it is enough to prove 
\begin{equation}\label{eq:fff1f}
\overline{g}^*\big((\alpha^*B_T)_\red\big)|_{\overline g^{-1}(U)}=(\overline \alpha^*\Delta_{Z,v})_\red|_{\overline g^{-1}(U)}.
\end{equation}
For \eqref{eq:fff1f}, we have
$$ (\overline \alpha^*\Delta_{Z,v})_\red=(\overline \alpha^*g^*B_g)_\red=(\overline{g}^*\alpha^*B_g)_\red,$$
where the first equality follows by pulling back the relation $\Delta_{Z,v}=(g^*B_{g})_\red$ by $\overline \alpha$ and taking the reduced part, 
and the second equality by the base change diagram. Since $\overline g^*\big((\alpha^*B_g)_\red\big) |_{\overline g^{-1}(U)}$ is reduced, 
we have 
$$(\overline{g}^*\alpha^*B_g)_\red|_{\overline g^{-1}(U)}=\overline{g}^*\big((\alpha^*B_g)_\red\big)|_{\overline g^{-1}(U)},$$
which proves (i).

\medskip
 
\emph{Step 3.}
Finally, we show (ii). By \eqref{eq:Delta}, it suffices to show 
\begin{equation}\label{eq:fff2f}
R_Z'|_{\overline g^{-1}(U)}=\overline{g}^*R'_T|_{\overline g^{-1}(U)}.
\end{equation}
By \eqref{suppf}, \eqref{ram1f}, \eqref{ram2f} and \eqref{eq:fff1f} we have
\begin{align*}
(\Supp R'_Z)|_{\overline g^{-1}(U)}&\cup(\Supp\overline \alpha^*\Delta_{Z,v})|_{\overline g^{-1}(U)}\\
&=(\Supp R_Z)|_{\overline g^{-1}(U)}\subseteq \overline g^{-1}(\Supp R_T|_U)\\
&\subseteq \overline g^{-1}(\Supp R_T'|_U)\cup(\Supp\overline{g}^*\alpha^*B_T)_{\overline g^{-1}(U)}\\
&=\overline g^{-1}(\Supp R_T'|_U)\cup(\Supp\overline \alpha^*\Delta_{Z,v})|_{\overline g^{-1}(U)}.
\end{align*}
Since $R_Z'$ and $\overline \alpha^*(\Delta_{Z,v})$ have no common components, this implies
$$(\Supp R_Z')|_{\overline g^{-1}(U)}\subseteq\overline g^{-1}(\Supp R_T'|_U).$$ 
Therefore, for \eqref{eq:fff2f} it is enough to show -- by the definition of $U$ -- that 
for each prime divisor $D\subseteq \overline g^{-1}(\Supp R_T')$ such that $\overline g(D)$ is a divisor in $\overline T$ we have
\begin{equation}\label{eq:eeeq}
\mult_D R_Z'=\mult_D\overline{g}^*R'_T.
\end{equation}

Fix such a prime divisor $D$. Denote $Q:=\overline g(D)$ and $P:=\alpha(Q)$, and let $e_Q=\mult_Q \alpha^*P$. Then
\begin{align}\label{eq:ramify}
\mult_D\overline{g}^*R'_T&=\mult_D \overline{g}^*\big(\alpha^*P-(\alpha^*P)_\red\big)\\
&=(e_Q-1)\mult_D \overline{g}^*Q=e_Q-1,\notag
\end{align}
where the last equality follows since $\overline{g}^*Q$ is reduced over the generic point of $Q$ by the assumption on weak semistability.
Furthermore, by the commutativity of the base change diagram, we also have
$$\mult_D \overline \alpha^*g^*P=\mult_D\overline{g}^*\alpha^*P=e_Q \mult_D \overline{g}^*Q=e_Q.$$
Since $g^*P$ is reduced over the generic point of $P$ by \eqref{eq:cclaim}, this shows that the ramification index of $\overline \alpha$ along $D$ is $e_Q$, which together with \eqref{eq:ramify} gives
$$\mult_D R_Z=\mult_D\overline{g}^*R'_T.$$
To finish the proof of \eqref{eq:eeeq} and of (ii), by \eqref{ram2f} we only need to show that $D\not\subseteq\Supp\overline \alpha^*\Delta_{Z,v}$. Assume otherwise: then $Q\subseteq \Supp\alpha^*B_T$ by \eqref{eq:fff1f}, hence $Q$ would not be a component of $R_T'$ by the construction of $R_T'$ in Step 1, a contradiction.
\end{proof}

\begin{pro}\label{pro:isomfibres}
Let $f\colon(X,\Delta)\rightarrow Y$ be a klt-trivial fibration of normal projective varieties with $X$ $\mathbb{Q}$-factorial. Assume $\Delta$ effective over the generic point of $Y$ and $\Delta-f^*B_f\geq0$.
Assume that $Y$ is an Ambro model and $M_f$ is semiample, let $\phi\colon Y\to V$ be the fibration induced by $M_f$.
Let $Y_r$ be the set of points $x\in Y$ such that $f^{-1}x$ is reduced.
Then there are a non empty open set $V_0\subseteq V$, an open subset $Y_0\subseteq Y$ with complement of codimension at least 2 and a set $I(Y)\supseteq \phi^{-1}V_0\cap Y_0\cap Y_r$ with the following property: for every
 $x_1,x_2\in I(Y)$ such that $\phi(x_1)=\phi(x_2)$, if $(F_i,\Delta_i)$ is the fibre over $x_i$ with $\Delta_i=\Delta^h\vert_{F_i}$,
then $(F_1,\Delta_1)\cong (F_2,\Delta_2)$.
\end{pro}
\begin{proof}
We apply Theorem \ref{ambro1} and find $\vartheta$ and $\chi$ and a diagram such that $\vartheta^*M_f=\chi^*M_{f^+}$.
In particular both $\vartheta^*M_f$ and $M_{f^+}$ are semiample. After passing to the Stein factorisation we can assume that $\chi$ has connected fibres.
Let $\widetilde X$ be the main component of the normalisation of $X\times_Y\widetilde Y$ with the natural morphism $\tau\colon \widetilde X\to X$ and let $\widetilde \Delta$ be defined by $K_{ \widetilde X}+\widetilde \Delta=\tau^*(K_X+\Delta)$.
By Theorem \ref{ambro1} there is an open set $\widetilde U\subseteq\widetilde Y$ and an isomorphism
$$(\widetilde X,\widetilde \Delta)\times_{\widetilde Y}\widetilde U\cong (X^+,\Delta^+)\times_{Y^+} \widetilde U .$$ 
By Proposition \ref{ambro2}, the isomorphism extends to 
$$(\widetilde X,\widetilde \Delta-\tilde f^*B_{\tilde f})\times_{\widetilde Y} \widetilde Y_0\to (X^+, \Delta^+)\times_{Y^+} \widetilde Y_0$$
with $\widetilde Y_0=\widetilde Y\setminus \tilde f(\Supp (\widetilde \Delta_v-\tilde f^* B_{\tilde f})^{<0})$.

There is a diagram 
$$
\xymatrix{
Y\ar[d]_{\phi} & \widetilde Y\ar[d]^{\tilde\phi} \ar[l]_{\vartheta}\ar[r]^{\chi}&Y^+\\
V& \widetilde V \ar[l]^{\sigma}&
}
$$
where $\sigma\circ\tilde\phi$ is the Stein factorisation of $\phi\circ\theta$.


By Lemma \ref{lem:excvert} there is a generically finite map $\lambda\colon Y^+\to \widetilde V$, and it is birational because $\tilde \phi$ is a fibration.

After passing to an open set $U^+$ of $Y^+$ we can assume that $\lambda$ is an isomorphism  
and let  $\widetilde V_0=\lambda(U^+)$.

By Lemma \ref{lem:excvert} and Lemma \ref{lem:reducedfibres}
after possibly shrinking $\widetilde V_0$ further, we can assume that the complement of 
  $\tilde f(\Supp (\widetilde \Delta_v-\tilde f^* B_{\tilde f})^{<0})\cap \tilde\phi^{-1}\widetilde V_0$ has codimension at least 2 in $\tilde\phi^{-1}\widetilde V_0$.

 Then, for $\tilde x_1,\tilde x_2\in \tilde\phi^{-1}\widetilde V_0$ in the same fibre of $\tilde\phi$, the two corresponding fibres are isomorphic, together with the boundaries.
 
Let $x_1,x_2\in Y_r\cap \vartheta^{-1}\tilde\phi^{-1}\widetilde V_0=Y_r\cap \phi^{-1}\sigma^{-1}\widetilde V_0$.
  If $\phi(x_1)=\phi(x_2)$, then there are $\tilde x_1,\tilde x_2\in \phi^{-1}\widetilde V_0$ such that $\theta(\tilde x_i)=x_i$ and $\tilde\phi(\tilde x_1)=\tilde\phi(\tilde x_2)$. 
  By Lemma \ref{redgen}, the restriction of $\tau$ to $\tilde f^{-1}(\tilde x_i)$ is an isomorphism, concluding the proof.
 \end{proof}

We also need the following \cite[Theorem 3.5]{Amb05a}; see also \cite[Theorem 1.2]{Flo14} for a sharper version.

\begin{thm}\label{thm:torsionAmbro}
Let $f\colon(X,\Delta)\to Y$ be a klt-trivial fibration, and assume that the moduli part $M_Y$ descends to $Y$. If $M_Y\equiv0$, then $M_Y\sim_\Q 0$.
\end{thm}

\subsection{Variation of a klt-trivial fibration}
In this section we give the definition and some properties of the variation of a fibration. 
For the original definition with $\Delta=0$ and some further discussion of the properties see \cite{Vieh83, Kol87, Fuj03}.

\begin{dfn}\label{def:var}
Let $(X,\Delta)$ be a pair and 
 let $f\colon(X,\Delta)\to Y$ be a fibration. We define the variation of $f$, denoted by $Var(f)$ as
 $$\min\left\{  \dim Y^+\left\vert 
 \begin{array}{l}
   \exists \vartheta\colon\widetilde Y\to Y\text{ generically finite}\\
 \exists \chi\colon\widetilde Y\to Y^+, f^+\colon(X^+,\Delta^+)\to Y^+\text{ fibrations}\\
 \text{such that the fibration induced by}\, f,\vartheta \text{ by  fibre}\\
 \text{product and the fibration induced by}\, f,\chi \text{ by  fibre}\\
 \text{product are birational over }\widetilde Y.
 \end{array}
 \right. \right\}$$
\end{dfn}

The following is a generalisation of \cite[Theorem 3.8]{Fuj03} to the case $\Delta\neq 0$, the proof is essentially the same.
\begin{pro}\label{pro:fujinovar}
 Let $f\colon(X,\Delta)\to Y$ be an lc-trivial fibration, assume $Y$ is an Ambro model. Then $\kappa(M_f)\leq Var(f)$.
\end{pro}
\begin{proof}
 Thet $\vartheta,\chi$ be such that $Var(f)=transdeg_k k(Y^+)$.
 Then, after perhaps passing to higher models of $\widetilde Y$ and $Y^+$, we have $\theta^*M_f=\chi^*M_{f^+}$.
 Therefore $\kappa(M_f)=\kappa(M_{f^+})\leq\dim Y^+=Var(f)$.
 
\end{proof}

\begin{pro}\label{pro:finvar}
 Let $(X,\Delta)$ be a pair with $X$ $\mathbb Q$-factorial and $\Delta\geq 0$ and with coefficients in $\mathbb Q$.
 Let $f\colon(X,\Delta)\to Y$ be a fibration such that $Var(f)=\dim Y$.
 Then there is a countable union $E$ of closed subsets of $Y$ and an open set $U\subseteq Y$ such that for every $y\in Y\setminus E$
 the set
 $$\{z\in U\vert\; (f^{-1}y,\Delta^h\vert_{f^{-1}y})\cong (f^{-1}z,\Delta^h\vert_{f^{-1}z})  \}$$
 is a finite set.
\end{pro}
\begin{proof}
 Set $F=f^{-1}y$.
 We fix a polarisation $A=p_F^*A_F+p_X^*A_X$ on $F\times X$, where $p_F$ and $p_X$ are the two projections.
 Then there is a quasi projective scheme $\Emb(F,X)\subseteq\Hilb(F\times X)$ representing the functor
 $\mathcal H ilb_P(F\times X)$ where $P$ is the Hilbert polynomial of the graph of  $F\to f^{-1}y\subseteq X$.
 There is also a universal family $u\colon \Univ(F,X)\to\Emb(F,X)$ and a diagram
 $$
 \xymatrix{
 \Univ(F,X)\ar[r]\ar[d]_u&F\times X\times \Emb(F,X)\ar[r]^>(.9){p_X}&X\ar[d]^f\\
 \Emb(F,X)&&Y.
 }
 $$
 After perhaps replacing the polarisation $A$ with $A+p_X^*f^*A_Y$ for a sufficiently ample divisor $A_Y$ on $Y$,
 all the fibres of $u$ are contracted by $f\circ p_X$.
 By the rigidity lemma there is $\phi\colon \Emb(F,X)\to Y$.
 Its image is $\{z\in Y\vert\; f^{-1}z\cong F\}$.
 
 Let $k\in\mathbb N$ be such that $k\Delta^h$ is a Cartier divisor. Set $D=k\Delta^h\vert_F$.
 Then there is a locally closed subscheme
 $\Emb\left((F,D),(X,\mathcal D)\right)$ of $\Emb(F,X)$ representing the functor $\mathcal{E}mb\left((F,D),(X,\mathcal D)\right)$
 with
 $$\mathcal E mb\left((F,D),(X,\mathcal D)\right)(Z)=\left\{
 \begin{array}{l}
  Z-\text{morphisms }\varphi\colon F\to X,\psi\colon D\to\mathcal D\\
  \text{flat over }Z\\
  \text{such that }\varphi\circ i=j\circ\psi, \varphi\text{ embedding}
 \end{array}
\right\}$$
 together with a universal family $$u\colon\Univ\left((F,D),(X,\mathcal D)\right)\to\Emb\left((F,D),(X,\mathcal D)\right).$$
 Therefore $$\{z\in U\vert\; (f^{-1}y,\Delta^h\vert_{f^{-1}y})\cong (f^{-1}z,\Delta^h\vert_{f^{-1}z})  \}=\phi(\Emb\left((F,D),(X,\mathcal D)\right))$$
 is the image of an algebraic set.
 By \cite[Theorem 2.6]{Kol87} for $y$ in the complement of a countable union of closed sets in $Y$
 the left hand side is at most countable. Therefore it is a finite set.
\end{proof}

%% file: supS.tex
\section{Semiample line bundles on simple normal crossings surfaces}\label{def:snc}

In this section we establish a criterion of semiampleness of certain line bundles on simple normal crossings surfaces.
For later use and for this section we introduce different notions of simple normal crossings varieties

\bigskip

 Let $\mathcal{Z}$ be a variety with irreducible components $\{Z_i : i \in I\}$. Assume that $\dim Z_i=k$ for every $ i \in I$. 
We say that $\mathcal{Z}$ is a \textit{simple normal crossing
variety}  \cite[Definition 6]{Kol14} if the $Z_i$ are smooth and every point $p\in \mathcal Z$ has an
open (Euclidean) neighborhood $p\in U_p\subseteq \mathcal Z$ and an embedding $U_p\to \mathbb C^{k+1}$
such 
that the image of $U_p$ is an open subset of the union of coordinate hyperplanes
$(z_1\cdot\ldots\cdot z_n = 0)$ with $n\leq k+1$.
A stratum of $\mathcal{Z}$  is any irreducible component of an intersection $\cap_{i\in J} Z_i$ for some
$J\subseteq I$.

Assume now $\mathcal{Z}=\cup_k \mathcal{Z}^{(k)}$ where $\mathcal{Z}^{(k)}$ is the union of irreducible components of dimension $k$.
We say that $\mathcal{Z}$ is a \textit{simple normal crossing
variety}   if $\mathcal{Z}^{(k)}$ is simple normal crossings in the above sense for every $k$ and for every $k$, for every stratum $Z$ of 
$\mathcal{Z}_{k-1}=\cup_{j<k}\mathcal{Z}^{(j)}$ we have that $\mathcal{Z}^{(k)}$ does not contain $Z$ and every 
point $p\in \mathcal Z^{(k)}\cap Z$ has an
open (Euclidean) neighborhood $p\in U_p\subseteq \mathcal Z^{(k)}\cup Z$ and an embedding $U_p\to \mathbb C^{k+1}$
such that he image of $U_p$ is an open subset of the union of coordinate hyperplanes
$(z_1\cdot\ldots\cdot z_{n} = 0)$ and the image of $Z$ is an open subset of $(z_{n+1}=\ldots=z_m=0)$ with $n<m\leq k+1$ \cite[Definition 3.24]{Kol07b}.

Finally, a curve is said to be \textit{seminormal} if every point has a neighbourhood if and only if it is analytically isomorphic to
the union of the n coordinate axes in $\mathbb A^n$ \cite[Example 10.12]{Kol13}.


\begin{rem}
 If $\mathcal{Z}\subseteq Y$ has pure codimension 1 and is a simple normal crossings variety, then it is a simple normal crossings divisor.
 
 If $\mathcal{Z}$ is a connected simple normal crossings variety of pure dimension 1, then it is a semistable curve.

\end{rem}

A divisor on a simple normal crossings variety can be recovered from its restrictions to its irreducible components plus a gluing condition on the intersections:

Let $\mathcal{Z}$ be a simple normal crossings variety of pure dimension $k$.
A divisor on $\mathcal{Z}$ is the data of a divisor $D_Z$ on every irreducible component $Z$ of $\mathcal{Z}$
with the property that, if $Z_1$ and $Z_2$ are two irreducible components of $\mathcal{Z}$, then $D_{Z_1}\vert_{Z_1\cap Z_2}=D_{Z_2}\vert_{Z_1\cap Z_2}$.

\bigskip

From now on, we assume that $\mathcal{S}$ is a simple normal crossings variety of pure dimension 2. We refer to $\mathcal{S}$ as a simple normal crossings surface.

\begin{lem}\label{lem:surf1}
 Let $\mathcal{S}$ be a connected simple normal crossings surface. 
 Assume that there is an integral curve $Q$ and a surjective morphism with connected fibres $\varphi\colon \mathcal{S}\to Q$, that for every $S\subseteq\mathcal S$ irreducible component
 $\varphi(S)$ is an irreducible curve.
 For an irreducible component $S$ of $\mathcal S$, we denote by $\varphi\vert_S\colon S\overset{f_S}{\rightarrow} C(S)\to Q $ the Stein factorisation. 
 Let $D$ be an effective divisor on $\mathcal{S}$ such that $\varphi(\Supp D)\subseteq Q^{\text{smooth}}$.
 Then there is a positive integer $m$ such that $mD$ is the pullback of a Cartier divisor in $Q$ if and only if for every irreducible component $S$ of $\mathcal S$
 there is a positive integer $d$ such that
 the restriction of $dD$ to $S$ is the pullback of a divisor in $C(S)$.
\end{lem}
\begin{proof}

 If there is a positive integer $m$ such that $mD$ is the pullback of a divisor in $Q$, then the statement on the restrictions of $D$ to the irreducible components of $\mathcal{S}$
 is obvious.
 
 We assume now that for every irreducible component $S$ of $\mathcal S$
 there is a positive integer $d$ such that
 the restriction of $dD$ to $S$ is the pullback of a divisor in $C$.
 
 By hypothesis, there are $p_1,\ldots,p_k$ in the smooth locus of $Q$ such that the support of $D$ is contained in $\varphi^{-1}\{p_1,\ldots,p_k\}$.
 We prove the statement by induction on $k$.
 If $k=0$, there is nothing to prove.
 Assume now that the statement holds for $k-1$.
 Let $Q_1$ be the irreducible component of $Q$ such that $p_1\in Q_1$.
 Let $S$ be an irreducible component of $\mathcal{S}$ such that $\varphi(S)=Q_1$.
 We set $D\vert_S=\sum_\ell\sum_j a_{\ell,j}F_{\ell,j}$ where for every $\ell$ the union $\cup_j F_{\ell,j}$ is a connected component of $\Supp D\vert_S$.
 Without loss of generality, we can assume that $\cup_j F_{1,j}$ is contained in $\varphi^{-1}p_1$.
 Let $\alpha$ be such that $\varphi^*(\alpha p_1)\vert_S=\sum_j a_{1,j}F_{1,j}+\sum_h\sum_j b_{h,j}F_{h,j}$.
 We want to prove that $\Supp D-\varphi^*(\alpha p_1)\subseteq\varphi^{-1}\{p_2,\ldots,p_k\}$.
 
 Assume that this is not the case, that is, assume that there is $S'$ and an irreducible component $F$ of $\varphi^{-1}p_1\cap S'$ such that $\text{coeff}_F(D-\varphi^*(\alpha p_1))$
 is not zero.
 The fibre $\varphi^{-1}p_1$ is connected, thus there are $S=S_0, S_1,\ldots,S_N=S'$ and for every $i$ a subvariety $\cup_j F^i_j$ of $S_i$ 
 and a point $q_i\in S_i$ with the following properties:
 \begin{itemize}
  \item $\cup_j F^i_j$ is the support of a fibre of $f_{S_i}$;
  \item $\cup_j F^0_j=\cup_j F_{1,j}$;
  \item $F\subseteq \cup_j F^N_j$;
  \item $q_i\in (\cup_j F^i_j)\cap (\cup_j F^{i+1}_j)$.
 \end{itemize}
We have that for every $j$ 
$$\text{coeff}_{F^0_j} (D-\varphi^*(\alpha p_1))=0\;\;\;\;\text{and}\;\;\;\; \text{coeff}_{F^N_j}( D-\varphi^*(\alpha p_1))\neq 0.$$
Then there is $i$ such that for every $j$ 
$$\text{coeff}_{F^i_j} (D-\varphi^*(\alpha p_1))=0\;\;\;\;\text{and}\;\;\;\; \text{coeff}_{F^{i+1}_j} (D-\varphi^*(\alpha p_1))\neq 0.$$
This is a contradiction as $D-\varphi^*(\alpha p_1)$ is a divisor on $\mathcal S$ and $D-\varphi^*(\alpha p_1)\vert_{S_i}$ but 
 $D-\varphi^*(\alpha p_1)\vert_{S_{i+1}}$ do not coincide on the intersection $S_i\cap S_{i+1}$.
\end{proof}

\begin{lem}\label{lem:surf2}
 Let $\mathcal{S}$ be a connected simple normal crossings surface. 
 Assume that there is a seminormal curve $Q$ and a surjective morphism with connected fibres $\varphi\colon \mathcal{S}\to Q$ and that for every $S\subseteq\mathcal S$ irreducible component
 $\varphi(S)$ is an irreducible curve.
 Let $\mathcal L$ be a line bundle on $\mathcal{S}$ such that 
 for every fibre $F$ of $\varphi$ the restriction $\mathcal L\vert_{(F)_{\text{red}}}$ has a nowhere vanishing section.
 Then there is a positive integer $m$ and a line bundle $\mathcal M$ on $Q$ such that $\mathcal L^{\otimes m}\sim\varphi^*\mathcal M$.
\end{lem}

\begin{proof}

Let $Q=\cup\overline Q_i$ be the decomposition of $Q$ into irreducible components.
Let $\mathcal{S}_i$ be the union of the irreducible components $S$ of $\mathcal{S}$ such that $\varphi(S)=\overline Q_i$ and let $\varphi\colon \mathcal{S}_i\overset{f_i}{\longrightarrow} Q_i\overset{\nu_i}{\longrightarrow}\overline Q_i $ be the Stein factorisation. 
The morphism $\nu_i$ is birational and finite.
We prove that $Q_i$ is normal.
Indeed for every irreducible component $S_{i,j}$ of $\mathcal{S}_i$ the restriction of $f_i$ to 
$S_{i,j}$ factors through the normalisation $Q_i^{\nu}$ of $Q_i$ and there is $f_{i,j}\colon S_{i,j}\to Q_i^{\nu}$.
As $\mathcal{S}_i$ has simple normal crossings, the restriction of $f_i$ to $S_{i,h}\cap  S_{i,k}$
factors through $Q_i^{\nu}$.
Thus, if  and if $x\in S_{i,h}\cap  S_{i,k}$, we have $f_{i,h}(x)=f_{i,k}(x)$ and there is a morphism 
$f'_i\colon \mathcal{S}_i\to Q_i^{\nu}$. By the uniqueness of the Stein factorisation $Q_i=Q_i^{\nu}$.

Since it is a curve, it is a smooth projective curve and the morphism $f_i$ is flat.

\medskip

The sheaf $f_{i*}(\mathcal L^{\vee}\vert_{\mathcal{S}_i})$ has generically rank 1 on $Q_i$ and, by semicontinuity, all its stalks are non zero. 
 Let $A_i$ be an ample line bundle on $Q_i$. After possibly replacing $A_i$ with a multiple, we can assume that $f_{i*}(\mathcal L^{\vee})\otimes A_i$ is globally generated and has therefore a non-zero global section. 
 Moreover, by the projection formula, as $A_i$ is locally free, we have
 $$H^0(\mathcal S_i, \mathcal L^{\vee}\otimes f_i^*A_i)=H^0(Q_i, f_{i*}(\mathcal L^{\vee}\otimes f_i^*A_i))=H^0(Q_i, f_{i*}(\mathcal L^{\vee})\otimes A_i).$$
Then there is a non-zero global section $s\in H^0(\mathcal S_i, \mathcal L^{\vee}\otimes f_i^*A_i)$
inducing an isomorphism of line bundles $\mathcal L^{\vee}\otimes f_i^*A_i\sim\mathcal O(D_i)$ with $D_i$ an effective Cartier divisor on $\mathcal{S}_i$.
On the general fibre the morphism $\mathcal L\to f_i^*A_i$ is an isomorphism, therefore $D_i$ is supported on fibres of $f_i$.

\medskip

For an irreducible component $S$ of $\mathcal S_i$, we denote by $f_i\vert_S\colon S\overset{f_S}{\longrightarrow} C\to Q_i $ the Stein factorisation.
By Zariski's lemma \cite[Lemma 8.2]{BPV}, for every irreducible component $S$ of $\mathcal{S}_i$, the restriction $D_i\vert_S$ is proportional to fibres of $f_S$.
By Lemma \ref{lem:surf1}, the divisor $D_i$ is proportional to fibres of $f_i$ and $\mathcal O(D_i)\sim_{\Q} f_i^*\delta_i$ with $\delta_i\geq 0$.
After tensoring $A_i$ with a higher multiple, we can assume that $\Supp \delta_i\subseteq \nu_i^{-1}Q^{\text{reg}}$.

\medskip

For $p\not\in Q^{\text{reg}}$ let $F_p=\varphi^{-1}p$. We notice that $F_p$ is a semistable curve.
As $Q$ is seminormal, 
 there are $A$ and $\delta$ on $Q$ such that $A\vert_{Q_i}=A_i$ and $\delta\vert_{Q_i}=\delta_i$ for every $i$ such that $f_i^*A_i\vert_{F_p}=\mathcal L\vert_{F_p} =f_i^*\delta_i\vert_{F_p}$ for every $p\not\in Q^{\text{reg}}$.

It follows that $\mathcal L\sim \varphi^*A(-\delta)$.
\end{proof}

\begin{rem}
To prove that $Q_i$ is normal we could also have argued in the following way.
Let $\sqcup S_{i,j}$ be the normalisation of $\mathcal{S}_i$ and let for every $i,j$
be $f_i\colon S_{i,j}\to V_{i,j}\overset{\sigma_{i,j}}{\longrightarrow} Q_i $ be the Stein factorisation of $\varphi_{S_{i,j}}$. Then $Q_i$ is the quotient of  $\sqcup V_{i,j}$
by the relation, for $x\in V_{ih}$ and $y\in V_{ik}$, $x\sim y$ if and only if $\sigma_{ih}(x)=\sigma_{ik}(y)$. This equivalence relation is finite, equidimensional and $\sqcup V_{i,j}$ is normal.
By \cite[Proposition 9.14]{Kol13} the curve $Q_i$ is normal as well.
\end{rem}

\begin{thm}\label{thm:sup}
 Let $\mathcal{S}$ be a connected simple normal crossings surface. 
 Assume  that there is an integral seminormal curve $Q$ and a surjective morphism with connected fibres $\varphi\colon \mathcal{S}\to Q$.
 Let $\mathcal L$ be a line bundle on $\mathcal{S}$ such that for every $S\subseteq \mathcal{S}$ the restriction $\mathcal L\vert_S$
 is semiample and the Stein factorisation of $\varphi\vert_S$ is the morphism induced by $\mathcal L\vert_S$.
 Assume that for every fibre $F$ of $\varphi$ the restricted line bundle $\mathcal L\vert_{(F)_{\text{red}}}$ has a non-zero section.
 Then $\mathcal L$ is semiample.
\end{thm}
\begin{proof}
Let $x\in \mathcal S$ be a point. We want to prove that there is a global section of $\mathcal L$ non zero along $x$.

We write $\mathcal{S}=\mathcal{S}_0\cup\mathcal{S}_1$ where 
$$\mathcal{S}_0=\{S\vert\; \mathcal L\vert_S\text{ has Kodaira dimension 0}\}$$
$$\mathcal{S}_1=\{S\vert\; \mathcal L\vert_S\text{ has Kodaira dimension 1}\}$$

Let $\varphi\colon \mathcal{S}_1\overset{f}{\longrightarrow} Q'\overset{\nu}{\longrightarrow}\overline Q_i $ be the Stein factorisation. 
The morphism $\nu$ is birational and finite.
By Lemma \ref{lem:surf2} there is a positive integer $m$ and a line bundle $\mathcal M$ on $Q'$ such that $\mathcal L^{\otimes m}\sim f^*\mathcal M$.
The line bundle $\mathcal M$ is ample on $Q'$.
After maybe taking a multiple of $m$, there is
 a global section  $s$ of $\mathcal M$ which is non zero on every irreducible component of $Q'$,
such that if $\nu(p_1)=\nu(p_2)$ then $s(p_1)=s(p_2)$ and such that $s(\varphi(x))\neq 0$.

Set $\varphi(\mathcal S_0)=\{q_1,\ldots,q_k\}$ and $F_i=\varphi^{-1}q_i$ taken with the reduced structure.
Thus for every $i$ we chose a global section $s_i$ of $\mathcal L\vert_{F_i}$ agreeing with $\varphi^*s$
on $F_i\cap \mathcal S_1$.
Thus the data $s_1,\ldots,s_k,f^*s$ define a global section of $\mathcal L$ which does not vanish on $x$.

\end{proof}

%% file: profiniteS.tex
\section{Profinite equivalence relations}\label{Prof}
Let $X$ be a scheme. A \textit{relation} on $X$ is the data of a scheme $\mathcal S$ and an embedding $\sigma\colon\mathcal S\to X\times X$ \cite[Definition 9.1]{Kol13}.
It is \textit{finite} if the projections $\sigma_i\colon\mathcal S\to X$ are finite for $i=1,2$.
A \textit{set theoretic equivalence relation}, or \textit{equivalence relation} for short, is a relation such that $\sigma$ is geometrically injective, $\mathcal S$ contains the diagonal (reflexive), is invariant by the involution of $X\times X$ exchanging the two factors (symmetric) and is transitive, that is, if we consider
$$
\xymatrix{
\mathcal S\times_X\mathcal S\ar[r]\ar[d]&\mathcal S\ar[d]^{\sigma_2}\\
\mathcal S\ar[r]_{\sigma_1}&X
}
$$
then there is a natural morphism $\sigma\colon\mathcal S\times_X\mathcal S\to X\times X\times X$
and $(\pi_1,\pi_2)\circ \sigma\left(\red(\mathcal S\times_X\mathcal S)\right)\to X\times X$ factors through $\mathcal S$ \cite[Definition 9.2]{Kol13}.

\begin{rem}\label{rem:feq}
If $\mathcal R$ is a finite equivalence relation on an algebraic variety (not necessarily irreducible) and $Z\subset X$ a subvariety then $\mathcal S Z=\{z\in X\vert\;\text{ there is}\;z'\in Z \text{ with}\;(z,z')\in\mathcal R\}$ is a finite union of subvarieties of $X$. Indeed, we have
$\mathcal R Z=\sigma_2\sigma_1^{-1}Z$.
\end{rem}

\begin{dfn}
Let $\mathcal R$ be an equivalence relation on $X$.
A subset $Z\subseteq X$ is invariant by $\mathcal R$ if one of the following equivalent condition is verified:
\begin{itemize}
\item for every $x\in X$, if there is $z\in Z$ which is equivalent to $x$, then $x\in Z$;
\item for every $x\in X$, if there is $z\in Z$ such that $(x,z)\in \mathcal R$ then $x\in Z$;
\item $\sigma_2\sigma_1^{-1}Z\subseteq Z$;
\item $\sigma_1\sigma_2^{-1}Z\subseteq Z$.
\end{itemize}
\end{dfn}

\begin{const}(Equivalence closure)\label{constr:eqclosure} 
The equivalence closure $\langle \mathcal S\rangle$ of a relation $\mathcal S$ is the smallest equivalence relation containing it.
We refer to \cite[9.3 ]{Kol13} for the complete construction, which consists in making $\mathcal S$
reflexive, symmetric and transitive.
We recall just that if $S_1, S_2\subseteq \mathcal S$ are irreducible components, then in order to make $\mathcal S$ transitive, we ``add" to $\mathcal S$ the variety $S_3=(\sigma_1\circ\pi_1,\sigma_2\circ\pi_2)(S_1\times_X S_2)$, where $\pi_i$ are the natural projections
$$
\xymatrix{
&&S_1\times_X S_2\ar[ld]_{\pi_1}\ar[rd]^{\pi_2}&&\\
&S_1\ar[ld]_{\sigma_1}\ar[rd]^{\sigma_2}&&S_2\ar[ld]_{\sigma_1}\ar[rd]^{\sigma_2}&\\
X&&X&&X
}
$$
\end{const}
An equivalence relation is called \textit{profinite} if it is the closure of a finite relation.

\begin{nt}
We denote by $\mathcal S_K$ (resp. $\mathcal S_{\leq K}$) the union of all the irreducible components of $\mathcal S$ of dimension $K$ (resp. $\leq K$). 
\end{nt}

\begin{lem}\label{lem:picche}
Let $X$ be a normal variety of maximal dimension $D$.
Notation as in Construction \ref{constr:eqclosure}. Then
\begin{enumerate}
\item $\dim S_3\leq\min\{\dim S_1,\dim S_2\}$;
\item if $S_1,S_2\subseteq \mathcal S_D$ and $\sigma_2(S_1)=\sigma_1(S_2)$, then every component of $S_3$ has dimension $\dim S_1=\dim S_2$; 
\item $\sigma_1(S_3)\subseteq \sigma_1(S_1)$, $\sigma_2(S_3)\subseteq \sigma_2(S_2)$. 
\end{enumerate}
\end{lem}
\begin{proof}

Both $\pi_1$ and $\pi_2$ are finite morphisms as they are the base change of $\sigma_1$ and $\sigma_2$ respectively, which are finite. Thus $$\dim S_1\times_X S_2\leq\min\{\dim S_1,\dim S_2\}.$$

If $\sigma_2(S_1)=\sigma_1(S_2)$, then $\dim S_1=\dim S_2$ because $\sigma_i$ is finite for $i=1,2$. Since $\pi_i$ is surjective,  $\dim S_1\times_X S_2=\min\{\dim S_1,\dim S_2\}$.\\
Assume that $S_1,S_2$ have dimension $D$. Then their image $X_1$ in $X$ is an irreducible component, and therefore normal.
By Chevalley's criterion \cite[14.4.4]{EGAIV} and \cite[Definition 1.44]{Kol13} the morphism $\sigma_2\colon S_1\to X_1$ is universally open.
Then $\sigma'_2\colon S_1\times_{X_1}S_2\to S_1$ is open and finite.\\
Let $ S_1\times_{X_1}S_2=W_1\cup\ldots\cup W_{\ell}$ be the decomposition into irreducible components. Then $U_1=W_1\cap(W_2\cup\ldots\cup W_{\ell})^c$ is open in $ S_1\times_{X_1}S_2$
and its image $\sigma'_2(U_1)$ is open in $S_2$.
Thus $\dim W_1=\dim U_1=\dim \sigma'_2(U_1)=\dim S_2$.

Since $\pi_i$ and $\sigma_i$ are finite for $i=1,2$, $(\sigma_1\circ\pi_1,\sigma_2\circ\pi_2)$ is finite as well, proving (1) and (2).
 
As for (3), we have $\sigma_i(S_3)=\sigma_i((\sigma_1\circ\pi_1,\sigma_2\circ\pi_2) S_1\times_X S_2)=(\sigma_i\circ\pi_i)( S_1\times_X S_2)\subseteq\sigma_i(S_i)$. 

\end{proof}

\begin{rem}\label{rem:picche} We denote by $\Delta_X$ the diagonal of $X\times X$
By Lemma \ref{lem:picche}, if $D=\dim (\mathcal S\setminus \Delta_X)$, then 
$\langle \mathcal S_D\rangle_D=\langle \mathcal S\rangle_D $.
\end{rem}

The following lemma is a slight generalisation of \cite[2.7]{BB}. It is a consequence of Lemma \ref{lem:picche}(2) for which we followed closely the proof of \cite[Lemma 9.14]{Kol13}.

\begin{lem}\label{lem:fiori}
Let $X$ be a normal variety of dimension $D$. Let $\mathcal S$ be a finite relation on $X$. Then 
$\langle \mathcal S_D\rangle=\langle \mathcal S\rangle_D $.
In particular $\langle \mathcal S\rangle_D $ is an equivalence relation.
\end{lem}
\begin{proof}
By Remark \ref{rem:picche} it is enough to prove that $\langle \mathcal S_D\rangle_D=\langle \mathcal S_D\rangle $.
One inclusion is obvious. For the other one, if $S_1,S_2$ are two components of $\mathcal S_D$, since $X$ is normal and $D=\dim X$, then either $\sigma_1(S_1)=\sigma_2(S_2)$ or 
$\sigma_1(S_1)\cap \sigma_2(S_2)=\emptyset$.
Therefore by Lemma \ref{lem:picche}(2), every irreducible component of $S_1\times_X S_2$ and of its projection in $X\times X$ has dimension $D$.

\end{proof}

\begin{lem}\label{lem:claim}
Let $X$ be a normal variety of dimension $D$. Let $\mathcal S$ be a finite reflexive and symmetric relation on $X$ and $\mathcal R$ the equivalence closure of $\mathcal S$. Assume that  $\mathcal R_D$ is finite. Then
$X_1=\mathcal R_D(\sigma_1(\mathcal S_{\leq D-1})\cup\sigma_2(\mathcal S_{\leq D-1}))$  is $\mathcal R$-invariant. 
\end{lem}
\begin{proof}
The set $X_1$ is $\mathcal R_D$-invariant.
It is enough to prove that $\sigma_j(\mathcal R_{\leq D-1})\subseteq X_1$ for $j=1,2$.
Set $\mathcal S^i= (\sigma_1\circ\pi_1,\sigma_2\circ\pi_2)\mathcal  S^{i-1}\times_X\mathcal  S^{i-1}$.
By \cite[9.3]{Kol13}, the equivalence closure of $\mathcal S$ is $\mathcal R=\cup \mathcal S^i$.
We will prove by induction on $i$ that $\sigma_j((\mathcal  S^i)_{\leq D-1})\subseteq X_1$ for $j=1,2$.

We have 
\begin{equation*}
\begin{split}
\mathcal  S^i&=(\sigma_1\circ\pi_1,\sigma_2\circ\pi_2)\mathcal  S^{i-1}\times_X\mathcal  S^{i-1}\\
&=(\sigma_1\circ\pi_1,\sigma_2\circ\pi_2)(\mathcal  S^{i-1}_D\times_X\mathcal  S^{i-1}_D)\cup\\ &(\sigma_1\circ\pi_1,\sigma_2\circ\pi_2)(\mathcal  S^{i-1}_D\times_X\mathcal  S^{i-1}_{\leq D-1}\cup\mathcal  S^{i-1}_{\leq D-1}\times_X\mathcal  S^{i-1}_D\cup \mathcal  S^{i-1}_{\leq D-1}\times_X\mathcal  S^{i-1}_{\leq D-1})
\end{split}
\end{equation*}
By Lemma \ref{lem:picche}(2) we have $(\sigma_1\circ\pi_1,\sigma_2\circ\pi_2)(\mathcal  S^{i-1}_D\times_X\mathcal  S^{i-1}_D)\subseteq \mathcal R_D$ and by Lemma \ref{lem:picche}(1) $$(\sigma_1\circ\pi_1,\sigma_2\circ\pi_2)( \mathcal  S^{i-1}_D\times_X\mathcal  S^{i-1}_{\leq D-1}\cup\mathcal  S^{i-1}_{\leq D-1}\times_X\mathcal  S^{i-1}_D\cup \mathcal  S^{i-1}_{\leq D-1}\times_X\mathcal  S^{i-1}_{\leq D-1})\subseteq\mathcal R_{\leq D-1} .$$
Therefore 
$$\mathcal  S^i_{\leq D-1}=(\sigma_1\circ\pi_1,\sigma_2\circ\pi_2)( \mathcal  S^{i-1}_D\times_X\mathcal  S^{i-1}_{\leq D-1}\cup\mathcal  S^{i-1}_{\leq D-1}\times_X\mathcal  S^{i-1}_D\cup \mathcal  S^{i-1}_{\leq D-1}\times_X\mathcal  S^{i-1}_{\leq D-1})$$
By Lemma \ref{lem:picche}(3) we have $\sigma_2(\sigma_1\circ\pi_1,\sigma_2\circ\pi_2)( \mathcal  S^{i-1}_D\times_X\mathcal  S^{i-1}_{\leq D-1})\subseteq \sigma_2 S^{i-1}_{\leq D-1}$ and by induction $\sigma_2 S^{i-1}_{\leq D-1}\subseteq X_1$, proving 
\begin{equation}\label{eq:side}
\sigma_2(\sigma_1\circ\pi_1,\sigma_2\circ\pi_2)( \mathcal  S^{i-1}_D\times_X\mathcal  S^{i-1}_{\leq D-1})\subseteq X_1.
\end{equation}
As for $\sigma_1(\sigma_1\circ\pi_1,\sigma_2\circ\pi_2)( \mathcal  S^{i-1}_D\times_X\mathcal  S^{i-1}_{\leq D-1})$, we have
 \begin{equation*}
 \begin{split}
 \sigma_1(\sigma_1\circ\pi_1,\sigma_2\circ\pi_2)( \mathcal  S^{i-1}_D\times_X\mathcal  S^{i-1}_{\leq D-1})
 \subseteq \sigma_1\left(\sigma_2\vert_{\mathcal R_D}\right)^{-1}\sigma_1\mathcal  S^{i-1}_{\leq D-1}\\
 \subseteq\sigma_1\left(\sigma_2\vert_{\mathcal R_D}\right)^{-1}X_1\subseteq X_1
 \end{split}
 \end{equation*}
 where the first inclusion is because
 \begin{equation*}
 \begin{split}
 (\sigma_1\circ\pi_1,\sigma_2\circ\pi_2)( \mathcal  S^{i-1}_D\times_X\mathcal  S^{i-1}_{\leq D-1})\\
=\{(x,y)\vert\; \exists z\in X,\; (x,z)\in\mathcal  S^{i-1}_D \;(z,y)\in\mathcal  S^{i-1}_{\leq D-1} \}\\
\subseteq \{(x,y)\vert\; \exists z\in \sigma_1(S^{i-1}_{\leq D-1}),\; (x,z)\in\mathcal  R_D  \}
 \end{split}
 \end{equation*}
 and the image via $\sigma_1$ of the last set coincides with $\sigma_1\left(\sigma_2\vert_{\mathcal R_D}\right)^{-1}\sigma_1\mathcal  S^{i-1}_{\leq D-1}$.
 The second inclusion follows by induction and the third because
 $X_1$ is $\mathcal R_D$-invariant.

 A very similar proof implies that $\sigma_j(\sigma_1\circ\pi_1,\sigma_2\circ\pi_2)\mathcal  S^{i-1}_{\leq D-1}\times_X\mathcal  S^{i-1}_D\subseteq X_1$ for $j=1,2$.

Again by Lemma \ref{lem:picche}(3) we have $\sigma_j(\sigma_1\circ\pi_1,\sigma_2\circ\pi_2)( \mathcal  S^{i-1}_{\leq D-1}\times_X\mathcal  S^{i-1}_{\leq D-1})\subseteq \sigma_2 S^{i-1}_{\leq D-1}$ for $j=1,2$ and $ \sigma_2 S^{i-1}_{\leq D-1}\subseteq X_1$ by induction.
\end{proof}

\begin{dfn}
Let $\mathcal S\to X\times X$ be a finite relation and $g\colon \tilde X\to X$ a finite morphism.
The pullback of $\mathcal S$ by $g$ is $g^*\mathcal S=\mathcal S\times_{X\times X} \tilde X\times \tilde X$.
\end{dfn}

\begin{lem}\label{lem:quadri}
Let $\mathcal S$ be a finite relation on a variety $X$ and let $D=\dim\mathcal S\setminus\Delta_X$.
Let $g\colon \tilde X\to X$ be a finite surjective morphism.
If $\langle\mathcal S_D\rangle_D$ is infinite, then $\langle g^*\mathcal S_D\rangle_D$ is infinite.

\end{lem}

\begin{proof}
For every $D$-dimensional component $S$ of $\langle\mathcal S_D\rangle$ the pull back $\widetilde S$ in $\widetilde X\times\widetilde X$ has dimension $D$.

\end{proof}


\begin{dfn}
A profinite equivalence relation $\mathcal R$ on an equidimensional variety $X$ is equidimensional if every irreducible component of $\mathcal R$ projects onto a connected component of $X$.
\end{dfn}
The definition coincides with what is called wide in \cite[Definition 2.1]{BB}.

\begin{pro}\label{pro:profinite}
Let $\mathcal S$ be a finite relation on a normal variety $X$, let $\mathcal R$ be the equivalence closure of $\mathcal S$. If $\mathcal R$ is not finite then there are
\begin{enumerate}
\item a subrelation $\mathcal R'\subseteq \mathcal R$
\item $Z_1,\ldots,Z_k$ subvarieties of $X$
\end{enumerate}
such that $\cup Z_i$ is $\mathcal R'$-invariant, $\mathcal R'\vert_{\cup Z_i}$ is an infinite equidimensional relation and the set of infinite equivalence classes is dense in $\cup Z_i$.
\end{pro}
\begin{proof}
We prove the statement by induction on $D=\dim X$.
If $\mathcal R_D$ is not finite, we let $Z_i$ be the irreducible components of $X$ of dimension $D$ which are dominated by infinitely many components of $\mathcal R_D$ and we set $\mathcal R'=\mathcal R_D$.\\
We assume now that $\mathcal R_D$ is finite.
We set $X_1=\mathcal R_D(\sigma_1(\mathcal S_{\leq D-1})\cup\sigma_2(\mathcal S_{\leq D-1}))$.
By Lemma \ref{lem:claim} the subvariety $X_1$ is $\mathcal R$-invariant.

By Lemma \ref{lem:quadri} the pullback of the restriction of $\mathcal R$ to $X_1$ via the normalisation of $X_1$ is not finite. We conclude by induction as the dimension of the normalisation of $X_1$ is at most $D-1$.
\end{proof}





%% file: cerbero1S.tex
\section{Gluing bases of fibrations}\label{Glu}

Throughout this section,
$\mathcal L$ will be a line bundle with the property that $\mathcal L\vert_T$ is semiample for every irreducible component $T\subseteq \mathcal T$.
For every $T$ we denote by $\phi_T\colon T\to V$ the  fibration  induced by a multiple of  $\mathcal L$.

\begin{dfn}
The equivalence relation $R_{\mathcal L}$ on the set $\bigsqcup_{T\in\mathcal T} V$ is the closure of the relation

$$x_1\sim x_2 \Leftrightarrow \exists T, T'\subseteq\mathcal T,\; \exists y\in T\cap T'\;\; \phi_T(y)=x_1, \phi_{T'}(y)=x_2. $$

\end{dfn}

\begin{rem}
Assume that $\mathcal T$ is a simple normal crossing divisor.
Let $\nu\colon\sqcup T\to\mathcal T$ be the normalisation. Let $\Xi^n$ be the normalisation of the non-normal locus of $\mathcal T$. Then there is an involution $\zeta\colon\Xi^n\to \Xi^n $ and we have
$(\zeta_1,\zeta_2)\colon\Xi^n\to \sqcup T \times \sqcup T $.
Let $\upsilon\colon\Xi^n\to\sqcup W $ be the fibration induced by $\mathcal L\vert_{\Xi}$.
The morphism $(\zeta_1,\zeta_2)$ induces a morphism $(\xi_1,\xi_2)\colon \sqcup W \to\sqcup V$.
Then the equivalence relation $(\xi_1,\xi_2)\colon \sqcup W \to\sqcup V$ coincides with $R_{\mathcal L}$.
\end{rem}

\begin{nt}
Let $\nu\colon\sqcup T\to\mathcal T$ be the normalisation.
For a subset $S\subseteq\sqcup V$ we will denote by $\phi^{-1}S$ the set $\nu\left(\sqcup \phi_T^{-1}(S\cap V)\right)$.
\end{nt}

\begin{rem}
If the line bundle $\mathcal L$ restricted to $\mathcal T$  is base point free, then the relation $\sim$ is finite and $\mathcal L$ induces a morphism $\phi\colon\mathcal T\to (\bigsqcup_{T\in\mathcal T} V)/R_{\mathcal L}$
\end{rem}

\begin{dfn}
Let $\mathcal T\subseteq Y$ be a divisor and let $\mathcal L$ be a line bundle such that $\mathcal L\vert_T$ is base point free for every $T\subseteq\mathcal T$ for every irreducible component.
Let $\phi_T\colon T\to V$ be the morphism induced by $\mathcal L$.
For an equivalence class $[x]$ of $R_{\mathcal L}$ we set the pseudo-fibre
as $$\mathcal T_{[x]}=\cup_{
x'\in[x]} 
\phi_T^{-1}(x')=\phi^{-1}[x].
$$

\end{dfn}

\begin{rem}
The relation $R_{\mathcal L}$ is finite if and only if $\mathcal T_{[x]}$  is an algebraic variety for every $[x]$. 
Indeed $R_{\mathcal L}$ is finite if and only if $[x]$ is a finite set for every $x$.
\end{rem}

\begin{pro}\label{pro:fmap}
Let $Y$ be a normal variety and let $\mathcal T\subseteq Y$ be a divisor.
Let $\mathcal L$ be a line bundle on $Y$ which is semiample on the irreducible components of $\mathcal T$. Let $\tau\colon \overline Y\to Y$ be a finite map and $\overline{\mathcal T}=\tau^{-1}\mathcal T$.
Then there is a commutative diagram 
$$
\xymatrix{
 \bigsqcup_{\overline T\subseteq\overline{\mathcal T}}\overline T\ar[r]^{\tau}\ar[d]_{\left( \phi_{\overline T}\right)}&\bigsqcup_{ T\subseteq\mathcal T} T\ar[d]^{\left(\phi_T\right)}\\
\bigsqcup\overline V\ar[r]_{\sigma}&\bigsqcup V.
} 
$$
with $\sigma$ a finite map. Moreover 
 for every $x\in \sqcup V$ we have $\sigma^{-1}[x]=\sqcup_{\sigma(\bar x)=x}[\bar x]$.
 
 In particular $\mathcal R_{\tau^*\mathcal L}=\sigma^*\mathcal R_{\mathcal L}$
\end{pro}

\begin{proof}
Assume that $\tau(\overline T)=T$.
There is a commutative diagram 
$$
\xymatrix{
\overline T\ar[r]^{\tau}\ar[d]_{\phi_{\overline T}}&T\ar[d]^{\phi_T}\\
\overline V\ar[r]_{\sigma_V}&V
} 
$$
where $\sigma_V\circ\phi_{\overline T}=\phi_T\circ\tau$.
And the $\sigma_V$ define a finite map $\sigma\colon \bigsqcup\overline V\to \bigsqcup V$.

Let $\bar x,\bar x'\in \bigsqcup\overline V$ such that $\bar x\sim\bar x'$.
Then there is $\bar y\in \overline T\cap\overline T'$ such that $\phi_{\overline T}(\bar y)=\bar x$ and $\phi_{\overline T'}(\bar y)=\bar x'$. By the commutativity of the diagram 
$\phi_{T}(\tau\bar y)=\sigma(\bar x)$ and $\phi_{T'}(\tau\bar y)=\sigma(\bar x')$.
Therefore $\sigma[\bar x]\subseteq [\sigma(\bar x)]$.

On the other hand let $\bar x\in \sigma^{-1}[x]$.
We want to prove that $\bar x$ is equivalent to a point in $\sigma^{-1}x$.
The point $\sigma\bar x$ is equivalent to $x$.
Therefore there are $\sigma\bar x\sim x_1\sim\ldots\sim x_k=x$.
We prove our statement by induction on $k$. If $k=1$, the statement is obvious.
We assume from now on that $k>1$.
Then $\sigma\bar x\sim x_1$ if and only if $\phi_T^{-1}(\sigma\bar x)\cap \phi_{T'}^{-1}(x_1)\neq\emptyset$. Therefore $\phi_{\overline T}^{-1}(\bar x)\cap \tau^{-1}\phi_{T'}^{-1}(x_1)\neq\emptyset$. Let $\bar y_1\in \phi_{\overline T}^{-1}(\bar x)\cap \tau^{-1}\phi_{T'}^{-1}(x_1)$ and $\bar x_1=\phi_{\overline T'}\bar y_1$. Then $\sigma(\bar x_1)= x_1$ and we can conclude by the inductive hypothesis.

\end{proof}

\begin{pro}\label{pro:bmap}
Let $Y$ be a normal variety and let $\mathcal T\subseteq Y$ be a divisor.
Let $\mathcal L$ be a line bundle on $Y$ which is semiample on the irreducible components of $\mathcal T$. Let $\varepsilon\colon \overline Y\to Y$ be a birational map
which is an isomorphism on the generic points of $T\cap T'$ for every $T,T'$ irreducible component of $\mathcal T$.
Let $\overline{\mathcal T}$ be the strict transform of $\mathcal T$.
Then there is a commutative diagram 
$$
\xymatrix{
 \bigsqcup_{\overline T\subseteq\overline{\mathcal T}}\overline T\ar[r]^{\varepsilon}\ar[rd]_{\left( \phi_{\overline T}\right)}&\bigsqcup_{ T\subseteq\mathcal T} T\ar[d]^{\left(\phi_T\right)}\\
&\bigsqcup V.
} 
$$
with $\phi_{\overline T}=\phi_T\circ\varepsilon$.
Then $\mathcal R_{\varepsilon^*\mathcal L}=\mathcal R_{\mathcal L}$.
\end{pro}

\begin{proof}
 It is enuogh to prove that the equivalence classes coincide.
 Let $x_1,x_2\in\bigsqcup V$ be such that there is $\bar y\in \overline T_1\cap\overline T_2$ with $\phi_{\overline T_i}(\bar y)=x_i$.
 The divisor $\overline T_i$ is the strict transform of $T_i\subseteq Y$.
 Then $y=\varepsilon (\bar y)$ is such that $\phi_{\overline T_i}(\bar y)=x_i$.
 This proves that $\mathcal R_{\varepsilon^*\mathcal L}\subseteq\mathcal R_{\mathcal L}$.
 
 Let $x_1,x_2\in\bigsqcup V$ be such that there is $y\in  T_1\cap T_2$ with $\phi_{T_i}( y)=x_i$.
 Let $\overline T_i$ be the strict transform of $T_i\subseteq Y$.
 As $\varepsilon$ is an isomorphism  on the generic point of $T_1\cap T_2$, the intersection $T_1\cap T_2\cap\varepsilon^{-1}y$ is non empty.
 If $\bar y$ is in the intersection, then $\phi_{\overline T_i}(\bar y)= \phi_{T_i}(\varepsilon y)=x_i$.
 This proves that $\mathcal R_{\varepsilon^*\mathcal L}\supseteq\mathcal R_{\mathcal L}$, concluding the proof.
\end{proof}

\begin{cor}\label{cor:frel}
Let $Y$ be a normal variety and let $\mathcal T\subseteq Y$ be a divisor.
Let $\mathcal L$ be a line bundle on $Y$ which is semiample on the irreducible components of $\mathcal T$.
Let $\theta\colon\overline Y\to Y$ be a generically finite map such that $\theta Exc(\theta)$ does not contain 
the generic points of $T\cap T'$ for every $T,T'$ irreducible component of $\mathcal T$. 
Let $\overline{\mathcal T}$ be the strict transform of $\mathcal T$.

Then there is a commutative diagram 
$$
\xymatrix{
 \bigsqcup_{\overline T\subseteq\overline{\mathcal T}}\overline T\ar[r]^{\theta}\ar[d]_{\left( \phi_{\overline T}\right)}&\bigsqcup_{ T\subseteq\mathcal T} T\ar[d]^{\left(\phi_T\right)}\\
\bigsqcup\overline V\ar[r]_{\sigma}&\bigsqcup V.
} 
$$
with $\sigma$ a finite map
such that $\mathcal R_{\theta^*\mathcal L}=\sigma^*\mathcal R_{\mathcal L}$.
\end{cor}
\begin{proof} Let $\theta=\varepsilon\circ\tau$ be the Stein factorisation.
Let $\widetilde{\mathcal T}=\tau^{-1}\mathcal T$.
By Proposition \ref{pro:fmap} there is a diagram 
$$
\xymatrix{
 \bigsqcup_{\widetilde T\subseteq\widetilde{\mathcal T}}\widetilde T\ar[r]^{\tau}\ar[d]_{\left( \phi_{\widetilde T}\right)}&\bigsqcup_{ T\subseteq\mathcal T} T\ar[d]^{\left(\phi_T\right)}\\
\bigsqcup\widetilde V\ar[r]_{\sigma}&\bigsqcup V.
} 
$$
with $\sigma$ a finite map such that $\mathcal R_{\tau^*\mathcal L}=\sigma^*\mathcal R_{\mathcal L}$.
By Proposition \ref{pro:bmap} there is a diagram 

$$
\xymatrix{
 \bigsqcup_{\overline T\subseteq\overline{\mathcal T}}\overline T\ar[r]^{\varepsilon}\ar[rd]_{\left( \phi_{\overline T}\right)}& \bigsqcup_{\widetilde T\subseteq\widetilde{\mathcal T}}\widetilde T\ar[d]^{\left( \phi_{\widetilde T}\right)}\\
&\bigsqcup\widetilde V.
} 
$$
 \end{proof}

\begin{lem}\label{lem:fconn}
Let $\mathcal T\subseteq Y$ be a reduced and connected divisor and let $\mathcal L$ be a line bundle such that $\mathcal L\vert_T$ is base point free for every $T\subseteq\mathcal T$ for every irreducible component.
Let $\phi_T\colon T\to V$ be the morphism induced by $\mathcal L$. 
Then $\mathcal T_{[x]}$ is connected.
\end{lem}
\begin{proof}
Let $y_1,y_2\in \mathcal T_{[x]}$. 
Then there are $x_2,\ldots,x_r$ such that $\phi_1(y_1)=x_1\sim x_2\ldots x_r\sim x_{r+1}=\phi_2(y_2)$.
Let $V_i$ be such that $x_i\in V_i$ and $T_i$ with $\phi_i\colon T_i\to V_i$.
Then there are $y_{i,i+1}\in T_{i,i+1}$ such that
$\phi_i(y_{i,i+1})=x_i$ and $\phi_{i+1}(y_{i,i+1})=x_{i+1}$.
Thus $$y_1,y_2\in \bigcup_{i=1}^{r+1}\phi_i^{-1}(x_i)\subseteq \mathcal T_{[x]}$$
and $\cup_{i=1}^{r+1}\phi_i^{-1}(x_i)$ is connected as for every $i$ there is $y_{i,i+1}\in\phi_i^{-1}(x_i)\cap\phi_{i+1}^{-1}(x_{i+1})$.

\end{proof}

%% file: graphS.tex
\section{Graph theory}\label{Graph}
We recall here a few basic notions of graph theory.
We follow the presentation of \cite{Stallings}.

A \textit{graph} $\Gamma$ consists of two sets $E$ and $V$ (edges and vertices), and two functions $E\to E$, $e\mapsto\bar e$ and $E\to V$, $e\mapsto i(e)$: for each $e \in E$, there is an element $\bar e\in E$, and an element $i(e)\in V$. The function $\bar{\cdot}$ is such that $\bar{\bar e}=e$ and $\bar e\neq e$. The vertex $i(e)$ is called the initial vertex of $e$, the vertex $t(e)=i(\bar e)$ is called the terminal vertex of $e$.

We call a graph \textit{finite} if both $V$ and $E$ are finite sets.

A \textit{map of graphs} $f\colon \Gamma_1\to\Gamma_2$ consists of a pair of functions, edges to edges, vertices to vertices, preserving the structure.
A map of graphs is \textit{surjective} if it is surjective on vertices and on edges.

We recall that \textit{pull-backs} exist in the category of graphs: given $f_1\colon\Gamma_1\to\Delta$ and  $f_2\colon\Gamma_2\to\Delta$ two maps of graphs, there is a graph $\Gamma_1\times_{\Delta}\Gamma_2$ together with surjective maps $g_i\colon \Gamma_1\times_{\Delta}\Gamma_2\to \Gamma_i$ such that $f_1\circ g_1=f_2\circ g_2$.

A \textit{path} in a graph $\Gamma$ is an $n$-tuple of edges $(e_1,\ldots,e_n)\in E^n$ such that $t(e_i)=i(e_{i+1})$. The vertices $i(e_1)$ and $t(e_n)$ are the initial vertex and terminal vertex of the path.

A \textit{circuit} is a path whose initial and terminal vertex coincide.
Equivalently, we define $C_n$ the standard circuit of length $n$ as the regular polygone with $n$ edges and 
a circuit in $\Gamma$ is a map of graphs $C_n\to\Gamma$.
A circuit is \textit{proper} if the map $C_n\to\Gamma$ is injective on the vertices.
The \textit{standard arc of length $n$} $A_n$ can be described as the interval $[0, n]$ subdivided at the integral points.
The vertices are $V=\{0,\ldots,n\}$, the edges are the oriented segments $[i, i+1]$ and $[i+1, i]$ between $i$ and $i+1$. The involution $\bar{\cdot}$ 
exchanges $[i, i+1]$ and $[i+1, i]$.

The \textit{homotopy equivalence} on paths is the relation generated by
$$(e_1,\ldots,e_n)\sim (e_1,\ldots,e_i,e,\bar e,e_{i+1},\ldots,e_n)$$ and the set of paths starting and ending at a same vertex $v$ modulo homotopy is denoted by $\pi_1(\Gamma,v)$ and called the fundamental group of $\Gamma$.
It has a natural group structure with respect to the concatenation of paths.

A path is \textit{reduced} if it contains no sub-paths of the form $e\bar e$ and one can prove that every path is homotopic to a reduced one.

Let $v$ be a vertex of the graph $\Gamma$. The \textit{star} of $v$ in $\Gamma$ is the set
$$St(v,\Gamma)=\{e\in E\vert\; i(e)=v\}.$$
A map of graphs $f\colon \Gamma_1\to\Gamma_2$ is a \textit{covering} if for each vertex $v$ of $\Gamma_1$ the natural function
$$f_v\colon St(v,\Gamma_1)\to St(f(v),\Gamma_2)$$
is bijective.
By \cite[4.1(d)]{Stallings} if $f\colon \Gamma_1\to\Gamma_2$ is a covering, then
$f\colon \pi_1(\Gamma_1,v)\to\pi_1(\Gamma_2,f(v))$ is an injective homomorphism and if the graphs are finite then $f \pi_1(\Gamma_1,v)\subseteq\pi_1(\Gamma_2,f(v))$ has finite index equal to the cardinality of $f^{-1}f(v)$.

This last remark combined with \cite[3.3]{Stallings} and \cite[4.4]{Stallings}, gives the following proposition
\begin{pro}\label{prop:liftingcircuits}
 If $f\colon \Gamma_1\to\Gamma_2$ is a surjective maps of finite graphs, then $f \pi_1(\Gamma_1,v)\subseteq\pi_1(\Gamma_2,f(v))$ has finite index $i\leq |f^{-1}f(v)|$.
\end{pro}

We conclude this section with an easy but useful lemma.
\begin{lem}\label{lem:surjcircuit}
Let $\Gamma$ be a finite graph. Then there is a standard circuit $C_N$ and a surjective morphism $C_N\to \Gamma$.
\end{lem}
\begin{proof}
We construct recursively a morphism $f\colon A_N\to \Gamma$. We notice that if $f$ is surjective on the edges then it is surjective on the vertices and that it is enough to show that for every $e\in E$ either there is $i$ such that $e=f[i,i+1]$ or such that $\bar e=f[i,i+1]$.
We set $\pi\colon E\to \hat E$ the quotient by the action of $\mathbb Z/2\mathbb Z$ sending $e$ to $\bar e$.
Let $e\in E$. We set $f[0,1]=e$.
Assume we have $f\colon A_k \to \Gamma$.
If $\hat E\setminus\pi\left(\{f[i,i+1]\}_i\right)$ is not empty, then we pick $e\in E\setminus\{f[i,i+1],\overline{f[i,i+1]}\}_i$, we pick a path $(e_1,\ldots,e_n)$ from $t(f[k-1,k])$
to $i(e)$ and we set
$$
\left\{
\begin{array}{lcl}
f[k+i-1,k+i]&=&e_i \;for\;i\leq n\\
f[k+n+1,k+n+2]&=&e.
\end{array}
\right.
$$
If $\hat E\setminus\pi \left(\{f[i,i+1]\}_i\right)$ is empty, then we pick a path $(e_1,\ldots,e_n)$ from $t(f[k-1,k])$
to $i(f[0,1])$ and we set $f[k+i-1,k+i]=e_i$ for $i\leq n$.

\end{proof}

%% file: trivialSS.tex
\section{Trivial line bundles on simple normal crossings varieties}\label{Trivial}

In this section we discuss a triviality condition for line bundles on reducible varieties and develop the tools for the proof of Theorem \ref{thm:step2}.
We are mostly concerned with the case of simple normal crossings varieties in the sense of Definition \ref{def:snc}. Lemmas \ref{lem:trivcirc}, \ref{lem:tors} and \ref{lem:tors1} can be seen as a refinement of \cite[Example 9.2.8]{NeMo}.

\begin{dfn}
Let $\mathcal Z=\cup Z$ be a reducible variety. We define the incidence graph $\Gamma^i(\mathcal Z)$ of  $\mathcal Z$ by $V^i=\{Z\vert\; Z\;irreducible\;component\;of\;\mathcal Z\}$ with an edge between $Z$ and $Z'$ for every connected component of $Z\cap Z'$.

\end{dfn}

\begin{nt}
A circuit $\mathcal C$ in $\Gamma^i(\mathcal Z)$ will be denoted by 
$$(\{Z_1,\ldots,Z_k\}, Z_{1,2}\ldots Z_{k,1})\text{ or }
(\{Z_i\},Z_{i,i+1})\text{  for short}$$
where the $Z_i$ are irreducible components of 
$\mathcal Z$ and for every the varietiy $Z_{i,i+1}$ is a connected component  of $Z_i\cap Z_{i+1}$, and $Z_{k,1}$ is a connected component of $Z_1\cap Z_k$.

We will refer to $\cup Z_i$ as the support of the circuit $\mathcal C$.
\end{nt}

\begin{rem}
If $\mathcal Z$ is a divisor with simple normal crossing support,
then $\Gamma^i(\mathcal Z)$ coincides with the 1-skeleton of the dual complex of $\mathcal Z$ (see \cite[Section 2]{dFKX17}).
\end{rem}

Throughout this subsection $Y$ will be a normal connected variety and $\mathcal Z\subseteq Y$ a reducible reduced and connected subvariety of $Y$. We will consider $\mathcal L$ a line bundle on $Y$ such that $\mathcal L\vert_Z\sim\mathcal O_Z$ for every irreducible component $Z$ of $\mathcal Z$.
\begin{dfn}
Let $\mathcal C=(\{Z_i\},Z_{i,i+1})$ be a circuit in $\Gamma^i(\mathcal Z)$. A section of the restriction of $\mathcal L$ to $\mathcal C$ (or of $\mathcal L\vert_{\mathcal C}$) is the data of $s_i\in H^0(Z_i,\mathcal L)$
such that $$s_i\vert_{Z_{i,i+1}}=s_{i+1}\vert_{Z_{i,i+1}}.$$
\end{dfn}

\begin{lem}\label{lem:trivcirc}
Let $\mathcal Z$ be a connected reduced simple normal crossings variety of pure dimension $k$. Let $\mathcal L$ be a line bundle on $\mathcal Z$ such that $\mathcal L\vert_Z\sim\mathcal O_Z$ for every irreducible component $Z$ of $\mathcal Z$.
Then $\mathcal L\vert_{\mathcal Z}\sim\mathcal O_{\mathcal Z}$
if and only if for every circuit $\mathcal C$ in $\Gamma^i(\mathcal Z)$, the restriction of $\mathcal L$ to $\mathcal C$ has a nowhere vanishing global section.
\end{lem}
\begin{proof}
If $\mathcal L$ is trivial, then it has a nowhere vanishing global section $s\in H^0(\mathcal Z,\mathcal L)$. Then for any circuit $\mathcal C=(\{Z_i\},Z_{1,2},\ldots,Z_{k,1})$  in $\Gamma^i(\mathcal Z)$ it is enough to set $s_i=s\vert_{Z_i}$.

Conversely, let $\mathcal Z=\bigcup Z_i$ be the decomposition of $\mathcal Z$ into its irreducible components.
 By possibly relabeling, we may assume that for any $i>1$ the subvariety $Z_i$ meets $\mathcal Z_{i-1}:=\bigcup_{j<i} Z_j$. Let $\mathcal Z_i=\mathcal Z_{i-1}\cup Z_i$. For each $i\neq j$ with $Z_i\cap Z_j\neq\emptyset$, we fix $p_{i,j}\in Z_i\cap Z_j$.
 Fix $s_1\in H^0(Z_1, \mathcal L)\setminus\{0\}$.
 We construct inductively a nowhere-vanishing section $\sigma_i\in H^0(\mathcal Z_i,\mathcal L) $
 such that $\sigma_i\vert_{Z_1}=s_1$.
 
 For $i>1$ we assume there is a section $\sigma_{i-1}\in H^0(\mathcal Z_{i-1},\mathcal L) $.
 Choose the largest $r<i$ such that $Z_i\cap Z_r\neq\emptyset$, and let $s_i\in H^0(Z_i,\mathcal L)\setminus\{0\}$ be the unique section such that 
 \begin{equation}\label{eq:trcirc1}
s_i(p_{i,r})=\sigma_{i-1}|_{Z_r}(p_{i,r}).
\end{equation} 

If $Z_j\cap Z_r=\emptyset$ for all $j<i$ with $j\neq r$ and $Z_i\cap Z_r$ is connected, then \eqref{eq:trcirc1} defines a nowhere-vanishing section $\sigma_i\in H^0(\mathcal Z_i,\mathcal L) $.

Otherwise, there exists $Z_s$ with $s<i$ and a point $p_{i,s}\in Z_i\cap Z_s$.
Then there exists a circuit $\mathcal C=(Z_{i_1},\ldots,Z_{i_k},Z_{1,2},\ldots,Z_{k,1})$
such that $Z_{i_1}=Z_s$, $Z_{i_{k-1}}=Z_r$, $Z_{i_k}=Z_i$, $p_{i,r}\in Z_{k-1,k}$ and $p_{i,s}\in Z_{k,1}$.
By assumption there exists a non-trivial global section of $\mathcal L\vert_{\mathcal C}$, which is the data of $\theta_i\in H^0(Z_{i_j},\mathcal L)$ for $j=1,\ldots,k$. By rescaling, we may assume that $\theta_1=\sigma_{i-1}\vert_{Z_{i_1}}$. Then, by the construction above, for every $1\leq j\leq k-1$ we have $\theta_j=\sigma_{i-1}\vert_{Z_j}$ and 
 $\theta_k=s_i$, and in particular
 \begin{equation}\label{eq:trcirc2}
s_i(p_{i,s})=\sigma_{i-1}|_{Z_s}(p_{i,s}).
\end{equation} 

Since this holds for any choice of $p_{i,s}\in Z_i\cap Z_s$, \eqref{eq:trcirc1} and \eqref{eq:trcirc2} define a nowhere-vanishing section $\sigma_i\in H^0(\mathcal Z_i,\mathcal L)$.
\end{proof}

\begin{dfn}
Let $\mathcal Z$ be a connected reduced simple normal crossings variety of pure dimension $k$.
Let $\mathcal L$ be a line bundle on $\mathcal Z$ such that $\mathcal L\vert_Z\sim\mathcal O_Z$ for every irreducible component $Z$ of $\mathcal Z$.
Let $\mathcal C=(\{Z_1,\ldots,Z_k\}, Z_{i,i+1})$ be a circuit in $\Gamma^i(\mathcal Z)$. We chose $s_1\in H^0(Z_1,\mathcal L)\setminus\{0\}$ and for every $i>1$ we set $s_i\in H^0(Z_i,\mathcal L)\setminus\{0\}$ as the unique section such that
$$s_i\vert_{Z_{i-1,i}}=s_{i-1}\vert_{Z_{i-1,i}}.$$
We define then 
$$
\begin{array}{rcl}
\Phi_{\mathcal L,\mathcal C}\colon H^0(Z_1,\mathcal L)&\to&H^0(Z_1,\mathcal L)\\
s&\mapsto&s\cdot s_{k+1}/s_1
\end{array}
$$
\end{dfn}

\begin{rem}
The map $\Phi_{\mathcal L,\mathcal C}$ is the identity if and only if the restriction of $\mathcal L$ to $\mathcal C$ admits a nowhere vanishing global section.
\end{rem}

It is easy to see that this does not depend on the choice of $s_1$.
Moreover, if $\mathcal C,\mathcal C'$ are circuits based in $Z_1$ and they are homotopically equivalent, then $\Phi_{\mathcal L,\mathcal C}=\Phi_{\mathcal L,\mathcal C'}$.
If $\mathcal C_1,\mathcal C_2$ are circuits based in $Z_1$
and $\mathcal C=\mathcal C_1\star\mathcal C_2$ is their concatenation, then 
$\Phi_{\mathcal L,\mathcal C}=\Phi_{\mathcal L,\mathcal C_2}\circ\Phi_{\mathcal L,\mathcal C_1}$.
All these remarks prove the following lemma.
\begin{lem}\label{lem:tors}
Let $\mathcal Z$ be a connected reduced simple normal crossings variety of pure dimension $k$, let $Z_1\subseteq\mathcal Z$ be an irreducible component.
Let $\mathcal L$ be a line bundle on $\mathcal Z$ such that $\mathcal L\vert_Z\sim\mathcal O_Z$ for every irreducible component $Z$ of $\mathcal Z$.
There is a group homomorphism
$$
\begin{array}{rcl}
\Phi_{\mathcal L}\colon \pi_1(\Gamma^i(\mathcal Z),Z_1)&\to&GL(H^0(Z_1,\mathcal L))\cong \mathbb C^*\\
 \mathcal C &\mapsto&\Phi_{\mathcal L,\mathcal C}
\end{array}
$$
which is trivial if and only if $\mathcal L\sim\mathcal O_{\mathcal Z}$.
\end{lem}

\begin{rem}\label{rem:tors}
In the context of the previous definition, 
for all $m$ we have 
$$\Phi_{\mathcal L^{\otimes m},\mathcal C}=\underbrace{\Phi_{\mathcal L,\mathcal C}\circ\dots\circ\Phi_{\mathcal L,\mathcal C}}_{m \text{ times}}.$$
\end{rem}

\begin{lem}\label{lem:tors1}
Let $\mathcal Z$ be a connected reduced simple normal crossings variety of pure dimension $k$, let $Z_1\subseteq\mathcal Z$ be an irreducible component.
Let $\mathcal L$ be a line bundle on $\mathcal Z$ such that $\mathcal L\vert_Z\sim\mathcal O_Z$ for every irreducible component $Z$ of $\mathcal Z$.
Then $\mathcal L$ is a torsion line bundle if and only if 
the image of $\Phi_{\mathcal L}$ is a finite subgroup of $\mathbb C^*$.
\end{lem}
\begin{proof}
If $\mathcal L$ is torsion, then there is a positive integer $m$ such that 
$\mathcal L^m\sim\mathcal O_{\mathcal Z}$. Therefore for every circuit $\mathcal C$ based in $Z_1$ the map $\Phi_{\mathcal L^{\otimes m},\mathcal C}$ is the identity. The conclusion follows from Remark \ref{rem:tors}.

Conversely, let $m$ be a positive integer such that the image of $\Phi_{\mathcal L}$ is contained in the $m$-th roots of 1.
Then for every circuit $\mathcal C$ in $\Gamma^i(\mathcal Z)$ based in $Z_1$
the map $\Phi_{\mathcal L,\mathcal C}\circ\dots\circ\Phi_{\mathcal L,\mathcal C}$ (composition $m$ times) is the identity.
By Remark \ref{rem:tors} this map is $\Phi_{\mathcal L^{\otimes m},\mathcal C}$, and then the restriction of $\mathcal L^{\otimes m}$ to every circuit $\mathcal C$ in $\Gamma^i(\mathcal Z)$ based in $Z_1$ admits a global section.
The statement follows from Lemma \ref{lem:trivcirc} and from the fact that every circuit is homotopic to a circuit based in $Z_1$.
\end{proof}

\subsection{Trivial line bundles on divisors}

\begin{dfn}\label{taugraphD}
Let $\tau\colon \overline Y\to Y$ be a finite map of normal projective varieties and let $\mathcal Z$ be a connected subvariety of $Y$.
Let $\overline{\mathcal Z}$ be the preimage of $\mathcal Z$ under $\tau$.
Set $\overline{\mathcal Z}=\tau^{-1}\mathcal Z$.
We define a graph $\Gamma^i(\overline{\mathcal Z}, \tau)\subseteq\Gamma^i(\overline{\mathcal Z})$ 
having as vertices the vertices of $\Gamma^i(\overline{\mathcal Z})$  and having an edge between $\overline Z$ and $\overline Z'$ if and only if there is an edge between $\overline Z$ and $\overline Z'$ in $\Gamma^i(\overline{\mathcal Z})$  and $\tau(\overline Z)\neq\tau(\overline Z')$.
\end{dfn}

\begin{const}\label{const:surjD}
Let $\tau\colon \overline Y\to Y$ be a finite map of normal projective varieties and let $\mathcal Z$ be a connected subvariety of $Y$.
Let $\overline{\mathcal Z}$ be the preimage of $\mathcal Z$ under $\tau$.
Then there is a natural map of graphs
$$
\tau\colon \Gamma^i(\overline{\mathcal Z},\tau)\to\Gamma^i(\mathcal Z)
$$
defined on vertices by
$\tau(v_{\overline Z})=v_{\tau\overline Z}$.
To an edge $e$ of $\Gamma^i(\overline{\mathcal Z})$ 
corresponding to a connected  component $\overline Z_0$ of $\overline Z\cap \overline Z'$ the map $\tau$ associates the unique connected  component of 
$\tau\overline Z\cap \tau\overline Z'$ containing $\tau\overline Z_0$.
\end{const}

\begin{lem}\label{lem:covering} Let $\tau\colon \overline Y\to Y$ be a finite map of normal projective varieties and let $\mathcal Z$ be a connected subvariety of $Y$.
Let $\overline{\mathcal Z}$ be the preimage of $\mathcal Z$ under $\tau$.
The map in Construction \ref{const:surjD} is surjective and for every $v$ vertex of $\Gamma^i(\mathcal Z)$ we have $|\tau^{-1}v|\leq \deg \tau$.
\end{lem}
\begin{proof}
The map is clearly surjective on vertices.
Let $Z_0$ be a connected  component of $Z\cap Z'$.
Let $\overline Z$ be an irreducible component of $\overline{\mathcal Z}$ such that $\tau\overline Z=  Z$.
The set $\tau\vert_{\overline Z}^{-1}Z_0$ is not empty and it is contained in $\overline Z\cap\tau^{-1}Z'$. Then there is an irreducible component $\overline Z'$ of
$\tau^{-1}Z'$ meeting $\overline Z$. 
Let $\overline Z_0$ be a connected  component of $\overline Z\cap\overline Z'$. Then $\tau$ sends the edge corresponding to $\overline Z_0$
to the edge corresponding to $Z_0$.
\end{proof}

Combining Lemma \ref{lem:covering} and Proposition \ref{prop:liftingcircuits} we get

\begin{cor}\label{cor:surjP1}
Let $\tau\colon \overline Y\to Y$ be a finite map of normal projective varieties and let $\mathcal Z$ be a connected subvariety of $Y$.
Let $\overline{\mathcal Z}$ be the preimage of $\mathcal Z$ under $\tau$ and fix an irreducible component $\overline Z_1$ of $\overline{\mathcal Z}$.
Then $\tau\pi_1(\Gamma^i(\overline{\mathcal Z},\tau),\overline Z_1)$ has finite index $k$ in $\pi_1(\Gamma^i(\mathcal Z),\tau \overline Z_1)$. Moreover $k\leq\deg\tau$.
\end{cor}

\begin{lem}\label{lem:tors2}
Let $\tau\colon \overline Y\to Y$ be a finite map of degree $d$ of normal projective varieties and let $\mathcal Z\subseteq Y$ be a simple normal crossings divisor.
Let $\overline{\mathcal Z}$ be the preimage of $\mathcal Z$ under $\tau$.
Let $\mathcal L$ be a line bundle on $Y$ such that $\mathcal L\vert_Z\sim \mathcal O_Z$ for every component $Z$ of $\mathcal Z$.
If $\mathcal L\vert_{\mathcal Z}\sim \mathcal O_{\mathcal Z}$
then $\tau^*\mathcal L\vert_{\overline{\mathcal Z}}\sim \mathcal O_{\overline{\mathcal Z}}$.
If $\tau^*\mathcal L\vert_{\overline{\mathcal Z}}\sim \mathcal O_{\overline{\mathcal Z}}$, then 
$\mathcal L^{d!}\vert_{\mathcal Z}\sim \mathcal O_{\mathcal Z}$
\end{lem}
\begin{proof}
If $\mathcal L\vert_{\mathcal Z}\sim\mathcal O_{\mathcal Z}$, then the pullback of the nowhere vanishing global section of $\mathcal L\vert_{\mathcal Z}$ by $\tau$ gives a nowhere vanishing global section of $\tau^*\mathcal L\vert_{\overline{\mathcal Z}}$, settling the first part of the statement.

Conversely, we assume that  $\tau^*\mathcal L\vert_{\overline{\mathcal Z}}\sim\mathcal O_{\overline{\mathcal Z}}$.
Fix an irreducible component $\overline Z_1$ of $\overline{\mathcal Z}$ and 
set $Z_1=\tau\overline Z_1$.
We want to prove that for every circuit $\mathcal C$ in $\Gamma^i(\mathcal Z)$ based on $Z_1$, the morphism $\Phi_{\mathcal L,\mathcal C}^{d!}$ is the identity. As 
$\pi_1(\Gamma^i(\mathcal Z),Z_1)$ is finitely generated, the result will follow from Lemma \ref{lem:tors1}.

\noindent By Corollary \ref{cor:surjP1}, the group $\tau \pi_1(\Gamma^i(\overline{\mathcal Z},\tau),\overline Z_1)$ is a subgroup of $\pi_1(\Gamma^i(\mathcal Z),Z_1)$
of index $k\leq d$.
Therefore there exist a circuit $\overline{\mathcal C}=(\{\overline Z_i\},\overline Z_{i,i+1})$ in $\Gamma^i(\overline{\mathcal Z})$  such that, if we denote by $\mathcal C^k$ the concatenation of $\mathcal C$ with itself $k$ times,
the circuits $\mathcal C^k$ and $\tau\overline{\mathcal C}$
are homotopically equivalent.
By Remark \ref{rem:tors} it is enough to prove that   $\Phi_{\mathcal L,\tau\overline{\mathcal C}}=\Phi_{\mathcal L,\mathcal C}^{k}$ is the identity.
We notice that if $Z_i=\tau(\overline Z_i)$ for $i=1,2$ and $Z_{1,2}=\tau(\overline Z_{1,2})$, we have a commutative diagram
$$
\xymatrix{
H^0(Z_1,\mathcal L)\ar[d]_{\wr}&H^0(Z_{1,2},\mathcal L)\ar[d]_{\wr}\ar[l]\ar[r]&H^0(Z_2,\mathcal L)\ar[d]_{\wr}\\
H^0(\overline Z_1,\tau^*\mathcal L)&H^0(\overline Z_{1,2},\tau^*\mathcal L)\ar[l]\ar[r]&H^0(\overline Z_2,\tau^*\mathcal L)
}
$$
where the horizontal arrows are the restriction isomorphisms and the vertical arrows are isomorphisms induced by the pullback by $\tau$.

Les $s$ be a global section for $\tau^*\mathcal L\vert_{\overline{\mathcal Z}}$.
Then, with the identifications of the previous diagram, the linear map $\Phi_{\mathcal L,\tau\overline{\mathcal C}}$ is the multiplication by $s/s=1$, therefore it is the identity.

Then $\Phi_{\mathcal L,\mathcal C}^{k}$ is the identity, and so is $\Phi_{\mathcal L,\mathcal C}^{d!}$, proving the statement.
\end{proof}

\subsection{Trivial line bundles on semistable curves}

In this subsection we present an analog of Construction \ref{const:surjD} and Lemma \ref{lem:tors2} for curves.

\begin{dfn}\label{taugraphC}
Let $\tau\colon \overline Y\to Y$ be a generically finite map of normal projective varieties and let $\mathcal Z$ be a connected curve in $Y$.
Let $\overline{ Y}$ be the a curve in $\mathcal Z$ such that $\tau\overline{\mathcal Z}=\mathcal Z$.
We define a graph $\Gamma^i(\overline{\mathcal Z}, \tau)\subseteq\Gamma^i(\overline{\mathcal Z})$ 
having as vertices the vertices of $\Gamma^i(\overline{\mathcal Z})$  and having an edge between $\overline Z$ and $\overline Z'$ if and only if there is an edge between $\overline Z$ and $\overline Z'$ in $\Gamma^i(\overline{\mathcal Z})$  and either
\begin{itemize}
\item $\tau{\overline Z}$ and $\tau{\overline Z}'$ are curves in $Y$, or
\item $\tau{\overline Z}$ is a curve in $Y$, and $\tau{\overline Z}'=p\in \tau{\overline Z}$, or
\item $\tau{\overline Z}=\tau{\overline Z}'=p\in Y$.
\end{itemize}
\end{dfn}

\begin{const}\label{const:mapC}
Let $\tau\colon \overline Y\to Y$ be a generically finite map of normal projective varieties and let $\mathcal Z$ be a a simple normal crossings curve in $Y$.
Let $\overline{\mathcal Z}$ be a simple normal crossings curve in $\mathcal Z$ such that $\tau\overline{\mathcal Z}=\mathcal Z$. Let $\overline Z_1\subseteq \overline{\mathcal Z}$ be  such that $Z_1=\tau \overline Z_1$ is a curve in $Y$.

Then there a homomorphism of groups
$$
\tau\colon \{\text{circuits in }\Gamma^i(\overline{\mathcal Z},\tau)\text{ based at }\overline Z_1\}\to\{\text{circuits in }\Gamma^i(\mathcal Z)\text{ based at } Z_1\}
$$
defined in the following way.
Let $\overline{\mathcal C}=(\{\overline Z_1,\ldots,\overline Z_k\},\overline Z_{i,i+1})$ 
be a circuit in $\Gamma^i(\overline{\mathcal Z},\tau)$.
If $\tau(\overline Z_i)$ is a curve for every $i$, we set $\tau\overline{\mathcal C}=(\{\tau\overline Z_1,\ldots,\tau\overline Z_k\},\tau\overline Z_{i,i+1})$.
Otherwise let $i_j$ and $h_j$ be such that $h_j>0$ and $i_j+h_j+1= i_{j+1}$, and 
\begin{itemize}
\item for every $s=1,\ldots, h_j$ we have $\tau(\overline Z_{i_j+1})=\tau(\overline Z_{i_j+s})$ is a point in $Y$,
\item  $\tau(\overline Z_{i_j})$ and $\tau(\overline Z_{i_j+h_j+1})$ are curves in $Y$.
\end{itemize}
If $\tau(\overline Z_{i_j})=\tau(\overline Z_{i_j+h_j+1})$ we set 
$$(\overline Z_{i_j},\overline Z_{i_j+1},\ldots,\overline Z_{i_j+h_j+1})\mapsto (\tau(\overline Z_{i_j})).$$
If $\tau(\overline Z_{i_j})\neq\tau(\overline Z_{i_j+h_j+1})$ we set
$$(\overline Z_{i_j},\overline Z_{i_j+1},\ldots,\overline Z_{i_j+h_j+1})\mapsto (\tau(\overline Z_{i_j}),\tau(\overline Z_{i_j+h_j+1}))$$
with the edge $\tau(\overline Z_{i_j+1})$ between $\tau(\overline Z_{i_j})$ and $\tau(\overline Z_{i_j+h_j+1})$.

\end{const}

\begin{lem}\label{lem:mapC}
Notation as in Construction \ref{const:mapC}. 
The map of Construction \ref{const:mapC} respects the homotopy of loops and defines thus a homomorphism of groups $\tau\colon\pi_1(\Gamma^i(\overline{\mathcal Z}, \tau),\overline Z_1)\to\pi_1(\Gamma^i(Z), Z_1)$.
\end{lem}
\begin{proof}
It is enough to prove that the two circuits $$\overline{\mathcal C}=(\{\overline Z_1,\ldots,\overline Z_k\},\overline Z_{i,i+1})\text{ and }$$ $$\overline{\mathcal C}'=(\{\overline Z_1,\ldots,\overline Z_j,\overline Z,\overline Z_j\ldots ,\overline Z_k\},\overline Z_{1,2},\ldots,\overline Z_{j-1,j},\overline Z',\overline Z',\overline Z_{j,j+1},\ldots,\overline Z_{k,1} )$$
have homotopically equivalent images.
If  $\tau\overline Z_j$ and $\tau\overline Z$ are curves, then it is clear.
If $\tau\overline Z$ is a point, then the path
$(\overline Z_j,\overline Z,\overline Z_j)$ has the same image as $(\overline Z_j)$.
If $\tau\overline Z$ is a curve and $\tau\overline Z_j$ a point, then let $h<j<k$ be such that  $\tau\overline Z_h$ and $\tau\overline Z_k$ are curves and $\tau\overline Z_i$ is a point for every $h<i<k$.

If $\tau\overline Z_h=\tau\overline Z_k$, then 
$$(\overline Z_h,\ldots,\overline Z_k)\mapsto (\tau\overline Z_h).$$
If moreover $\tau\overline Z_h=\tau\overline Z$, then 
$$(\overline Z_h,\ldots,\overline Z_j,\overline Z,\overline Z_j\ldots ,\overline Z_k)\mapsto (\tau\overline Z_h).$$
If $\tau\overline Z_h\neq\tau\overline Z$, then 
$$(\overline Z_h,\ldots,\overline Z_j,\overline Z,\overline Z_j\ldots ,\overline Z_k)\mapsto (\tau\overline Z_h,\tau\overline Z,\tau\overline Z_h).$$
In both cases we get homothopically equivalent circuits.

\medskip

If $\tau\overline Z_h\neq\tau\overline Z_k$, then 
$$(\overline Z_h,\ldots,\overline Z_k)\mapsto (\tau\overline Z_h,\tau\overline Z_k).$$
Then either $\tau\overline Z_h=\tau\overline Z$ or $\tau\overline Z_k=\tau\overline Z$, and in both cases
$$(\overline Z_h,\ldots,\overline Z_j,\overline Z,\overline Z_j\ldots ,\overline Z_k)\mapsto (\tau\overline Z_h,\tau\overline Z_k).$$

\end{proof}

\subsection{Trivial line bundles and pullbacks}

We prove in this subsection that, if we have a generically finite morphism between two immersed simple normal crossings varieties, then a line bundle is trivial on the first variety
if and only if its pullback is trivial on the second.

\begin{lem}\label{grothC}
Let $\mathcal Z$ be a connected simple normal crossings variety of dimension at least 1 and let $\mathcal L$ be a line bundle on $\mathcal Z$ which is trivial for every irreducible component of $\mathcal Z$.
Then there is a simple normal crossings curve $\mathcal K\subseteq \mathcal Z$
such that the restriction of $\mathcal L$ to $\mathcal K$ is trivial if and only if $\mathcal L$ is trivial.
\end{lem}
\begin{proof}
If $\mathcal L$ is trivial, then for every curve  $\mathcal K\subseteq \mathcal Z$,
 the restriction of $\mathcal L$ to $\mathcal K$ is trivial.
 
 For the other implication, we proceed by induction on $$\dim \mathcal Z=\max\{\dim Z\vert\; Z\text{ irreducible component of }\mathcal Z\}.$$
 If $\dim \mathcal Z=1$, then we set $\mathcal K= \mathcal Z$.
 We assume now the existence of such a curve for  connected simple normal crossings varieties of dimension $k-1$. Let $\mathcal Z$ be a connected simple normal crossings variety of dimension
$k$. Let $\mathcal Z^{(k)}$ be the union of all the irreducible component of dimension $k$ and let $\mathcal Z_{k-1}$ be the union of all the irreducible component of dimension at most $k-1$.
Let $A$ be a section of a very ample divisor on   $\mathcal Z^{(k)}$ such that $H^1(\mathcal Z^{(k)},\mathcal O(-A))=\{0\}$ and $H^1(\mathcal Z^{(k)},\mathcal L(-A))=\{0\}$. 
In particular, for every connected component $Z$ of $\mathcal Z^{(k)}$ the intersection $A\cap Z$ is connected.
Assume moreover that $A\supseteq \mathcal Z^{(k)}\cap \mathcal Z^{(1)}$ and that $A\cup \mathcal Z_{k-1}$ is a simple normal crossings variety.

We set $\mathcal W=A\cup \mathcal Z_{k-1}$. 
For every irreducible component $Z$ of $\mathcal Z_{k-1}$ such that $Z\cap \mathcal Z^{(k)}\neq\emptyset$, we have $Z\cap A\neq\emptyset$.
Indeed, if  $\dim Z\cap \mathcal Z^{(k)}\geq 1$, it is true because $A$ is ample. If  $\dim Z\cap \mathcal Z^{(k)}=0$, then it is true by construction of $A$.

Then $\mathcal W$ is a connected  simple normal crossings variety of dimension $k-1$. In order to conclude, it is enough to prove that 
if the restriction of $\mathcal L$ to $\mathcal W$ is trivial, then $\mathcal L$ is trivial.

If the restriction of $\mathcal L$ to $\mathcal W$ is trivial, then there is a section $\bar s\in H^0(\mathcal W, \mathcal L)\setminus\{0\}$.
As $H^1(\mathcal Z^{(k)},\mathcal L(-A))=\{0\}$, there is a section $s^k\in H^0(\mathcal Z^{(k)},\mathcal L)$ such that $s^k\vert_{A}=\bar s\vert_A$.

We want to show that $s^k$ and $\bar s\vert_{\mathcal Z_{k-1}}$ glue to a section of $\mathcal Z$.
This happens if and only if $(s^k,\bar s\vert_{\mathcal Z_{k-1}})$ is in the kernel of 
$$
\begin{array}{rcl}
\alpha\colon H^0(\mathcal Z^{(k)},\mathcal L)\oplus H^0(\mathcal Z_{k-1},\mathcal L)&\to& H^0(\mathcal Z^{(k)}\cap \mathcal Z_{k-1},\mathcal L)\\
(s_1,s_2)&\mapsto& s_1-s_2.
\end{array}
$$

We have a commutative diagram

$$
\xymatrix{
 H^0(\mathcal Z^{(k)},\mathcal L)\oplus H^0(\mathcal Z_{k-1},\mathcal L)\ar[d]_{\wr}\ar[r]^{\alpha}& H^0(\mathcal Z^{(k)}\cap \mathcal Z_{k-1},\mathcal L)\ar@{^{(}->}[d]\\
 H^0(A,\mathcal L)\oplus H^0(\mathcal Z_{k-1},\mathcal L)\ar[r]^{\beta}& H^0(A\cap \mathcal Z_{k-1},\mathcal L)
}
$$

Since $(\bar s\vert_A, \bar s\vert_{\mathcal Z_{k-1}})$ is in the kernel of $\beta$,
it follows that $(s^k,\bar s\vert_{\mathcal Z_{k-1}})$ is in the kernel of $\alpha$.

\end{proof}

\begin{lem}\label{lem:tors3}
Let $\varepsilon\colon \overline Y\to Y$ be a generically finite map of normal projective varieties and let $\mathcal Z$ be a connected simple normal crossings subvariety of $Y$.
Assume that the preimage $\overline{\mathcal Z}$ of $\mathcal Z$ under $\varepsilon$ is a simple normal crossings variety.
Let $\mathcal L$ be a line bundle on $Y$ such that $\mathcal L\vert_Z\sim \mathcal O_Z$ for every irreducible component $Z$ of $\mathcal Z$.
Then $\mathcal L\vert_{\mathcal Z}$ is torsion if and only if 
$\varepsilon^*\mathcal L\vert_{\overline{\mathcal Z}}$ is torsion.
\end{lem}
\begin{proof}
If $\mathcal L\vert_{\mathcal Z}\sim \mathcal O_{\mathcal Z}$, then $\varepsilon^*\mathcal L\vert_{\overline{\mathcal Z}}\sim \mathcal O_{\overline{\mathcal Z}}$.

Conversely, we assume that  $\tau^*\mathcal L\vert_{\overline{\mathcal Z}}\sim\mathcal O_{\overline{\mathcal Z}}$.
By Lemma \ref{grothC} there is a semistable curve $\mathcal K\subseteq \mathcal Z$ such that  the restriction of $\mathcal L$ to $\mathcal K$ is trivial if and only if $\mathcal L$ is trivial.

\begin{cla}\label{cla:tors3}
There is a semistable curve $\overline{\mathcal K}\subseteq \overline Y$ such that
$\tau \overline{\mathcal K}=\mathcal K$ and the image of the homomorphism $\tau\colon\pi_1(\Gamma^i(\overline{\mathcal K}, \tau),\overline K_1)\to\pi_1(\Gamma^i(\mathcal K), K_1)$ 
has finite index in $\pi_1(\Gamma^i(\mathcal K), K_1)$.
\end{cla}

Assuming the claim, we conclude the proof.

Fix an irreducible component $\overline K_1$ of $\overline{\mathcal K}$ such that  $K_1=\tau\overline K_1$ is a curve.
We want to prove that there is a positive integer $h$ such that for every circuit $\mathcal C$ in $\Gamma^i(\mathcal K)$ based in $K_1$,  morphism $\Phi_{\mathcal L,\mathcal C}^h$ is the identity. As 
$\pi_1(\Gamma^i(\mathcal K),K_1)$ is finitely generated, the result will follow from Lemma \ref{lem:tors1}.

\noindent By Claim \ref{cla:tors3}, the group $\tau\left( \pi_1(\Gamma^i(\overline{\mathcal K},\tau),\overline K_1)\right)$ is a subgroup of $\pi_1(\Gamma^i(\mathcal K),K_1)$
of finite index $k$.
Therefore, there exists a circuit $\overline{\mathcal C}$ in $\Gamma^i(\overline{\mathcal K})$  such that, if we denote by $\mathcal C^k$ the concatenation of $\mathcal C$ with itself $k$ times,
the circuits $\mathcal C^k$ and $\tau\overline{\mathcal C}$
are homotopically equivalent.
By Remark \ref{rem:tors}, it is enough to prove that   $\Phi_{\mathcal L,\tau\overline{\mathcal C}}=\Phi_{\mathcal L,\mathcal C}^{k}$ is the identity.
Let  $K_i,K_{i+1}$ be curves in $\tau\overline{\mathcal C}$ with the edge $p_i$ between them. Let $\overline K_i,\overline K_{i,j},\overline K_{i+1}$ be curves in $\overline{\mathcal C}$ with $K_h=\tau\overline K_h$ for $h=i,i+1$ and $\tau\overline K_{i,j}=p_i$ for $j=1,\ldots,\ell_i$, 
where for every $j=0,\ldots,\ell_i$ we denote by $\bar p_{i,j}$ the edge between $\overline K_{i,j}$ and $\overline K_{i,j+1}$, with $\overline K_{i,0}=\overline K_i$ and $\overline K_{i,\ell_i+1}=\overline K_{i+1}$.
We have commutative diagrams 
$$
\xymatrix{
H^0(K_i,\mathcal L)\ar[rr]\ar[d]_{\wr}&&H^0(p_i,\mathcal L)\ar[d]_{\wr}\\
H^0(\overline K_i,\tau^*\mathcal L)\ar[r]&H^0(\bar p_{i,0},\tau^*\mathcal L)&H^0(\cup_j\overline K_{i,j},\tau^*\mathcal L)\ar[l]
}
$$

and

$$
\xymatrix{
H^0(p_i,\mathcal L)\ar[d]_{\wr}&&H^0(K_{i+1},\mathcal L)\ar[d]_{\wr}\ar[ll]\\
H^0(\cup_j\overline K_{i,j},\tau^*\mathcal L)\ar[r]&H^0(\bar p_{i,\ell_i},\tau^*\mathcal L)&H^0(\overline K_{i+1},\tau^*\mathcal L)\ar[l]
}
$$
where the horizontal arrows are the restriction isomorphisms and the vertical arrows are isomorphisms induced by the pullback by $\tau$.

Les $s$ be a global section for $\tau^*\mathcal L\vert_{\overline{\mathcal Z}}$.
Then, with the identifications of the previous diagram, the linear map $\Phi_{\mathcal L,\tau\overline{\mathcal C}}$ is the multiplication by $s/s=1$, therefore it is the identity.

Then $\Phi_{\mathcal L,\mathcal C}^{k}$ is the identity,  proving the statement.

\bigskip

We are left with the proof of Claim \ref{cla:tors3}.

Let $$
\begin{array}{l}
\mathcal K_f=\{K\subseteq \mathcal K\text{ irreducible component}\vert\; K\not\subseteq\tau Exc(\tau)\}\\
\mathcal K_e=\{K\subseteq \mathcal K\text{ irreducible component}\vert\; K\subseteq\tau Exc(\tau)\}.
\end{array}
$$
For $K\in \mathcal K_f$, let $\widetilde K$ be its strict transform.\\
Let $K\in \mathcal K_e$. For every irreducible component $Z\subseteq \tau^{-1}K$ surjecting onto $K$, let $H_i$
be hyperplane sections such that $K_Z=Z\cap\bigcap H_i$ is a reduced curve.
We can moreover find the $H_i$ such that if $\cup Z_i$ is connected, then $\widetilde K=\cup_Z K_Z$ is connected.\\
If $p\in \mathcal K^{sing}\cap \tau Exc(\tau)$, and $\{p\}=K_1\cap K_2$, for every irreducible component $Z\subseteq\tau^{-1}p$ let $H_i$
be hyperplane sections such that $K_Z=Z\cap\bigcap H_i$ is a reduced curve 
and has the following property:
if $\widetilde K_1, \widetilde K_2$ are such that $\tau \widetilde K_i=K_i$, then $ \widetilde K_i\cap \tau^{-1}p\subseteq K_p$.
We can moreover find the $H_i$ such that if $\cup Z_i$ is connected, then the union $K_p=\cup_Z K_Z$ is connected.\\
Finally, we set $$\overline{\mathcal K}=\cup\{\widetilde K\vert \; K\in\mathcal K_f\cup \mathcal K_e \}\cup\{K_p\vert \;  p\in \mathcal K^{sing}\cap \tau Exc(\tau)\}.$$
By the generality of $\mathcal K$, we can assume that $\overline{\mathcal K}$ is a simple normal crossings curve.

We want to prove now that $\tau\colon\pi_1(\Gamma^i(\overline{\mathcal K}, \tau),\overline K_1)\to\pi_1(\Gamma^i(\mathcal K), K_1)$ 
has finite index in $\pi_1(\Gamma^i(\mathcal K), K_1)$.

Let $\mathcal C=(K_1,\ldots,K_{\ell}, p_i)$ be a circuit in $\Gamma^i(\mathcal K)$.
Let $\overline K_1\subseteq \overline{\mathcal K}$ be such that $\tau \overline K_1=K_1$.
Let $N$ be the number of curves in $\overline{\mathcal K}$ surjecting onto $K_1$.
We construct a circuit $\overline{\mathcal C}$ in $\Gamma^i(\overline{\mathcal K}, \tau)$
such that $\tau \overline{\mathcal C}=m\mathcal C$ in the group of circuits based on $K_1$ with $m$ dividing $N!$.

\smallskip

We assume now that we have $\overline K_i$ $\;$ for $i=1,\ldots,r+q\ell$, and $\overline K_{i,j}$ for  $i=1,\ldots,r-1+q\ell$ and $j=1,\ldots,\ell_i$ 
and edges $\bar q_{i,j}\in \overline K_{i,j}\cap \overline K_{i,j+1}$ for $j=0,\ldots,\ell_i$
such that 
$\tau\overline K_i=K_{\bar i}$, where $\bar i$ is the remainder of the euclidean division of $i$ by $\ell$, and $\tau\overline K_{i,j}=p_i$.

If $q_r\not\in\tau Exc(\tau)$, then we let $\overline K_{r+1+q\ell}$ be a curve such that 
$$\tau\overline K_{r+1+q\ell}=K_{r+1}\text{ and }\tau^{-1}K_{r,r+1}\cap \overline K_{r+1+q\ell}\cap \overline K_{r+q\ell}\neq\emptyset.$$ 
We set $\bar q_{r+q\ell,0}$ as a point in  $\tau^{-1}K_{r,r+1}\cap \overline K_{r+1+q\ell}\cap \overline K_{r+q\ell}$.

If $q_r\in\tau Exc(\tau)$, then let $\overline K'_{r+1}$
such that $\tau\overline K'_{r+1}= K_{r+1}$ and $\overline K'_{r+1}$ meets a connected component of $ \tau^{-1}q_r$ meeting $\overline K_{r+q\ell}$.\\
Let $\overline K_{r+q\ell,j}$ be such that 
\begin{itemize}
 \item $\overline K_{r+q\ell,j}\cap \overline K_{r+q\ell,j+1}\neq\emptyset$,
 \item $\overline K_{r+q\ell}\cap \overline K_{r+q\ell,1}\neq\emptyset$, and
 \item $\overline K_{r+q\ell,\ell_{r+q\ell}}\cap \overline K'_{r+1}\neq\emptyset$.
\end{itemize}

We set $\bar p_{r+q\ell,0}\in\overline K_{r+q\ell}\cap \overline K_{r+q\ell,1}$, $\bar p_{r+q\ell,j}\in\overline K_{r+q\ell,j}\cap \overline K_{r+q\ell,j+1}$ and $\bar p_{r+q\ell,\ell_i}\in\overline K_{r+q\ell,\ell_{r+q\ell}}\cap \overline K'_{r+1}$.
Finally, we set $\overline K_{r+1+q\ell}=\overline K'_{r+1}$.

Then there are $q_1<q_2$ with $q_2-q_1\leq N$ such that $\overline K_{1+q_1\ell}=\overline K_{1+q_2\ell}$.
Then we set $\gamma=(\overline K_1,\ldots,\overline K_{q_1\ell})$ 
and $\overline{\mathcal C}=\gamma*(\overline K_{q_1\ell},\ldots,\overline K_{q_1\ell})*\gamma^{-1}$.
We have $\tau \overline{\mathcal C}=(q_2-q_1)\mathcal C=\underbrace{\mathcal C*\ldots*\mathcal C}_{q_2-q_1\ \text{times}}$.
\end{proof}

%% file: dlt4lctS.tex
\section{Restriction of the moduli part to log canonical centres}\label{Restr}

The goal of this section is to describe the restriction of the moduli part to a log canonical centre of $(Y,\Sigma_f)$.
Part of the results can be seen as a higher codimensional version of \cite[Proposition 4.2]{FL19}.
We refer to \cite{Hu20} for similar results.

\begin{dfn}(Definition 3.12 \cite{FL19})\label{dfn:acceptable}
Let $f\colon (X,\Delta)\to Y$ be an lc-trivial (respectively klt-trivial) fibration. Then $f$ is \emph{acceptable} if there exists another lc-trivial (res\-pec\-ti\-ve\-ly klt-trivial) fibration $\bar f\colon (\overline X,\overline \Delta)\to Y$ such that $\overline \Delta$ is effective on the generic fibre of $\bar f$, and a birational morphism $\mu\colon X\to\overline X$ such that $f=\bar f\circ\mu$ and $K_X+\Delta\sim_\Q\mu^*(K_{\overline X}+\overline \Delta)$. Note that then the horizontal part of $\Delta^{<0}$ with respect to $f$ is $\mu$-exceptional. Note also that any birational base change of an acceptable lc-trivial (respectively klt-trivial) fibration is again an acceptable lc-trivial (respectively klt-trivial) fibration.
$$
\xymatrix{
(X,\Delta)\ar[dr]_f \ar[r]^\mu & (\overline X,\overline \Delta)\ar[d]^{\bar f}\\
& Y
}
$$
\end{dfn}

\begin{dfn}(Definition 3.10 \cite{FL19})
Let  $f\colon(X, \Delta)\rightarrow Y$ be an lc-trivial fibration, where $(X,\Delta)$ is log smooth and $Y$ is smooth. Fix a prime divisor $T$ on $Y$. An \emph{$(f,T)$-bad divisor} is any reduced divisor $\Sigma_{f,T}$ on $Y$ which contains:
\begin{enumerate}
\item[(a)] the locus of critical values of $f$,
\item[(b)] the closed set $f(\Supp\Delta_v)\subseteq Y$, and
\item[(c)] the set $\Supp B_f\cup T$.
\end{enumerate}
\end{dfn}

The next result is a corollary of \cite[Proposition 4.2]{FL19}.

\begin{pro}\label{pro:FL}
Let  $f\colon(X, \Delta)\rightarrow Y$ be an acceptable klt-trivial fibration, where $(X,\Delta)$ is a log smooth log canonical pair and $Y$ is a smooth Ambro model for $f$. 
Assume that there exists an $(f,0)$-bad divisor $\Sigma_f\subseteq Y$ which has simple normal crossings, and such that the divisor $\Delta+f^*\Sigma_f$ has simple normal crossings support. Denote 
$$\Delta_X=\Delta+\sum_{\Gamma\subseteq \Sigma_f}\gamma_{\Gamma}f^*\Gamma,$$
where $\gamma_\Gamma$ are the generic log canonical thresholds with respect to $f$ as in Definition \ref{dfn:cbf}. 
Let $Z=T_1\cap\ldots\cap T_k$ be a log canonical centre of $(Y,\Sigma_f)$.
Denote 
$$\Xi_Z:=(\Sigma_f-\sum T_i)|_Z.$$ 
Let $S$ be a minimal log canonical centre of $(X, \Delta_X)$ over $Z$, which exists by \cite[Lemma 4.1]{FL19} 
Let 
$$f|_S\colon S\overset{h}{\longrightarrow} Z'\overset{\tau}{\longrightarrow} Z$$
be the Stein factorisation, and let $R$ denote the ramification divisor of $\tau$ on $Z'$. 
Then:
\begin{enumerate}
\item[(i)] if $K_S+\Delta_S=(K_X+ \Delta_X)|_S$, then $h\colon (S, \Delta_S)\rightarrow Z'$ is a klt-trivial fibration with $B_h\geq0$, 
\item[(ii)] $\tau^*(M_f|_Z)\sim_\Q M_h+R'+E$, where $M_f$ is chosen so that $Z\not\subseteq M_f$ and
$$R'=\sum\limits_{\Gamma\not\subseteq \tau^{-1}(\Xi_Z)}(\mult_\Gamma R)\cdot \Gamma\quad\text{and}\quad E=\sum\limits_{\Gamma\not\subseteq\tau^{-1}(\Xi_Z)}(\mult_\Gamma B_h)\cdot\Gamma.$$

\end{enumerate}
\end{pro}
\begin{proof}
The proof follows the same line as \cite[Proposition 4.2]{FL19}. In particular, Steps 1-5 are the same:
we find a birational map $\rho\colon (X,\Delta_X)\dasharrow (W,\Delta_W)$ over $Y$ such that, if $\psi\colon(W,\Delta_W)\to Y$ is the induced lc-trivial fibration,
then $(\phi^*\Sigma_f)_{\red}\leq\Delta_{W,v}$.
After replacing $T$ with $Z$ in Step 5, the fibration $h\colon (S, \Delta_S)\rightarrow Z'$ is klt-trivial.

\medskip

\emph{Step 6.}
Let $T_1,\ldots,T_h$ be components of $\Sigma_f$ such that $Z=T_1\cap\ldots\cap T_h$. 
By equation \cite[(13)]{FL19} every component $D_1$ of $\psi^*T_1$ which dominates $T_1$ has coefficient $1$ in $\Delta_W$. Denote $\Delta_{D_1}:=(\Delta_W-D_1)|_{D_1}$, 
so that the Stein factorisation of $\psi|_{D_1}\colon (D_1,\Delta_{D_1})\to T_1$ gives an lc-trivial fibration. Let $\Xi_{T_1}=(\Sigma_f-T_1)\vert_{T_1}$ and let $P$ be a component 
of $(\psi|_{D_1})^*\Xi_{T_1}$. 
Since $(\psi|_{D_1})^*\Xi_{T_1}=(\psi^*\Sigma_{f}-\psi^*T_1)|_{D_1}$, and each component of $\psi^*\Sigma_{f}$ is a component of $\Delta_W^{=1}$ by \cite[(12) and (13)]{FL19}, 
this implies that $P$ is a component of $(\Delta_W^{=1}-D_1)|_{D_1}=\Delta_{D_1}^{=1}$. In other words,
\begin{equation*}
\big((\psi|_{D_1})^*\Xi_{T_1}\big)_{\red}\leq\Delta_{D_1}^{=1}.
\end{equation*}
Assume that for $i>1$ there are components $D_1,\ldots,D_i$ such that $\phi(D_j)=T_j$ and $\big((\psi|_{D_1\cap\ldots\cap D_i})^*\Xi_{T_1\cap\ldots\cap T_i}\big)_{\red}\leq\Delta_{D_1\cap\ldots\cap D_i}^{=1}$,
where $\Xi_{T_1\cap\ldots\cap T_i}=(\Sigma_f-T_1-\ldots-T_i)\vert_{T_1\cap\ldots\cap T_i}$ and $\Delta_{D_1\cap\ldots\cap D_i}^{=1}=(\Delta_W^{=1}-D_1-\ldots- D_i)\vert_{D_1\cap\ldots\cap D_i}$.

There is a component  $D_{i+1}$ of $\psi^*T_{i+1}$ which has coefficient $1$ in $\Delta_W$.
 Denote $\Delta_{D_1\cap\ldots\cap D_{i+1}}=(\Delta_W-D_1-\ldots- D_{i+1})\vert_{D_1\cap\ldots\cap D_{i+1}}$, 
so that the Stein factorisation of $\psi|_{D_1\cap\ldots\cap D_{i+1}}\colon ({D_1\cap\ldots\cap D_{i+1}},\Delta_{D_1\cap\ldots\cap D_{i+1}})\to T_1\cap\ldots\cap T_{i+1}$ gives an lc-trivial fibration. 
Let $\Xi_{T_1\cap\ldots\cap T_{i+1}}=(\Sigma_f-T_1-\ldots-T_{i+1})\vert_{T_1\cap\ldots\cap T_{i+1}}$ and let $P$ be a component 
of $(\psi|_{D_1\cap\ldots\cap D_{i+1}})^*\Xi_{T_1\cap\ldots\cap T_{i+1}}$. 
As before,
\begin{equation*}
\big((\psi|_{D_1\cap\ldots\cap D_{i+1}})^*\Xi_{T_1\cap\ldots\cap T_{i+1}}\big)_{\red}\leq\Delta_{D_1\cap\ldots\cap D_{i+1}}^{=1}.
\end{equation*}

We proved by induction that there are $D_1,\ldots,D_h$ such that 
\begin{equation}\label{eq:1002}
\big((\psi|_{D_1\cap\ldots\cap D_h})^*\Xi_{Z}\big)_{\red}\leq\Delta_{D_1\cap\ldots\cap D_h}^{=1}.
\end{equation}

Now, by \cite[Proposition 2.6]{FL19} there are components $D_1,\ldots,D_h$ of $\Delta_W$ and $S_1,\dots,S_k$ of $\Delta_D^{=1}$, where $D=D_1\cap\ldots\cap D_h$, such that $S_W$ is a component of $S_1\cap\dots\cap S_k$, and note that the $S_i$ dominate $Z$. 
This and \eqref{eq:1002} imply
\begin{equation}\label{eq:1003}
\big((\psi|_{D})^*\Xi_Z\big)_{\red}\leq\Delta_{D}^{=1}-S_1-\dots-S_k,
\end{equation}
hence
$$\big((\psi|_{S_W})^*\Xi_Z\big)_{\red}\leq(\Delta_{D}^{=1}-S_1-\dots-S_k)|_{S_W}\leq \Delta_{S_W}^{=1}.$$

Thus, for every prime divisor $P\subseteq \Supp \tau^*\Xi_Z$, the generic log-canonical threshold $\gamma_P$ of $(S_W,\Delta_{S_W})$ with respect to $h_W^*P$ is zero. If we define
$$E:=\sum\limits_{\Gamma\not\subseteq\tau^{-1}(\Xi_Z)}(\mult_\Gamma B_{h_W})\cdot\Gamma=\sum\limits_{\Gamma\not\subseteq\tau^{-1}(\Xi_Z)}(\mult_\Gamma B_h)\cdot\Gamma,$$
where the second equality follows from \cite[(17)]{FL19}, then
\begin{equation}\label{eq:ef}
B_{h_W}=(\tau^*\Xi_Z)_\red+E.
\end{equation}

\medskip

Finally, Steps 7 is the same after replacing $T$ with $Z$.
 \end{proof}

\begin{pro}\label{pro:mmp2fg}
Let $f\colon( X, \Delta)\rightarrow  Y$ be an acceptable klt-trivial fibration.
Assume that $Y$ is an Ambro model for $f$ and that there exists a simple normal crossings divisor $R$ on $Y$ such that the support of the divisor 
$\Delta + f^{-1}\Sigma_f$ has simple normal crossings.
Assume that $f$ is semistable.
Set $$\Delta_X=\Delta+\sum_{\Gamma\subseteq \Sigma_f}\gamma_{\Gamma}f^*\Gamma$$
where $\gamma_\Gamma$ are the generic log canonical thresholds with respect to the klt-fibration $f$ as in Definition \ref{dfn:cbf}.
Then there exists a birational map $\rho\colon X\dasharrow W$ and a fibration $\psi\colon W\rightarrow Y$ such that:
\begin{enumerate}[label=(\alph*)]
\item \label{stepA} the pair $(W, \Delta_W)$ is $\Q$-factorial dlt, where $\Delta_W:=\rho_*\Delta_X$, and $\Delta_W\geq 0$;
\item \label{stepB} $\psi\colon (W,\Delta_W)\rightarrow Y$ is a klt-trivial fibration;
\item \label{stepC} $\rho\colon (X,\Delta_X)\dasharrow (W,\Delta_W)$ is crepant birational;
\item \label{stepD} the discriminant of $\psi$ is $\Sigma_f$ and the moduli part is $M_f$;
\item \label{stepE} $\Delta_{W,v}=\psi^*\Sigma_f$.
\end{enumerate}
Let $Z$ be a log canonical centre of $(Y,\Sigma_f)$ and let $S$ be a minimal log canonical centre of $(W, \Delta_W)$ over $Z$. Let $\psi|_S\colon S\overset{h}{\longrightarrow} Z'\overset{\tau}{\longrightarrow} Z$ be the Stein factorisation. 
\begin{enumerate}[label=(\roman*)]
\item \label{stepi} If $K_S+\Delta_S=(K_W+ \Delta_W)|_S$, then $h\colon (S, \Delta_S)\rightarrow Z'$ is a klt-trivial fibration.
\item \label{stepii} Assume that $\tau^*M_f\vert_Z=M_h$. Then $\Delta_{S,v}=h^*B_h$ and $B_h=(\tau^*\Xi_Z)_\red$.
\item \label{stepiii}  Let $Z$ be a component of $\Sigma_{f}$ such that $M_f\vert_Z\equiv 0$. Then $\tau^*M_f\vert_Z=M_{h}\sim_{\mathbb Q}0$ and $\Delta_{S,v}=h^*B_h$.
\item \label{stepiv} If either $\tau^*M_f\vert_Z=M_h$ or $M_f\vert_Z\equiv 0$, then $h$ has reduced fibres over an open set meeting all the irreducible components of $B_h$. 
\end{enumerate}
\end{pro}

\begin{proof}
\emph{Step 1.}
The existence of $\rho$ satisfying \ref{stepA}, \ref{stepB}, \ref{stepC}, \ref{stepD} follows from Steps 2 and 3 of the proof of \cite[Proposition 4.2]{FL19}. We have then
\begin{equation}\label{eq:6a}
K_W+\Delta_W\sim_\Q\psi^*(K_Y+\Sigma_{f}+M_f).
\end{equation}
The divisor $\Delta_{W,v}$ is reduced, and
by \cite[Proposition 4.2, (13)]{FL19}
 $\Delta_{W,v}=(\psi^*\Sigma_{f})_{red}$.
As for  \ref{stepE}, every component $D$ of $\Delta_{W,v}$ is a log canonical centre of $(W,\Delta_W)$.
By \cite[Lemma 2.8]{FL19} there is a centre $D_X$ of $(X,\Delta_X)$ such that $\rho$ induces a birational map $\rho\vert_{D_X}\colon D_X\dasharrow D$.
Therefore
$$\Delta_{W,v}=\sum_{i=1}^n D_i= \sum_{i=1}^n \rho_* D_{i,X}=\rho_* f^*\Sigma_{f} $$ 
the last equality following from the semistability of $f$.
Let $(p,q)\colon Z\to X\times W$ be a resolution of the indeterminacy of $\rho$. Then $\rho_* f^*\Sigma_{f}=q_* p^*f^*\Sigma_{f}=q_* q^*\psi^*\Sigma_{f}=\psi^*\Sigma_{f}$ proving \ref{stepE}.

\medskip

\emph{Step 2.}
The proof of \ref{stepi} follows the same lines as \cite[Proposition 4.2]{FL19}, which has slightly different hypotheses. We recall it here for completeness.

By restricting the equation \eqref{eq:6a} to $S$ we obtain
\begin{equation}\label{eq:8}
K_S+\Delta_S\sim_\Q(\psi |_S)^*(K_Z+\Xi_Z+M_f\vert_Z),
\end{equation}
where $\Xi_T=(\Sigma_{f}-T)|_T$. Thus $h$ is an lc-trivial fibration, and moreover, it is a klt-trivial fibration. Indeed, if there existed a log canonical centre $\Theta$ of $(S,\Delta_S)$ which dominated $T'$, then $\Theta$ would be a log canonical centre of $(W,\Delta_W)$ by \cite[Proposition 3.9.2]{Fuj07c}, which contradicts the minimality of $S$. This proves \ref{stepi}.

\medskip

\emph{Step 3.}
In order to show \ref{stepii} and \ref{stepiii}, denote by $M_h$ and $B_h$ the moduli part and the discriminant of $h$. From \eqref{eq:8} we have
\begin{equation}\label{eq:9}
\tau^*(K_{Z}+\Xi_Z+M_f\vert_Z)= K_{Z'}+B_h+M_h.
\end{equation}
 By \cite[Lemma 2.8]{FL19}, there is a centre $S_X$ of $(X,\Delta_X)$ such that $\rho$ induces a birational map $\rho\colon S_X\dasharrow S$. Moreover, if we define $\Delta_{S_X}$ by $K_{S_X}+\Delta_{S_X}=(K_X+\Delta_X)\vert_{S_X}$, by \ref{stepE} the restriction
$\rho\colon (S_X,\Delta_X)\dasharrow (S,\Delta_S)$ is crepant birational.

If $f\vert_{S_{X}}=\tau_{X}\circ h_{X}$ is the Stein factorisation, then we claim that $\tau_{X}=\tau$. Indeed, let $(p,q)\colon W\to S_{X}\times S$ be the resolution of indeterminacies of the birational map $\rho|_{S_{X}}\colon S_{X}\dashrightarrow S$. Both $p$ and $q$ have connected fibres by Zariski's main theorem, since $S_{X}$ and $S$ are normal. Then every curve contracted by $p$ is contracted by $h\circ q$, and thus $f|_{S_{X}}$ factors through $T'$ by the Rigidity lemma \cite[Lemma 1.15]{Deb01}. This proves the claim.

By \eqref{eq:ef} there exists an effective divisor $E$ such that
$$
B_{h}=(\tau^*\Xi_Z)_\red+E.
$$
Write the Hurwitz formula for $\tau$ as $K_{Z'}=\tau^*K_{Z}+R$. Then 
\begin{multline}\label{eq:10}
\tau^*(K_{Z}+\Xi_Z)
=K_{Z'}+B_h-E-R+\tau^*\Xi_Z-(\tau^*\Xi_Z)_\red.
\end{multline}
We notice moreover that
$$\tau^*\Xi_Z-(\tau^*\Xi_Z)_\red\leq R.$$

\medskip

\emph{Step 4.}
We assume that  $\tau^*M_f\vert_Z=M_h$ and we prove that $\Delta_{S,v}=(h^*B_h)_\red$.
Then \eqref{eq:8} becomes $\tau^*(K_{Z}+\Xi_Z)
=K_{Z'}+B_h$.
Equation \eqref{eq:10} implies that $-E-R+\tau^*\Xi_Z-(\tau^*\Xi_Z)_\red=0$.
In particular $E=0$ and 
\begin{equation}\label{eq:fff3}
B_h=(\tau^*\Xi_Z)_\red.
\end{equation}
Therefore, by \ref{stepE}, by the fact that $S$ is a minimal log canonical centre of $(W,\Delta_W)$ over $T$ and by \eqref{eq:fff3} we have
\begin{equation}\label{eq}
\Delta_{S,v}=(h^*B_h)_\red.
\end{equation}

\medskip

\emph{Step 5.}
We assume that  $M_f\vert_Z\equiv 0$ and we prove that $\Delta_{S,v}=(h^*B_h)_\red$ and $\tau^*M_f\vert_Z=M_{h}\sim_{\mathbb Q}0$.

Equations \eqref{eq:9} and \eqref{eq:10} imply that $\tau^*(M_f|_Z)\geq M_{Z'}$. Since $M_f|_Z\equiv 0$ and $M_h$ is pseudoeffective by Theorem \ref{nefness} and Remark \ref{rem:modpseff}, we 
get $\tau^*(M_f|_Z)= M_{Z'}$. In particular, $M_h\equiv 0$, hence $M_h\sim_\Q0$ by Theorem \ref{thm:torsionAmbro}.
Moreover, 
$\tau^*(K_{Z}+\Xi_Z)
=K_{Z'}+B_h$ and $E=0$, proving that 
$B_h=(\tau^*\Xi_Z)_\red$.
Therefore, by \ref{stepE}, by the fact that $S$ is a minimal log canonical centre of $(W,\Delta_W)$ over $T$ and by \eqref{eq:fff3} we have
\begin{equation}\label{eq:bis}
\Delta_{S,v}=(h^*B_h)_\red.
\end{equation}

\medskip

\emph{Step 6.} Assuming that $\Delta_{S,v}=(h^*B_h)_\red$, we prove that $\Delta_{S,v}=h^*B_h$.
By Remark \ref{rem:restred} the fibration $h_X$ has reduced fibres.

To prove \ref{stepii} we reason as in \ref{stepE}.
Let $D$ be an irreducible component of $\Delta_{S,v}$. Then $D$ is a log canonical centre of $(S,\Delta_S)$ and therefore of $(W,\Delta_W)$.
By \cite[Lemma 2.8]{FL19} there is a log canonical centre $D_X$ of $(X,\Delta_X)$
such that $\rho$ induces a birational map $D_X\dasharrow D$.

Then $$\Delta_{S,v}=\sum_{i=1}^n D_i=\sum_{i=1}^n (\rho\vert_{S_X})_* D_{i,X}=(\rho\vert_{S_X})_*(\Delta_{S_X,v})^{=1}=(\rho\vert_{S_X})_*h_X^*B_h=h^*B_h.$$

\medskip

\emph{Step 7.}
Finally, \ref{stepiv} follows directly from Step 6, as $h^*B_h$ is a reduced divisor.

\end{proof}



%% file: thm1.tex
\section{Finiteness of the equivalence relation for the moduli part}\label{thm1}

This section is devoted to the proof of the finiteness of the equivalence relation induced by $\mathcal O_Y(mM_f)$ on a connected divisor $\mathcal T$.

\begin{ass}\label{ass1} We consider the following set of assumptions on a triple $(f\colon(X,\Delta)\to Y,\mathcal T,\Sigma_f)$ or $(f,\mathcal T,\Sigma_f)$ for short.
\begin{enumerate}
\item $f\colon (X,\Delta)\to Y$ is  an acceptable klt-trivial fibration;
\item $\Sigma_f$ is a simple normal crossings divisor and is an    $(f,T)$-bad for every $T\subseteq\mathcal T$;
\item for every $T\subseteq\mathcal T$ the restriction $\mathcal O_T(mM_f)$ is semiample and we denote by $\phi_T$ the induced fibration;
\item $f$ is semistable.
\end{enumerate}
\end{ass}
 In particular by \cite[Proposition 8.4.9, Definition 8.3.6, Theorem 8.5.1]{Kol07} the base $Y$ is an Ambro model and $\mathcal T$ is simple normal crossing.
 
 \begin{thm}\label{thm:step1}
Let  $(f\colon(X,\Delta)\to Y,\mathcal T,\Sigma_f)$  be a triple satisfying Assumption \ref{ass1}.
 Let $m$ be a positive integer such that $mM_f$ is a Cartier divisor and let $\mathcal L=\mathcal O(mM_f)$.
 Assume Conjecture \ref{con:maxvar}.
Then the equivalence relation $\mathcal R_{\mathcal L}$ is finite.
\end{thm}

The following lemma is a higher-codimensional version of \cite[Proposition 4.4]{FL19} (see also \cite{Hu20}).

\begin{lem}\label{lem:descends}
Let $(f,\mathcal T,\Sigma_f)$ be a triple satisfying Assumption \ref{ass1}(1,2,3).
Let $\mathcal P,\overline{\mathcal P}$ be two sets of log-canonical centres of $\Sigma_f$
such that 
\begin{enumerate}
\item[(i)] if $P,Q\in \mathcal P$ (resp. $\overline P,\overline Q\in \overline{\mathcal P}$) then $P\subseteq Q$ implies $P=Q$ (resp. $\overline P\subseteq \overline Q$ implies $\overline P=\overline Q$).
\item[(ii)] for every $\overline P\in \overline{\mathcal P}$ there is $ P\in \mathcal P$ such that $\overline P\subseteq P$
\item[(iii)] whenever $\overline P\subseteq P\subseteq T$ we have $\phi_T(\overline P)=\phi_T(P)$.
\end{enumerate}
Let $\mathcal P\to Nklt(X,\Delta+\sum_{\Gamma\subseteq \Sigma_f}\gamma_{\Gamma}f^*\Gamma )$ be a function such that $P\mapsto S_X(P)$ and $S_X(P)$ is minimal over $P$.
For every pair $(P,\overline P)$ such that $\overline P\subseteq P$ let $R_X(P,\overline P)$ be a log-canonical centre of $(X,\Delta+\sum_{\Gamma\subseteq \Sigma_f}\gamma_{\Gamma}f^*\Gamma )$
minimal over $\overline P$ and such that $R_X(P,\overline P)\subseteq S_X(P)$.
Then there is a diagram
$$
\xymatrix{
X_0\ar[r]^{\eta}\ar[d]_{f'}& X\ar[d]^{f}\\
Y_0\ar[r]_{\varepsilon}& Y
}
$$
where $\varepsilon$ is a birational morphism with the following properties.
For every $P\in\mathcal P$ (resp. $\overline P\in\overline{\mathcal P}$)
 let $P_0$ (resp. $\overline P_0$) be the strict transform of $P$ and $S_0$ the strict transform of $S_X(P)$ (resp. $R_0$ of $R_X(P,\overline P)$).
Let $f_0\vert_{S_0}\colon S_0\overset{ h}{\longrightarrow} P'_0\overset{\tau}{\longrightarrow}P_0$ (resp. $f_0\vert_{R_0}\colon R_0\overset{ g}{\longrightarrow} \overline P'_0\overset{\sigma}{\longrightarrow}\overline P_0$) be the Stein factorisation.
Then the following hold:
\begin{enumerate}
\item $\varepsilon$ is an isomorphism at the generic point of every subvariety $P\in\mathcal P$
$\overline P\in\overline{\mathcal P}$;
\item $\varepsilon$ is an isomorphism at the generic point of $T\cap T'$ for every $T,T'\subseteq\mathcal T$;
\item $\eta$ is a desingularisation of the fibre product which is an isomorphism over $Y'\setminus\Exc(\varepsilon)$;
\item for every $P\in\mathcal P$ we have $M_{h}=\tau^*M_{f_0}\vert_{P_0}$ and $P_0'$ is an Ambro model;
\item for every $\overline P\in\overline{\mathcal P}$ we have $M_{g}=\sigma^*M_{f_0}\vert_{\overline P_0}$ and $\overline P_0'$ is an Ambro model;
\item $\varepsilon^{-1}\Sigma_f$ has simple normal crossings. 
\end{enumerate}

\end{lem}
\begin{proof}
We say that $P\in \mathcal P$ satisfies $(\star)$ if, denoting by $f\vert_{S_X}\colon S_X\overset{h_X}{\longrightarrow} P'\overset{\tau_X}{\longrightarrow} P$ the Stein factorisation, 
we have $M_{h_X}=\tau_X^*M_{f}\vert_{P}$ and $P'$ is an Ambro model.
We prove by induction on the cardinality of $$\mathcal P'=\{P\in \mathcal P\vert\; P\text{ does not satisfy }(\star)\}$$
that there is $\varepsilon$ satisfying (1-4) and (6).
If the cardinality of $\mathcal P'$ is zero, there is nothing to prove.

\medskip

Otherwise, we pick $P\in \mathcal P$.
By \cite[Proposition 3.9.2]{Fuj07c} there are $D_1,\ldots, D_{\ell}\subseteq \Supp \Delta^{=1}$ such that 
$S_X=D_1\cap\ldots\cap D_{\ell}$. 
We set $\Delta_{S_X}=(\Delta_{X_2}-\sum D_i)\vert_{S_2}$. 

Let $f\vert_{S_X}\colon S_X\overset{h_X}{\longrightarrow} P'\overset{\tau_X}{\longrightarrow} P$
be the Stein factorisation.
By Proposition \ref{pro:FL} the morphism $h_X$ is a klt-trivial fibration
and there is an effective divisor $E$ such that
$M_{h_X}=\tau_X^*M_{f}\vert_P-E$.
Let $P\subseteq T$,
let $C$ be a general curve in $P$ contained in a fibre of $\phi_T$ and let $\widetilde C$ be a curve in $P'$ such that $\tau_X(\widetilde C)=C$.
Then $$0\leq M_{h_X}\cdot \widetilde C =M_{f}\cdot C-E\cdot \widetilde C\leq 0.$$
Therefore $E$ is a vertical divisor with respect to $\phi_T\circ\tau_X$.
We call $\overline E$ the union of the components of $\tau_X(\Supp E)$ which are not of components of $T\cap T'$ for some $T,T'\subseteq\mathcal T$.

 We let $\varepsilon\colon Y_0\to Y$ be 
the composition of the blow up $\mu\colon Y_1\to Y$ of  $\overline E$
with a log resolution of $(Y_0,\mu^{-1}\Sigma_f)$ centered in the singular locus. 
 Let $X_0$ 
 be a normalisation of the main component of the base change followed by a desingularisation centered in the singular locus, with the natural map $f_0\colon X_0\to Y_0$. 
 Since $\tau_X(\Supp E)$ is vertical with respect to $\phi_T$, the divisor $\overline E$ satisfies the same property.
 Therefore, if $\overline P\subseteq P$, the morphism
$\varepsilon$ is an isomorphism on the generic point of $\overline P$ as this subvariety is such that $\phi_T(\overline P)=\phi_T(P)$.
If  $\overline P\not\subseteq P$ or $Q\in\mathcal P$ and $Q\neq P$, the morphism
$\varepsilon$ is obviously an isomorphism on the generic point of $\overline P$ or $Q$.
Moreover, it is an isomorphism at the generic point of the intersections $T\cap T'$.

Following the proof of \cite[Proposition 4.4]{FL19}, replacing $T$ with $P$ and \cite[Proposition 4.2(ii)]{FL19}
with Proposition \ref{pro:FL}(ii), we have that, if $S_0$ is the strict transform of $S_X$ in $X_0$,
$P_0$ is the strict transform of $P$ in $Y_0$ and $S_0\to P'_0\to P_0$ is the Stein factorisation, then
$\tau_0^*M_{f_0}=M_{h_0}$ and $P'_0$ is an Ambro model.

Let $Q\in\mathcal P$ satisfying property $(\star)$.
There is a diagram
$$
\xymatrix{
S(Q)_0\ar[d]_{h_0}\ar[r]^{\eta}&S_X(Q)\ar[d]^{h_X}\\
Q'_0\ar[r]^{\zeta}\ar[d]_{\tau_0}&Q'\ar[d]^{\tau_X}\\
Q_0\ar[r]_{\varepsilon}&Q
}
$$
By applying $\zeta^*$ to $\tau^*M_f\vert_Q=M_{h_X}$ we get
$$\tau_0^*M_{f_0}\vert_{Q_0}=\tau_0^*\varepsilon^*M_f\vert_Q=\zeta^*\tau_X^*M_f\vert_Q=\zeta^*M_h=M_{h_0}.$$
Since $Q'$ is an Ambro model and $\zeta$ is birational, $Q'_0$ is one too.

Let $\mathcal P_0$ be the set of strict transforms of elements of $\mathcal P$. 
Then the cardinality of the set $\{P\in\mathcal P_0\vert\; P\text{ does not satisfy }(\star) \}$ is at most $|\mathcal P'|-1$ and 
we conclude by induction.

As for (5), the proof is completely analogous. 

\end{proof}

\begin{proof}[Proof of Theorem \ref{thm:step1}]
Assume that $\mathcal R_{\mathcal L}$ is not a finite equivalence relation.
By Proposition \ref{pro:profinite} there is $\mathcal Z\subseteq \sqcup V$ and a subrelation $\mathcal R'\subseteq\mathcal R_{\mathcal L}$
such that $\mathcal Z$ is  $\mathcal R'$-invariant  $\mathcal R'\vert_{\mathcal Z}$ is equidimensional and the set of infinite equivalence classes is dense in $\mathcal Z$.

\smallskip

Let $P\subseteq \phi^{-1}\mathcal Z$ be an irreducible component surjecting onto an irreducible component of $\mathcal Z$.
Then $\mathcal L\vert_P$ is not big. Indeed, if it were big, then $\phi\vert_P$ would be a birational morphism and generically on $\phi(P)$ the induced equivalence relation would be the gluing $\sqcup T\to\mathcal T$, thus finite.

\medskip

\emph{Step 1.} We can assume that every irreducible component of $\phi^{-1}\mathcal Z$ is a log canonical centre of $(Y,\Sigma_f)$.

Indeed, let $\phi^{-1}\mathcal Z=W_1\cup\ldots\cup W_k$ be the decomposition into irreducible components.
We can assume that there is $h$ such that $W_i$ is a centre of  $(Y,\Sigma_f)$ for $i>h$. 
Let $\delta\colon Y_1\to Y$ be such that $\delta^{-1}(W_1\cup\ldots\cup W_h\cup\Sigma_f)$ has simple normal crossings.
The morphism $\delta$ is an isomorphism over the generic point of $T$ and $T\cap T'$ for every $T,T'\subseteq\mathcal T$.
Let $\eta'\colon X'\to X$ be the natural morphism followed by a desingularisation of the main component of $X\times_Y Y_1$ and set $K_{X'}+\Delta'=\eta^{\prime*}(K_X+\Delta)$.
Let $\eta_1\colon X_1\to X'$ be a log resolution of $(X',\Delta')$. We can assume that the birational morphism $X_1\to X$ is an isomorphism on $Y\setminus \delta Exc(\delta)$.
Let $f_1\colon X_1\to Y_1$ be the natural morphism and we define $\Delta_1$ by $K_{X_1}+\Delta_1=\eta_1^*(K_{X'}+\Delta')$.

We apply Theorem \ref{thm:alt} to $X_1,Y_1$, with $Z=\Supp \Delta_{X_1}\cup f_1^{-1}\delta^{-1}\Sigma_f$.
We get $a,b\colon(\widetilde X, \widetilde Y)\to (X_1,Y_1)$
\'etale outside $Exc(\delta)$.
Let $\Sigma_{\tilde f}=b^{-1}\delta^{-1}\Sigma_f$. Then $\tilde f^{-1}\Sigma_{\tilde f}\cup  a^{-1}\Supp\Delta_1$ 
has simple normal crossings support.
Define $\widetilde \Delta$ by $K_{\widetilde X}+\widetilde \Delta=a^*(K_{X_1}+\Delta_1)$
and $\eta=\varepsilon_X\circ a$. Thus $(\widetilde X,\widetilde \Delta)$ is log smooth, $\tilde f\colon(\widetilde X,\widetilde \Delta)\to\widetilde Y$ is acceptable
and $\tilde f^{-1}\Sigma_{\tilde f}$ has simple normal crossings. Thus $\Sigma_{\tilde f}$ has simple normal crossings.

We let $\widetilde{\mathcal T}$ be the strict transform of $\mathcal T$.
By \cite[Proposition 8.4.9, Definition 8.3.6, Theorem 8.5.1]{Kol07} the variety $\widetilde Y$ is an Ambro model.
Then $(\tilde f\colon (\widetilde X,\widetilde \Delta)\to\widetilde Y,\widetilde{\mathcal T},\Sigma_{\tilde f})$ satisfies Assumption \ref{ass1}.
We set $\theta=b\circ\delta$. Then $\theta$ is a generically finite morphism satisfying the hypothesis of Corollary \ref{cor:frel}.
If $\sigma$ is as in Corollary \ref{cor:frel}, then $\sigma^*\mathcal R'\subseteq\mathcal R_{\theta^*\mathcal L}$ and $\sigma^{-1}\mathcal Z$ is $\sigma^*\mathcal R'$-invariant.
We have $\bar\phi^{-1}\sigma^{-1}\mathcal Z=\theta^{-1}\phi^{-1}\mathcal Z=b^{-1}\delta^{-1}\phi^{-1}\mathcal Z$.

\noindent By our construction $\delta^{-1}\phi^{-1}\mathcal Z$ is a union of log canonical centres of the log smooth pair $(Y_1,\delta^{-1}\Sigma_f)$.
Since $\Sigma_{\tilde f}=b^{-1}\delta^{-1}\Sigma_f$ has simple normal crossings, the set $b^{-1}\delta^{-1}\phi^{-1}\mathcal Z$ is a union of log canonical centres of $(\widetilde Y, \Sigma_{\tilde f})$.

\medskip

\emph{Step 2.}
Let $P,Q$ be irreducible components of $\phi^{-1}\mathcal Z$ such that either there exists $T$ with
$P,Q\subseteq T$ and $\phi_T(P)=\phi_T(Q)$
or $P\subseteq T$, $Q\subseteq T'$ and $\phi_T(P)=\phi_T(Q\cap T)$.
Let $H_{\alpha}$ be ample divisors such that $\Sigma_f+\sum H_{\alpha}$ has simple normal crossings
and the restriction of $\phi_T$ to $P\cap Q\cap \bigcap H_{\alpha}$ is generically finite and surjective. We set $\overline P=P\cap Q\cap \bigcap H_{\alpha}$.
By replacing $\Delta$ with $\Delta+\sum_{\alpha} f^*H_{\alpha}$
and $\Sigma_f$ with $\Sigma_f+\sum_{\alpha} H_{\alpha}$ we can assume that $\overline P$ is a log canonical centre of $(Y, \Sigma_f)$.

We set 
$$
\begin{array}{ccl}
\mathcal P&=&\{P\subseteq \phi^{-1}\mathcal Z \;irreducible\;component\}\\
\overline{\mathcal P}&=&\{\overline{P}\subseteq \phi^{-1}\mathcal Z\; log\; canonical\; centre\; of\; (Y, \Sigma_f)\; such\; that\; \phi_T\vert_{\overline{P}}\; is\; finite\}.
\end{array}
$$

Let $\mathcal P\to Nklt(X,\Delta+\sum_{\Gamma\subseteq \Sigma_f}\gamma_{\Gamma}f^*\Gamma )$ be a function such that $P\mapsto S_X(P)$ and $S_X(P)$ is minimal over $P$.
For every pair $(P,\overline P)$ such that $\overline P\subseteq P$ let $R_X(P,\overline P)$ be a log-canonical centre of $(X,\Delta+\sum_{\Gamma\subseteq \Sigma_f}\gamma_{\Gamma}f^*\Gamma )$
minimal over $\overline P$ and such that $R_X(P,\overline P)\subseteq S_X(P)$.

Then $\mathcal P$ and $\overline{\mathcal P}$ satisfy the hypotheses of 
Lemma \ref{lem:descends}, and there is a diagram
$$
\xymatrix{
X_0\ar[r]^{\eta}\ar[d]_{f_0}& X\ar[d]^{f}\\
Y_0\ar[r]_{\varepsilon}& Y
}
$$
with $\varepsilon$ birational and such that the exceptional locus does not contain any of the $P\in\mathcal P$, $\overline P\in\overline{\mathcal P}$ or $T\cap T'$ and
for every $P\in\mathcal P$ we have $M_{h}=\tau^*M_{f_0}\vert_{P_0}$;
 for every $\overline P\in\overline{\mathcal P}$ we have $M_{g}=\sigma^*M_{f_0}\vert_{\overline P_0}$ (notation as in Lemma \ref{lem:descends}).
We define $\Delta_0$ by $K_{X_0}+\Delta_0=\eta^*(K_X+\Delta)$.

We apply Theorem \ref{thm:alt} to $X_0,Y_0$, with $Z=\Supp \Delta_0\cup f_0^{-1}\varepsilon^{-1}\Sigma_f$.
We get $a,b\colon(\widetilde X, \widetilde Y)\to (X_0,Y_0)$
\'etale outside $Exc(\varepsilon)$.
Let $\Sigma_{\tilde f}=b^{-1}\varepsilon^{-1}\Sigma_f$. Then $\tilde f^{-1}\Sigma_{\tilde f}\cup  a^{-1}\Supp\Delta_0$ 
has simple normal crossings support.
Define $\widetilde \Delta$ by $K_{\widetilde X}+\widetilde \Delta=a^*(K_{X_0}+\Delta_0)$.
Thus $(\widetilde X,\widetilde \Delta)$ is log smooth, $\tilde f\colon(\widetilde X,\widetilde \Delta)\to\widetilde Y$ is acceptable
and $\tilde f^{-1}\Sigma_{\tilde f}$ has simple normal crossings. This implies that $\Sigma_{\tilde f}$ has simple normal crossings. 
By \cite[Proposition 8.4.9, Definition 8.3.6, Theorem 8.5.1]{Kol07} the variety $\widetilde Y$ is an Ambro model.

We let $\widetilde{\mathcal T}$ be the strict transform of $\mathcal T$.
Then $(\tilde f\colon (\widetilde X,\widetilde \Delta)\to\widetilde Y,\widetilde{\mathcal T},\Sigma_{\tilde f})$ satisfies Assumption \ref{ass1}.
We set $\theta=b\circ\varepsilon$. Then $\theta$ is a generically finite morphism satisfying the hypothesis of Corollary \ref{cor:frel}.
If $\sigma$ is as in Corollary \ref{cor:frel}, then $\sigma^*\mathcal R'\subseteq\mathcal R_{\theta^*\mathcal L}$ and $\sigma^{-1}\mathcal Z$ is $\sigma^*\mathcal R'$-invariant.
We have $\bar\phi^{-1}\sigma^{-1}\mathcal Z=\theta^{-1}\phi^{-1}\mathcal Z=b^{-1}\delta^{-1}\phi^{-1}\mathcal Z$.

As $\varepsilon$ is an isomorphism on the general point of every component of $\phi^{-1}\mathcal Z$, the preimage $\varepsilon^{-1}\phi^{-1}\mathcal Z$ is a union of log canonical centres of $(Y_0,\varepsilon^{-1}\Sigma_f)$.
Moreover  $\varepsilon^{-1}\Sigma_f$ has simple normal crossings by Lemma \ref{lem:descends}
Since $\Sigma_{\tilde f}=b^{-1}\varepsilon^{-1}\Sigma_f$ has simple normal crossings, the set $b^{-1}\delta^{-1}\phi^{-1}\mathcal Z$ is a union of log canonical centres of $(\widetilde Y, \Sigma_{\tilde f})$.

We prove now that for every $P\in\mathcal P$, if $P_1$ is the strict transform of $P$ in $\widetilde Y$ and $S_1$  is the strict transform of $S_X(P)$ in $X_1$ and 
 $\tilde f\vert_{S_1}\colon S_1\overset{ h_1}{\longrightarrow} P'_1\overset{\tau_1}{\longrightarrow}P_1$ is the Stein factorisation,
then 
 $M_{h_1}=\tau_1^*M_{f_1}\vert_{P_1}$ and $P_1$ is an Ambro model.

(The same proof will imply that 
for every  $\overline P\in\overline{\mathcal P}$ if  $\overline P_1$ is the strict transform in $\widetilde Y$ and $R_1$ is the strict transform of  $R_X(P,\overline P)$ in $X_1$ and 
  $f_0\vert_{R_0}\colon R_0\overset{ g}{\longrightarrow} \overline P'_0\overset{\sigma}{\longrightarrow}\overline P_0$ is the Stein factorisation, then  $M_{g}=\sigma^*M_{f_0}\vert_{\overline P_0}$.)
  
We have a diagram
$$
\xymatrix{
S_1\ar[r]^{a}\ar[d]_{h_1}& S_0\ar[d]^{h_0}\\
P'_1\ar[d]_{\tau_1}& P'_0\ar[d]^{\tau_0}\\
P_1\ar[r]_{\varepsilon}& P_0
}
$$
  Every curve contracted by $h_1$ is contracted by $h_0\circ a$. Therefore by the Rigidity lemma there
  is a generically finite morphism $P'_1\to P'_0$.
  By Lemma \ref{lem:gfiniteAm}, $P'_1$ is an Ambro model.
  Then
  $$\tau_1^*M_{\tilde f}\vert_{P_1}=\tau_1^*b^*M_{f_0}\vert_{P_0}=\nu^*\sigma^*\tau_0^*M_{f_0}\vert_{P_0}=\nu^*\sigma^*M_{h_0}=M_{h_1}.$$

By replacing $(X,\Delta)$ with $(\widetilde X, \widetilde\Delta+\sum_{\Gamma\subseteq \Sigma}\gamma_{\Gamma}\tilde f^*\Gamma)$, $\Sigma_f$ with $\Sigma_{\tilde f}$
we can make the following

\begin{ass}\label{assumptions}
\begin{enumerate}
\item Every irreducible component of $\phi^{-1}\mathcal Z$ is a log canonical centre of $(Y, \Sigma_f)$
\item for every $P\in \mathcal P$  we have $M_{h_X}=\tau_X^*(M_f\vert_{P})$ ,
\item for every $\overline P\in \overline{\mathcal P}$  we have $M_{g_X}=\sigma_X^*(M_f\vert_{\overline P})$
\end{enumerate}
\end{ass}

\medskip

\emph{Step 3.}
We run now an MMP with scaling as in \cite[Proposition 4.2]{FL19}.
By Proposition \ref{pro:mmp2fg}, there is $\rho\colon (X,\Delta_X)\dasharrow (W,\Delta_W)$
such that $
\psi^*\Sigma_{f}= \Delta_{W,v}.$

By \cite[Lemma 2.8]{FL19} for every $P,\overline P$ there are log canonical centres $S$ and $R$ of $(W,\Delta_W)$ with birational morphisms induced by $\rho$
$$\rho\vert_{S_X(P)}\colon S_X(P)\dasharrow  S \; \;\; \;\; \;\rho\vert_{R_X(P,\overline P)}\colon R_X(P,\overline P)\dasharrow  R.$$

Let $P$ be a component of $\phi^{-1}\mathcal Z$, $\overline P\subseteq P\subseteq T$ as above and let $S$ be the strict transform of $S_X(P)$, $R$ of $R_X(P,\overline P)$
and $\Delta_S$, $\Delta_R$ defined by adjunction.

Let $\phi|_S\colon S\overset{h}{\longrightarrow} P'\overset{\tau}{\longrightarrow} P$
and $\phi|_R\colon R\overset{g}{\longrightarrow} \overline P'\overset{\sigma}{\longrightarrow} \overline P$
be the Stein factorisations.
Then $M_{h}=\tau^*(M_f\vert_{P})=\tau^*(M_{\phi}\vert_{P})$ and
 $M_{g}=\sigma^*(M_f\vert_{\overline P})=\sigma^*(M_{\phi}\vert_{\overline P})$.

By Proposition \ref{pro:mmp2fg}\eqref{stepii} we have $\Delta_S-h^*B_h\geq 0$.

Then we can apply Proposition \ref{pro:isomfibres}
and there are non empty sets
$Z_0,P'_0,P'_r$, where $P'_r$ be the set of points $x$ such that $h^{-1}x$ is reduced, 
$Z_0$ and $P'_0$ are open, the complement of $P'_0$ in $P'$ has codimension at least 2
and $I(P')\supseteq P'_0\cap \phi_T^{-1}Z_0\cap P'_r$
 with the following property: for every
 $x_1,x_2\in I(P')$ such that $\phi_T(x_1)=\phi_T(x_2)$, if $(F_i,\Delta_i)$ is the fibre over $x_i$ with $\Delta_i=\Delta^h\vert_{F_i}$,
then $(F_1,\Delta_1)\cong (F_2,\Delta_2)$.

We claim that $\tau^{-1}\overline P$ meets the set $I(P')$
and that $R$ is a connected component of $h^{-1}\tau^{-1}\overline P$.
We prove the claim in Step 4.
Assuming the claim, we finish the proof.

We denote by $\Lambda_{\overline P}\subsetneq\overline P$ the locus where $g\colon R\to\overline P$ has non-reduced fibres.
Every fibre over $\overline P\setminus\Lambda_{\overline P}$ is isomorphic, with the boundary, to a fibre over $P\setminus \overline P$.

Let $\overline P\subseteq P_1\cap P_2$, let $R_i$ be the strict transform of $R_X(P_i,\overline P)$.
By \cite[4.45(1) and 4.45.8]{Kol13} there is a crepant birational map $(R_1,\Delta_{R_1})\dasharrow ( R_2,\Delta_{R_2})$ over $\overline P$.
Let $g_i\colon (R_i,\Delta_{R_i})\to \overline P'$ for $i=1,2$ be the induced klt-trivial fibrations.
For $x\in\overline P'$ general the fibre of $g_1$ over $x$ is crepant birational to the fibre of $g_2$ over $x$.

Consider the set $$\Lambda=\bigcup_{\overline P\subseteq P\subseteq T}\mathcal R'\phi_T(\Lambda_{\overline P}).$$
The set $\Lambda$ is a countable union of proper closed subsets of $\mathcal Z$.
Since it is closed under $\mathcal R'$, the infinite equivalence classes $[x]$ of $\mathcal R'$
such that $[x]\subseteq \mathcal Z\setminus\Lambda$ form a dense subset of $ \mathcal Z\setminus\Lambda$.

Fix $\overline P$ and let $R$ be a minimal log canonical centre of $(W,\Delta_W)$ over $\overline P$, with $g\colon(R,\Delta_R)\to \overline P'$ the klt-trivial fibration.
By the discussion above, if $[x]\subseteq \mathcal Z\setminus\Lambda$ then for every $x_1,x_2\in\phi^{-1}[x]\cap\overline P$ the fibres over $x_1$ and $x_2$ are crepant birational to each other, with their boundaries.

Since the classes $[x]\subseteq \mathcal Z\setminus\Lambda$ form a dense subset of $ \mathcal Z\setminus\Lambda$, the union of the intersections $\phi^{-1}[x]\cap\overline P$ is a dense subset of $\overline P$.
By construction, if $[x]\subseteq \mathcal Z$ then $\phi^{-1}[x]\cap \overline P$ is an infinite set.


On the other hand, we have by construction $M_g=\sigma^*M_f\vert_{\overline P}=\sigma^*\phi_T^*A$ 
where $A$ is an ample divisor on $V$. As $\phi_T\vert_{\overline P}$ is generically finite,
$M_g$ is big.


By  Proposition \ref{pro:fujinovar} the variation of $g$ is maximal.

If $\dim R-\dim \overline P=\dim W-\dim Y$, then the crepant birational fibres are in fact isomorphic and by Proposition \ref{pro:finvar} there is a finite number of fibres isomorphic to a fixed general one.

If $\dim R-\dim \overline P=\dim W-\dim Y$, then by Conjecture \ref{con:maxvar} there is a finite number of fibres crepant birational to a fixed general one.

\medskip

\emph{Step 4.} We prove that  $\tau^{-1}\overline P$ meets the set $I(P')$
and that $R$ is a connected component of $h^{-1}\tau^{-1}\overline P$.
Let $P=T_1\cap\ldots\cap T_k$ with $T_i\subseteq\Sigma_f$ and $\Xi_P=(\Sigma_f-T_1-\ldots-T_k)\vert_P$.

First, we prove the following statement:
\begin{cla}\label{cla:ind}
let $Q$ be a component of $\Xi_P$ such that $\phi_T(Q)=\phi_T(P)$. Then every irreducible component of $\tau^{-1}Q$ meets $P'_0\cap \tau^{-1}\phi^{-1}Z_0\cap P'_r$ and every connected component of $h^{-1}\tau^{-1}Q$ is irreducible
and a minimal log canonical centre over $Q$.
\end{cla}

Since $M_h=\tau^*M_{\psi}\vert_P$, by Proposition \ref{pro:mmp2fg}\ref{stepii} we have $B_h=(\tau^*\Xi_P)_{red}$.
Thus $\tau^{-1}Q\subseteq\Supp B_h$.
Let $Q'\subseteq \tau^{-1}Q$. Since the complement of $P'_0$ in $P'$ has codimension 2, $Q'$
meets $P'_0$. Since $\phi_T(Q)=\phi_T(P)$, $Q'$ meets $\tau^{-1}\phi^{-1}Z_0$.

Finally, every irreducible component of $h^{-1}Q'$ is a log canonical centre of $(W,\Delta_W)$,
therefore $\rho$ is an isomorphism at its generic point and the restriction of $\psi$ to it has generically reduced fibre. 
We proved that $Q'$ meets $I(P')$.

Let $K$ be a connected component of $h^{-1}Q'$. Then the general fibre of $\psi\vert_K$ is 
 isomorphic to a fibre of $h$. Thus $K$ is irreducible and $(K,\Delta_K)$ is generically klt over $Q$, ending the proof of Claim \ref{cla:ind}.

\medskip

We prove now the statement on $\overline P$ by induction on the codimension of $\overline P$ in $P$.
If the codimension is 1, it follows from Claim \ref{cla:ind}.
If the codimension is at least 2, there is a component $Q$ of $\Xi_P$ such that $\overline P\subseteq Q$.
By \ref{cla:ind}, every connected component $K$ of $h^{-1}\tau^{-1}Q$ is irreducible and a minimal log canonical centre of $(W,\Delta_W)$ over $Q$.
Let $\phi|_K\colon S\overset{\ell}{\longrightarrow} Q'\longrightarrow Q$ be the Stein factorisation,
let $\vartheta\colon Q'\to P'$ be the induced finite map.
By Proposition \ref{pro:mmp2fg}\ref{stepii}, $\Delta_K-\ell^*B_{\ell}\geq0$.
Then we can apply Proposition \ref{pro:isomfibres}, and there is a set $I(Q')$.
We notice that  $I(Q')=\theta^{-1}I(P')$.
By the inductive hypothesis $\theta^{-1}\tau^{-1}\overline P$ meets $I(Q')$.
Thus $\tau^{-1}\overline P$ meets $I(P')$.

\end{proof}

%% file: thm2.tex
 \section{Triviality of the moduli part on pseudo-fibres}\label{thm2}
This section is entirely devoted to the proof of our second main technical result: if the moduli part is numerically zero along a simple normal crossings reducible connected variety, then it is torsion along it.

\begin{thm}\label{thm:step2}
Let $f\colon(X,\Delta)\to Y$ be an acceptable klt-trivial fibration, where $(X,\Delta)$ is a log smooth log canonical pair and $Y$ is a smooth Ambro model for $f$.
Let $\mathcal T$ be a connected divisor such that there is a simple normal crossings $(f,\mathcal T)$-bad divisor $\Sigma_{f}$ and such that the restriction of $M_f$ to $T$ is semiample for every $T\subseteq \mathcal T$. Let $m$ be a positive integer such that $mM_Y$ is a Cartier divisor.

 Set $\mathcal L=\mathcal O_Y(mM_f)$. Assume that $\mathcal R_{\mathcal L}$ is a finite equivalence relation.
Then for a general equivalence class $[x]$ of $R_{\mathcal L}$ the restriction
of $\mathcal L$ to $\mathcal T_{[x]}$ is torsion.
\end{thm}
\begin{proof}
\emph{Step 1.}
Since $\mathcal R_{\mathcal L}$ is a finite equivalence relation, the set $\mathcal T_{[x]}$ is a finite union of irreducible subvarieties of $Y$.

As $[x]$ is general, the subvariety $\mathcal T_{[x]}$ has simple normal crossings in the sense of \ref{def:snc}.

Let $\varepsilon\colon Y'\to Y$ be a birational morphism such that $\varepsilon^{-1}\mathcal T_{[x]}$ is divisorial and $\varepsilon^*\Sigma_f$ has simple normal crossings support.
By Lemma \ref{lem:tors3}, the restriction $\mathcal L\vert_{\mathcal T_{[x]}}$ is torsion if and only if $\varepsilon^*\mathcal L\vert_{\varepsilon^{-1}\mathcal T_{[x]}}$ is torsion.
Let $X'$ be a normalisation of the main component of $X\times_Y Y'$ with $\varepsilon_X\colon X'\to X$
and $f'\colon X'\to Y'$
 the induced morphisms.
By \cite[Proposition 5.5]{Amb04} we have $M_{f'}=\varepsilon^*M_f$.
Then for every $T\subseteq\varepsilon^{-1}\mathcal T_{[x]}$ we have $M_{f'}\vert_T\equiv 0$.

Let $(a,b)\colon (\widetilde X,\widetilde Y)\to (X',Y')$ be a semistable reduction such that $b^{-1}\varepsilon^{-1}\Sigma_f$ and  $a^{-1}\varepsilon_X^{-1}(\Delta+f^*\Sigma_f)$ have simple normal crossings supports.
By Lemmas \ref{lem:tors2} and \ref{lem:tors3}, the pullback
$\varepsilon^*\mathcal L\vert_{\varepsilon^{-1}\mathcal T_{[x]}}$ is torsion
if and only if $b^*\varepsilon^*\mathcal L\vert_{b^{-1}\varepsilon^{-1}\mathcal T_{[x]}}$ is.
After replacing $X,Y,f$ with $\widetilde X,\widetilde Y,\tilde f$ we can assume that $f$ is semistable.
After replacing $\mathcal T$ with $b^{-1}\varepsilon^{-1}\mathcal T_{[x]}$ we have to prove that the restriction of $M_f$ to the divisor $\mathcal T$ is torsion.
By Proposition \ref{pro:mmp2fg}\ref{stepiii} for every irreducible component $T\subseteq \mathcal T$
we have $\mathcal L\vert_T\sim_{\mathbb Q}0$. After replacing $m$ by a multiple, we can assume that for every irreducible component $T\subseteq \mathcal T$
we have $\mathcal L\vert_T\sim0$.

\emph{Step 2.} We fix a circuit $\mathcal C=(\{T_1,\ldots,T_k\},\{T_{i,i+1}\})$ in $\Gamma^{i}(\mathcal T)$. By Lemma \ref{lem:trivcirc} it is enough to prove that $\Phi_{\mathcal L,\mathcal C}$ has finite order.
We set $$\Delta_X=\Delta+\sum_{\Gamma\subseteq \Sigma_f}\gamma_{\Gamma}f^*\Gamma$$ and run an MMP as in Proposition \ref{pro:mmp2fg}. We get a crepant birational map $\rho\colon (X,\Delta_X)\dasharrow (W,\Delta_W)$ over $Y$ and a klt-trivial fibration $\psi\colon (W,\Delta_W)\to Y$.
For every $i$ let $S_i$ be a log canonical centre of $(W,\Delta_W)$ minimal over $T_i$. 
We let $S_i^0$ and $S_i^1$ be log canonical centres of $(W,\Delta_W)$ minimal over $T_{i,i+1}$ and with $S_i^{\ell}\subseteq S_{i+\ell}$. 
Let $\Delta_{S_i^{\ell}}$ be the boundary defined by
$(K_W+\Delta_W)\vert_{S_i^{\ell}}=K_{S_i^{\ell}}+\Delta_{S_i^{\ell}}$.
The varieties sit in the following diagram
$$
\xymatrix{
S_i\ar[d]&S_i^0\ar[l]\ar[rd]&&S_i^1\ar[r]\ar[ld]&S_{i+1}\ar[d]\\
T_i&&T_{i,i+1}\ar[ll]\ar[rr]&&T_{i+1}
}
$$


The fibration $\psi\colon (W,\Delta_W)\to Y$ is a crepant, dlt, log structure in the sense of \cite[Section 4.4]{Kol13}. 
By \cite[4.45(1) and 4.45.8]{Kol13} there is a crepant birational map $$\lambda_i\colon(S_i^0,\Delta_{S_i^0})\dasharrow (S_i^1,\Delta_{S_i^1}).$$
By \cite[Lemma 2.8]{FL19} there are centres $ S_{X, i}^{\ell}$ of $(X,\Delta_X)$ such that the restriction of $\rho$ induces a birational map 
$\rho\colon S_{X, i}^{\ell}\dasharrow  S_i^{\ell}$. We let $\psi\vert_{S_i^{\ell}}\colon S_i^{\ell} \overset{g_i}{\longrightarrow} Q_i\overset{\sigma_i}{\longrightarrow} T_{i,i+1} $ and $f\vert_{S_{X, i}^{\ell}}\colon S_{X, i}^{\ell} \overset{ g_{X, i}}{\longrightarrow} Q_i\overset{\sigma_i}{\longrightarrow} T_{i,i+1} $ be the Stein factorisation.
Let $V\subseteq T_{i,i+1}$ be a non-empty open set such that over $\sigma_i^{-1}V$ the map $\rho\vert_{S_{X, i}^{\ell}}$ is defined at every generic point of every fibre over $q\in V$ and does not extract any component of the fibres of $f\vert_{S_{X, i}^{\ell}}$ for $\ell=0,1$. In particular, 
\begin{equation}\label{Kappa}
\begin{split}
&\text{
the fibres of $g_i$ over points of $\sigma_i^{-1}V$ are reduced because they are}\\
&\text{push forward of fibres of $g_{X, i}$, and those are reduced by}\\
&\text{Remark \ref{rem:restred}. Set $K_i=T_{i,i+1}\setminus V$.}
\end{split}
\end{equation}
%

\emph{Step 3.} Let $\Delta_{S_i}$ be defined by $(K_W+\Delta_W)\vert_{S_i}=K_{S_i}+\Delta_{S_i}$. Let $\psi\vert_{S_i}\colon S_i\overset{h_i}{\longrightarrow} T'_i\overset{\tau_i}{\longrightarrow} T_i $ be the Stein factorisation. By \cite[Proposition 4.2]{FL19}, $h_i$ is a klt-trivial fibration. By Proposition \ref{pro:mmp2fg} \ref{stepiii} we have $\tau^*(M_f)\vert_T\sim_{\mathbb Q} M_h$.
By Proposition \ref{pro:mmp2fg} \ref{stepiii} we have $\Delta_{S_i}-h_i^*B_{h_i}\geq 0$.
Moreover $h_i$ has reduced fibres over the generic points of every component of $B_{h_i}$ by Proposition \ref{pro:mmp2fg} \ref{stepiv}.

By Theorem \ref{ambro1} there is a diagram
$$
\xymatrix{
S_i\ar[d]_{h_i}&&F_i\ar[d]\\
T'_i&\widetilde T_i\ar[l]^{\vartheta_i}\ar[r]_{\rho_i}&\{x_i\}
}
$$
where $\vartheta_i$ is a finite map. Let $\widetilde S_i$ be the normalisation of the main component of $S_i\times_{T'_i}\widetilde T_i$ with the natural map
$\tilde h_i\colon\widetilde S_i\to \widetilde T_i$.
By Theorem \ref{ambro1} there is a birational map $\eta\colon (\widetilde S_i,\Delta_{\widetilde S_i})\dasharrow (F_i,\Delta_i)\times \widetilde T_i$.
After possibly composing $\vartheta_i$ with a finite map (or by the proof of Theorem \ref{ambro1}, \cite[Theorem 3.3]{Amb05a}), we can 
assume that $\tilde h_i$ is weakly semistable in codimension 1.
By Lemma \ref{lem:reducedfibres} we have $(\Delta_{\widetilde S_i}-\tilde h_i^*B_{\tilde h_i})\vert_{\tilde h_i^{-1}U}\geq 0$ with $U$ an open set of $\widetilde T_i$ meeting $\vartheta_i^{-1}\tau_i^{-1}T_{i,i+1}$ and $\vartheta_i^{-1}\tau_i^{-1}T_{i-1,i}$ non trivially.
We set $J'_i=\widetilde T_i\setminus U$ and $J_i=J'_i\cup \widetilde T_i^{sing}$.
By Proposition \ref{ambro2}, the birational map $\eta$ can be extended to an isomorphism
$\eta\colon (\widetilde S_i,\Delta_{\widetilde S_i}-\tilde h_i^*B_{\tilde h_i})\to (F_i,\Delta_i)\times \widetilde T_i$
 over 
$\widetilde T_i\setminus J_i$.
It follows that
$$\OO(mM_{\widetilde h_i})\vert_{\widetilde T_i\setminus J_i}\sim \widetilde h_{i*}\OO(\pi_i^*(m(K_{F_i}+\Delta_{ F_i})))\vert_{\widetilde T_i\setminus J_i} $$
where $\pi_i\colon F_i\times \widetilde T_i\to F_i$ is the first projection.

We fix
$q_i\in T_{i,i+1}$ with $q_i\not\in K_i$, $\vartheta_i^{-1}\tau_i^{-1}q_i\not\subseteq J_i$,
$\vartheta_{i+1}^{-1}\tau_{i+1}^{-1}q_i\not\subseteq J_{i+1}$.
We also let $p_i^0\in\tau_i^{-1}(q_i)$ be a point  
such that $p_i^0 \not\in \vartheta_i (J_i)$ and $p_i^1\in\tau_{i+1}^{-1}(q_{i+1})$ be a point  
such that $p_i^1 \not\in \vartheta_{i+1} (J_{i+1})$.

By \eqref{Kappa}, by our choice of $p_i^{\ell}$ the fibre $G_i^{\ell}$ of $h_{i+\ell}$ over $p_i^{\ell}$ is reduced.
By Lemma \ref{redgen} we have $(G_i^{\ell},(\Delta_{\widetilde S_i}-\tilde h_i^*B_{\tilde h_i})\vert_{G_i^{\ell}})\cong (F_{i+\ell},\Delta_{i+\ell})$.
Thus we have a canonical isomorphism $$\OO_{T_i}(mM_f)_{q_i}\cong H^0(F_i,m(K_{ F_i}+\Delta_{F_i})).$$

\emph{Step 4.} By our choice of $p_i^{\ell}$ the crepant birational map $\lambda_i\colon(S_i^0,\Delta_{S_i^0})\dasharrow(S_i^1,\Delta_{S_i^1})$
restricts to a crepant birational map $\lambda_i\colon(G_i^0,\Delta_{G_i^0})\dasharrow(G_i^1,\Delta_{G_i^1})$.

The map $\lambda_i$ composed with the isomorphisms with $F_i$ and $F_{i+1}$
gives a crepant birational map $\chi_{i,i+1}\colon (F_i,\Delta_i)\dasharrow (F_{i+1},\Delta_{i+1})$
such that there is a diagram
$$
\xymatrix{
\OO_{ W_{q_i}}(m(K_W+\Delta_W))\ar@{=}[r]\ar[d]_{R_{i+1}}& \OO_{ W_{q_i}}(m(K_{W}+\Delta_W))\ar[d]^{R_i}\\
\OO_{ F_{i+1}}(m(K_{ F_{i+1}}+\Delta_{F_{i+1}}))\ar[r]_{\chi_{i,i+1}^*}&
\OO_{F_i}(m(K_{ F_i}+\Delta_{ F_i}))
}
$$
where  $W_{ q_i}$ is the fibre of $\psi$ over $q_i$ and
 $R_i$ and $R_{i+1}$ are the Poincar\'e residue maps existing by \cite[4.45(4)]{Kol13} restricted to $W_{ q_i}$. 

\medskip

Then 
$$\Phi_{\mathcal L,\mathcal C}=\chi_{1,2}^*\circ\ldots\circ \chi_{k,1}^*.$$
Thus $\Phi_{\mathcal L,\mathcal C}$ is in the image of the crepant birational representation
$$\Bir^\mathrm{c}(F_1,\Delta_{F_1})\rightarrow \GL\big(H^0( F_1,m(K_{ F_1}+\Delta_{F_1}))\big)$$
which is finite by Theorem \ref{thm:cbir}.
\end{proof}